\documentclass[11pt]{amsart}

\usepackage{amssymb, amsmath, amscd, amsthm, graphicx, psfrag,rotating}
\usepackage[matrix,arrow,curve]{xy}
\xyoption{dvips}              

\newcommand{\R}{{\mathbb R}}

\newcommand{\Q}{{\mathbb Q}}

\def\tr{{\rm tr}\,}
\newcommand{\calO}{{\mathcal{O}}}

\newcommand{\calE}{{\mathcal{E}}}
\newcommand{\calL}{{\mathcal{L}}}

\newcommand\rth{\refstepcounter{equation}}
\newcommand\numb{\rth{\rm \theequation}}
\numberwithin{equation}{section}

\DeclareMathOperator{\Tot}{Tot}

\newcommand{\BC}{{\mathbb C}}
\newcommand{\BN}{{\mathbb N}}
\newcommand{\BK}{{\mathbb K}}
\newcommand{\BR}{{\mathbb R}}
\newcommand{\BZ}{{\mathbb Z}}
\newcommand{\BQ}{{\mathbb Q}}
\newcommand{\ASS}{{\mathcal{ASSOC}}}

\newcommand{\COMM}{{\mathcal{COMM}}}

\newcommand{\Poiss}{{\mathcal{P}oiss}}

\newcommand{\rank}{\mathrm{rank}}

\newcommand{\calK}{{\mathcal{K}}}

\newcommand{\hemb}{\overline{Emb}_d}
\newcommand{\rhhemb}{H_*(\hemb,\Q)}

\theoremstyle{plain}
\newtheorem{thm}{Theorem}[section]
\newtheorem{prop}[thm]{Proposition}
\newtheorem{proposition}[thm]{Proposition}
\newtheorem{lemma}[thm]{Lemma}
\newtheorem{cor}[thm]{Corollary}
\newtheorem{conj}[thm]{Conjecture}
\newtheorem{cor-def}[thm]{Corollary-Definition}

\newtheorem*{theorem*}{Theorem}

\theoremstyle{definition}
\newtheorem{definition}[thm]{Definition}
\newtheorem{remark}[thm]{Remark}
\newtheorem{notation}[thm]{Notation}

\theoremstyle{remark}

\newtheorem*{notation*}{Notation}

\begin{document}


\title{Hodge decomposition in the homology of long knots}

\author{Victor Turchin}
\address{Kansas State University, USA.}
\email{turchin@ksu.edu} \urladdr{http://www.math.ksu.edu/\~{}turchin/}

%


\begin{abstract}
The paper describes a natural splitting in the rational homology and homotopy of the spaces of long knots. This decomposition presumably arises from the cabling maps in the same way as a natural decomposition in the homology of loop spaces arises from power maps. The generating function for the Euler characteristics of the terms of this splitting is presented. Based on this generating function we show that both the homology and homotopy ranks of the spaces in question grow at least exponentially. Using natural graph-complexes we show that this splitting on the level of the bialgebra of chord diagrams is exactly the splitting defined earlier by Dr.~Bar-Natan. In the Appendix we present tables of computer calculations of the Euler characteristics. These computations give a certain optimism that the Vassiliev invariants of order $> 20$ can  distinguish knots from their inverses.
\end{abstract}

\maketitle

\sloppy

\tableofcontents

\section*{Plan of the paper. Main results}
The paper is divided into 3 parts. The first part is introductive. We describe there some general facts and constructions about the knot spaces we study. We also define there a convenient terminology for the sequel. The only new thing in this part is the definition of the Hodge decomposition in the rational homology and homotopy of knot spaces. In the second part we introduce natural graph-complexes computing the rational homology and homotopy of knot spaces together with their Hodge splitting. This construction generalizes an earlier construction of Bar-Natan that describes the primitives of the bialgebra of chord diagrams. To recall the latter bialgebra is intimately related to the Vassiliev knot invariants, which are also called invariants of finite type. The main result of this part is Theorem~\ref{t82}. We also formulate there Conjecture~\ref{c121} that gives one more motivation for the study of the Hodge decomposition. It says that the Hodge decomposition of one-dimensional knots encodes information about the homology and homotopy of higher dimensional knots. Part~3 is more computational. Its main result is Theorem~\ref{t131} that describes the generating function of the Euler characteristics of the terms of the Hodge splitting. Based on the formula for the generating function we show that the ranks of the rational homology and homotopy of the spaces of long knots grow at least exponentially, see Theorems~\ref{th:exp_compl_homol}, \ref{th:exp_compl_homot}, \ref{th:cumulative}.

\part{Introduction}\label{p1}

\section{Spaces of long knots modulo immersions}\label{s1}

Denote by $Emb_d$, $d\geq 3$, the space of long knots, i.e. the space of smooth non-singular embeddings $\BR^1\hookrightarrow\BR^d$ coinciding with a fixed linear embedding $t\mapsto (t,0,0,\ldots,0)$ outside  a compact subset of $\BR^1$. Similarly $Imm_d$ is the space of long immersions. The homotopy fiber of the inclusion $Emb_d\hookrightarrow Imm_d$ will be denoted by $\hemb$. This {\it space $\hemb$ of long knots modulo immersion} and its rational homology and homotopy will be the object of our study. Its homology has a more natural interpretation then that of $Emb_d$.\footnote{The homology $\rhhemb$, $d\geq 4$, is the Hochschild homology of the Poisson algebras operad~\cite{LTV,T-HLN,T-OS}.} On the other hand $\hemb$ is homotopy equivalent to the product~\cite[Proposition~5.17]{Sinha-OKS}:
$$
\hemb\simeq Emb_d\times\Omega^2S^{d-1}.
$$
So, there is no much difference the homology and homotopy of $\hemb$ or of $Emb_d$ is studied.

\begin{proposition}\label{p11}
The spaces $\hemb$, $d\geq 3$, are acted on by the operad $C_2$ of little squares.
\end{proposition}

\begin{proof}[Sketch of the proof]
R.~Budney constructed a model for the spaces $Emb_d^{framed}$, $d\geq 3$, of framed long knots that have a natural $C_2$-action~\cite{Budney}. His construction can be easily generalized to the space $Imm_d^{framed}$ of framed long immersions. The inclusion $Emb_d^{framed}\hookrightarrow Imm_d^{framed}$ turns out to be a map of $C_2$-algebras. But the homotopy fiber of this inclusion is homotopy equivalent to $\hemb$, which implies the result.
\end{proof}

In case $d\geq 4$, there is a different construction of such an action due to D.~Sinha~\cite{Sinha-OKS}. It is an open question whether these two $C_2$-actions are equivalent.

\begin{cor}\label{c12}
For any field of coefficients $\BK$ the homology $H_*(\hemb,\BK)$ is a graded bicommutative bialgebra.
\end{cor}

\begin{cor}\label{c13}
Since the spaces $\hemb$, $d\geq 4$, are connected,
then for a field $\BK$ of characteristics zero the homology  $H_*(\hemb,\BK)$ is a graded polynomial bialgebra generated by $\pi_*(\hemb)\otimes \BK$.
\end{cor}

\section{Cosimplicial model for $\hemb$}\label{s2}

For any $d\geq 4$, D.~Sinha defined a cosimplicial model ${C_d}^\bullet =\{ {C_d}^n\, |\, n\geq 0\}$, whose homotopy totalization is weakly homotopy
equivalent to $\hemb$~\cite{Sinha-TKS}. The $n$-th stage ${C_d}^n$ of this model is an appropriate compactification of the configuration space of $n$
distinct points labeled by $1,2,\ldots,n$ in $I\times\BR^{d-1}=[0,1]\times\BR^{d-1}$. The space ${C_d}^n$ can also be viewed as the space of (bipointed) embeddings $\{0,1,\ldots,n,n+1\}\hookrightarrow I\times\BR^{d-1}$ sending 0 to $(0,\bar 0)$ and $n+1$ to $(1,\bar 0)$. The codegeneracy $s_i\colon {C_d}^n\to{C_d}^{n-1}$, $i=1\ldots n$, is the forgetting of the $i$-th point in configuration, the coface $d_i\colon {C_d}^n\to {C_d}^{n+1}$, $i=0\ldots n+1$, is the doubling of the $i$-th point in the direction of the vector $(1,\bar 0)$.

The homology $\{H_*({C_d}^n,\BK),\, n\geq 0\}$, and homotopy $\{\pi_*({C_d}^n)\otimes\BK,\, n\geq 0\}$ form respectively a cosimplicial coalgebra that we will denote by ${A_d}^\bullet$:
$$
{A_d}^\bullet=\{{A_d}^n,\, n\geq 0\}=\{H_*({C_d}^n,\BK),\, n\geq 0\},
$$
and a cosimplicial Lie algebra
$$
{L_d}^\bullet=\{{C_d}^n,\, n\geq 0\}=\{\pi_*({C_d}^n)\otimes\BK,\, n\geq 0\}.
$$
We will usually assume that $\BK=\BQ$.

The cosimplicial model ${C_d}^\bullet$ for $\hemb$ defines the Bousfield-Kan homology and homotopy spectral sequences, whose first term in the homological case is
$$
E^1_{-p,*}=H_*^{Norm}({C_d}^p,\BK)=N{A_d}^p,
$$
and in the homotopy case:
$$
{\calE}^1_{-p,*}=\pi_*^{Norm}({C_d}^p)\otimes\BK=N{L_d}^p.
$$
Where the normalized part $H_*^{Norm}({C_d}^p,\BK)=N{A_d}^p$, $\pi_*^{Norm}({C_d}^n)\otimes\BK=N{L_d}^p$ is the intersection of kernels of degeneracies:
$\bigcap_{i=1}^p\ker s_{i*}$.

\begin{thm}[\cite{ALTV-coformal,LT,LTV}]\label{t21}
For a field $\BK$  of characteristics zero and for $d\geq 4$, both the homology and homotopy Bousfield-Kan spectral sequences associated with ${C_d}^\bullet$ collapse at the second term:
$$
E^2_{*,*}=E^\infty_{*,*}=H_*(\hemb,\BK), \qquad \calE^2_{*,*}=\calE^\infty_{*,*}=\pi_*(\hemb)\otimes \BK.
$$
\end{thm}

This theorem means that rationally the homology and homotopy of $\hemb$ are computed by the normalized complexes:
$$
\Tot {A_d}^\bullet =\left(\oplus_{n\geq 0} N{A_d}^n,d\right), \qquad\qquad
\Tot {C_d}^\bullet =\left(\oplus_{n\geq 0} N{C_d}^n,d\right),
$$
where the differential $d$ is as usual the alternated sum of cofaces $d=\sum_{i=0}^{n+1}(-1)^id_{i*}$.

\section{Hodge decomposition}\label{s3}
As usual we assume that the main field of coefficients $\BK$ is of characteristics zero. Gerstenhaber and Schak~\cite{GS}
 defined the so called Hodge decomposition in the Hochschild (co)homology of  commutative algebras with coefficients in a symmetric  bimodule $M$ over $A$.
$$
HH_*(A,M)=\oplus_i HH_*^{(i)}(A,M),\qquad HH^*(A,M)=\oplus_i HH_{(i)}^*(A,M).
\eqno(\numb)\label{eq31}
$$
The construction of Gerstenhaber and Schak~\cite{GS} used the $S_n$-action on the components of Hochschild complexes:
$$
(\oplus_{n\geq 0} M\otimes A^{\otimes n},\delta), \qquad (\oplus_{n\geq 0} Hom(A^{\otimes n}, M), d).
$$
They defined the families of orthogonal projectors $e_n^{(i)}\in \BK[S_n]$, $1\leq i\leq n$ if $n>0$, and $i=0$ if $n=0$, satisfying the following properties:

$$
\text{
\begin{tabular}{l}
$\bullet$ $e_n^{(i)}\cdot e_n^{(j)}=\delta_{ij}e_n^{(i)}$ (they are orthogonal projectors);\\
$\bullet$ $\sum_{i=0}^n e_n^{(i)}=1$ (the family of projectors $e_n^{(i)}$, $i=0\ldots n$, is complete);\\
$\bullet$ $\delta e_n^{(i)}=e_{n-1}^{(i)} \delta$, $de_n^{(i)}=e_{n+1}^{(i)}d$, $0=1\ldots n$.
\end{tabular}
}\eqno(\numb)\label{eq_properties}
$$
In the above formula $e_n^{(i)}=0$ if $i$ is not in the range $1\ldots n$ (in case $n=0$ one has $e_0^{(i)}=0$ if $i\neq 0$).

By the last property the complexes~\eqref{eq31} split into a direct sum of complexes with the $i$-th complex being the image of the projection $e^{(i)}=\sum_{n=0}^\infty e_n^{(i)}$. This splitting induces splitting in Hodge (co)homology which is called \lq\lq Hodge decomposition". Notice that the same construction works equally well in the differential graded setting~\cite{BFG_AdamsOp,Vigue}.

Later on J.-L.~Loday gave a  more general set up for this splitting~\cite{Loday}. He noticed that it takes place for any (co)simplicial complex that can be factored through the category $\Gamma$ (resp. $\Gamma^{op}$) of finite pointed sets. Recall that a simplicial (resp. cosimplicial) vector space is a functor from the category $\Delta^{op}$ (resp. $\Delta$) to the category of vector spaces:
$$
X_\bullet\colon\Delta^{op}\longrightarrow Vect, \qquad X^\bullet\colon\Delta\longrightarrow Vect.
$$
We will also consider cosimplicial dg-vector spaces. In that case the target category is $dg-Vect$ differential graded vector spaces. The objects of $\Delta^{op}$ are sets
$$
[n]=\{0,1,\ldots,n+1\}, \,\, n=0,1,2,\ldots
$$
The morphisms are the bipointed ordered maps $[n]\to [m]$ preserving the linear order. By \lq\lq bipointed" we mean maps sending 0 to 0, and $n+1$ to $m+1$. The morphisms are generated by the so called {\it face maps}:

$$
d_i\colon [n]\to [n-1], \quad i=0\ldots n,
$$
$$
d_i(j)=\begin{cases} j,& \text{if $j\leq i$;}\\
                     j-1,& \text{if $j>i$}
       \end{cases}
$$
(the preimage $d_i^{-1}(i)$ has two points $i$ and $i+1$), and {\it degeneracies}:
$$
s_i\colon [n]\to [n+1], \quad i=1\ldots n+1,
$$
$$
s_i(j)=\begin{cases} j,& \text{if $j<i$;}\\
j+1,& \text{if $j\geq i$}
\end{cases}
$$
(the preimage $d_i^{-1}(i)$ is empty).

The simplicial and cosimplicial complexes are defined as follows:
$$
(\oplus_{n\geq 0} X_n,\partial), \qquad (\oplus_{n\geq 0} X^n, d),
\eqno(\numb)\label{eq33}
$$
where the differential $\partial$ (respectively $d$) is the alternated sum of (co)faces $d_i$ plus (in the differential graded case) the inner differential of each $X_n$ (respectively $X^n$).

It is usually convenient to consider the normalized complexes:
$$
\Tot X_\bullet = (\oplus_{n\geq 0} NX_n,\partial), \qquad \Tot X^\bullet=(\oplus_{n\geq 0} NX^n, d),
$$
which are quasi-isomorphic to the initial ones~\eqref{eq33}. The normalized part is defined as follows
$$
NX_n=X_n\left/+_{i=1}^n \mathrm{Im}\, s_i,\right.
\qquad
NX^n=\cap_{i=1}^n \mathrm{ker}\, s_i.
$$

The category $\Gamma$ has objects
$$
\underline{n}=\{*,1,2,\ldots,n\},\,\, n=0,1,2,\ldots
\eqno(\numb)\label{eq_gamma_elements}
$$
The morphisms $Mor_\Gamma(\underline{m},\underline{n})$ are the pointed maps $\underline{m}\to\underline{n}$ (we consider each set~\eqref{eq_gamma_elements} to be pointed in~$*$). One has a natural functor $\rho\colon\Delta^{op}\longrightarrow\Gamma$ induced by the set maps:

$$
[n]\stackrel\rho\longrightarrow{\underline n},
$$
$$
\rho(i)=
\begin{cases}
i,& 1\leq i\leq n;\\
*,& \text{$i=0$ or $n+1$}.
\end{cases}
$$
By abuse of the language $\rho(d_i)$, $\rho(s_i)$ will still be denoted by $d_i$, $s_i$ and called faces and degeneracies respectively. The morphisms of $\Gamma$ are generated by faces, degeneracies, and also by the isomorphisms of each object $\underline{n}$ (that are given by the $S_n$-group action).

\begin{remark}\label{r31}
Let $\Gamma_{cycl}$ denote the subcategory of $\Gamma$, whose objects are the same and morphisms are the pointed maps $\underline{m}\to\underline{n}$ preserving the cyclic order. It is easy to see, that $\Gamma_{cycl}$ is isomorphic to $\Delta^{op}$ via the functor $\rho$:
$$
\rho\colon\Delta^{op}\stackrel\simeq\longrightarrow\Gamma_{cycl}.
$$
\end{remark}

J.-L.~Loday noticed~\cite{Loday} that if a simplicial (resp. cosimplicial) differential graded vector space $X_\bullet$ (resp. $X^\bullet$) can be factored through the category $\Gamma$ (resp. $\Gamma^{op}$):
$$
\xymatrix{
\Delta^{op}\ar[r]^-{X_\bullet}\ar[d]_\rho&dg{-}Vect\\
\Gamma\ar[ru]&
}
\qquad
\xymatrix{
\Delta\ar[r]^-{X^\bullet}\ar[d]_\rho&dg{-}Vect\\
\Gamma^{op}\ar[ru]&
}
\eqno(\numb)\label{eq_gamma_factor}
$$
then each $X_n$ (resp. $X^n$) admits an $S_n$-action, and moreover the projections $e_n^{(i)}\in\BK[S_n]$ satisfy the properties~\eqref{eq_properties}, and therefore define the Hodge splitting
$$
\Tot X_\bullet=\oplus_{i\geq 0} \Tot^{(i)} X_\bullet, \qquad
\Tot X^\bullet=\oplus_{i\geq 0} \Tot^{(i)} X^\bullet.
$$

\section{Operadic point of view. $\Sigma$-cosimplicial spaces, and commutative $\Sigma$-cosimplicial spaces}\label{s4}
Notice that the Hodge splitting for Hochschild complexes was possible only for {\it commutative} algebras with coefficients in a {\it symmetric} bimodule. Only in this case the symmetric group action is \lq\lq nicely" compatible with the differential which is built out of the product. In this section we define a natural formalism that generalizes this idea and will be helpful to detect $\Gamma$ and $\Gamma^{op}$-modules.

Let $\calO$ be any of the following three operads:
\begin{itemize}
\item non-$\Sigma$ operad  of associative algebras $\ASS=\{\ASS(n),\, n\geq 0\}$ whith $\ASS(n)=\BK$ for all $n$.

\item operad of commutative algebras $\COMM=\{\COMM(n),\, n\geq 0\}$ with $\COMM(n)=\BK$ being the trivial representation of $S_n$.

\item operad of associative algebras which will be also denoted by $\ASS=\{\ASS(n),\, n\geq 0\}$, but in this $\Sigma$-case $\ASS(n)=\BK[S_n]$. To make a difference between non-$\Sigma$ et $\Sigma$ cases we will always specify which one is considered.
\end{itemize}

\begin{definition}\label{weak_bimodule}
$M=\{M(n), n\geq 0\}$ is a {\it weak bimodule} over the operad $\calO$ if it is endowed with a series of composition maps:
$$
\overline{\circ}_i\colon \calO(n)\otimes M(k)\to M(n+k-1),\,\, i=1\ldots n, \text{ (left action)};
\eqno(\numb)\label{eq_left_action}
$$
$$
\underline{\circ}_i\colon M(k)\otimes \calO(n)\to M(k+n-1),\,\, i=1\ldots k, \text{ (right action)},
\eqno(\numb)\label{eq_right_action}
$$
satisfying natural associativity property, and (in the $\Sigma$-case) compatibility with the $S_n$-group action (in the $\Sigma$-case we assume that $M(n)$ are $S_n$-modules).\footnote{We say \lq\lq{\it weak}" because the left action \lq\lq is weak". The usual left action is given by a series of maps
$\calO(k)\otimes \bigl(M(k_1)\otimes M(k_2)\otimes\ldots M(k_n)\bigr)\to M(k_1+\ldots +k_n)$.}
\end{definition}

The result of composition $\overline{\circ}_i(o,m)$, and $\underline{\circ}_i(m,o)$, for $o\in\calO(n)$, and $m\in M(k)$, will be denoted by $o\circ_i m$, and $m\circ_i o$.

\vspace{.3cm}

\begin{figure}[h]
\psfrag{OM}[0][0][1][0]{$o\circ_3 m$}
\psfrag{MO}[0][0][1][0]{$m\circ_2 o$}
\psfrag{o}[0][0][1][0]{$o$}
\psfrag{m}[0][0][1][0]{$m$}
\includegraphics[width=16cm]{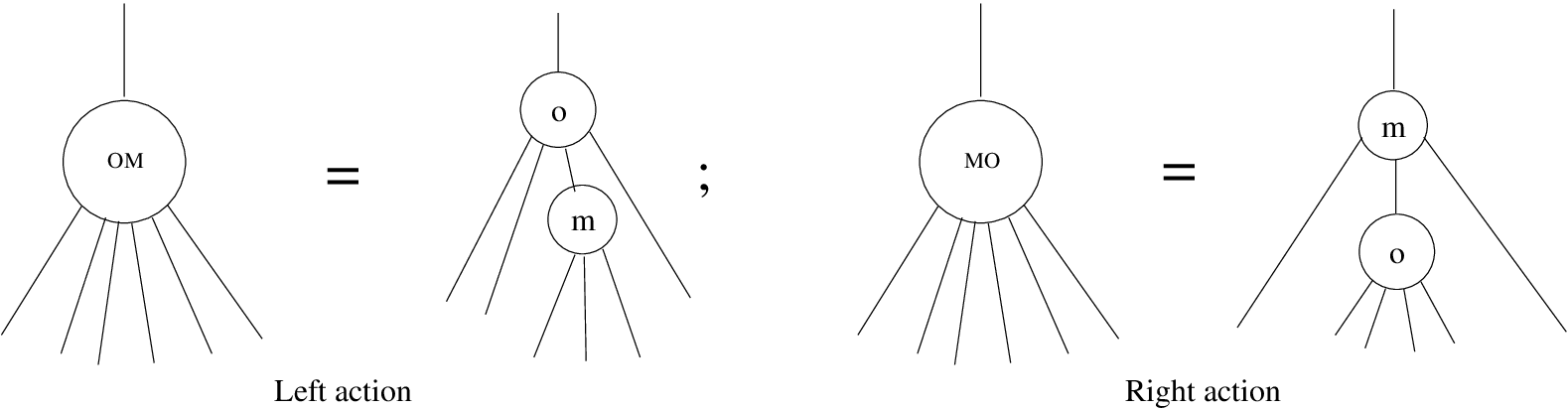}
\caption{}\label{fig1}
\end{figure}


We will also adopt notation using formal variables. For example, $o\circ_3 m$ from the above picture will be denoted as $o(x_1,x_2,m(x_3,x_4,x_5),x_6)$, and $m\circ_2o$ as $m(x_1,o(x_2,x_3,x_4,x_5),x_6)$.

\begin{lemma}\label{l41}
The structure of a cosimplicial vector space is equivalent to the structure of a weak non-$\Sigma$ bimodule over (non-$\Sigma$) $\ASS$.
\end{lemma}

\begin{lemma}\label{l42}
The structure of a $\Gamma^{op}$-module is equivalent to the structure of a weak bimodule over $\COMM$.
\end{lemma}

The same is true in the differential graded case in which the acting operads $\ASS$ and $\COMM$ are considered with trivial differential.

\begin{proof}[Proof of Lemmas~\ref{l41}, \ref{l42}] Instead of giving a formal proof let us consider a few examples how  cosimplicial and $\Gamma^{op}$ structure maps correspond to the composition operations~(\ref{eq_left_action}-\ref{eq_right_action}). Let $M=\{ M(n),\, n\geq 0\}$
be a non-$\Sigma$ weak bimodule over (non-$\Sigma$) $\ASS$. Let us pick $m=6$, $n=3$, and some map $\alpha\in Mor_{\Delta^{op}}([6],[3])$:


\begin{center}
\includegraphics[width=4cm]{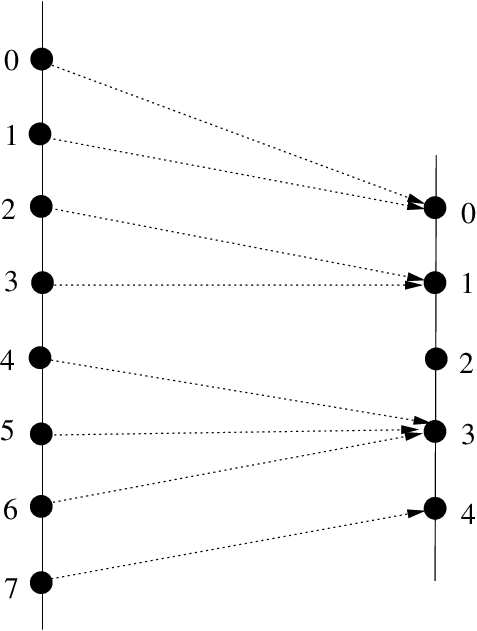}
\end{center}

\vspace{0cm}

By Remark~\ref{r31} this morphism corresponds to the following morphism in $Mor_{\Gamma_{cycl}}(\underline{6},\underline{3})$.

\vspace{.2cm}

\begin{center}
\includegraphics[width=5cm]{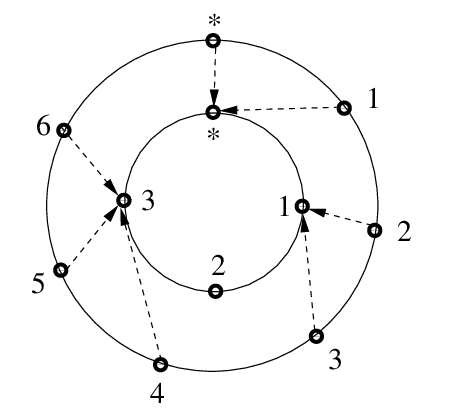}
\end{center}


 We define the map
$\alpha_M\colon M(3)\to M(6)$ using the compositions~\eqref{eq_left_action},~\eqref{eq_right_action}:

\vspace{.2cm}

\begin{center}
\psfrag{Alp}[0][0][1][0]{$\alpha_M(m)$}
\psfrag{m}[0][0][1][0]{$m$}
\includegraphics[width=10cm]{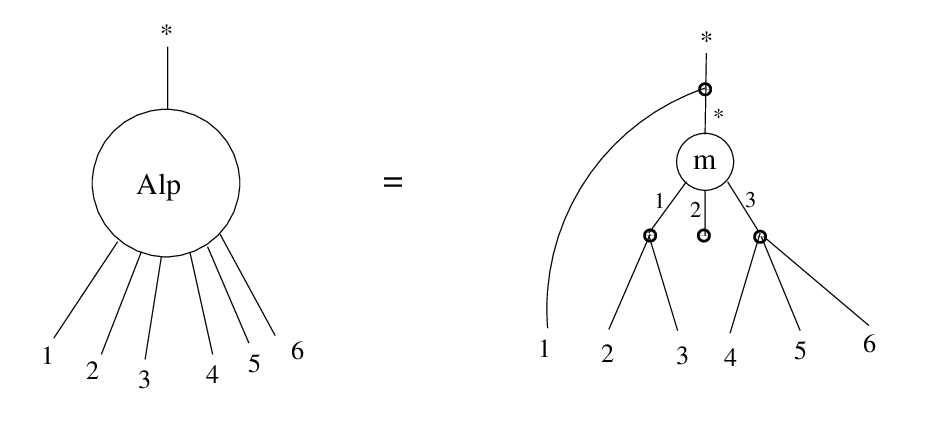}
\end{center}

Or equivalently
$$
\alpha_M(m)(x_1,x_2,x_3,x_4,x_5,x_6)=x_1m(x_2x_3,1,x_4x_5x_6).
$$
By
 \begin{picture}(10,15)
 \put(5,8){\circle*{3}}
 \put(5,8){\line(0,1){10}}
 \end{picture},
 \begin{picture}(10,15)
 \put(5,8){\circle*{3}}
 \put(5,8){\line(0,1){10}}
 \put(5,8){\line(0,-1){10}}
 \end{picture},
  \begin{picture}(10,15)
 \put(5,8){\circle*{3}}
 \put(5,8){\line(0,1){10}}
  \put(5,8){\line(-1,-2){5}}
 \put(5,8){\line(1,-2){5}}
 \end{picture},
  \begin{picture}(10,15)
 \put(5,8){\circle*{3}}
 \put(5,8){\line(0,1){10}}
  \put(5,8){\line(-1,-2){5}}
 \put(5,8){\line(1,-2){5}}
  \put(5,8){\line(0,-1){10}}
 \end{picture},
 $\ldots$
 in the above picture we represent respectively $1\in\ASS(0)$, $x_1\in\ASS(1)$, $x_1x_2\in\ASS(2)$, $x_1x_2x_3\in\ASS(3)$, $\ldots$. The idea is to denote the output of $m$ and of $\alpha(m)$ by $*$, and their inputs by $1,2,3,\ldots$, and then to see how the output and inputs of $\alpha_M(m)$ are connected to the output and inputs of $M$.

\vspace{.5cm}

Similarly if $M$ is a weak bimodule over $\COMM$, then any morphism $\alpha\colon\underline{m}\to\underline{n}$ defines a map $\alpha_M\colon M(n)\to M(m)$. For example, the map $\alpha\colon \underline{4}\to\underline{2}$:

\begin{center}
\includegraphics[width=3cm]{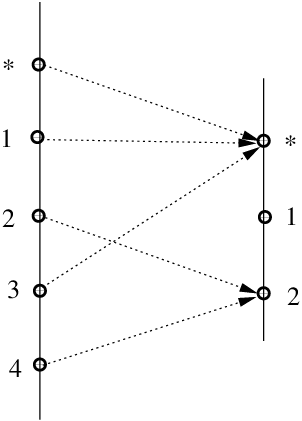}
\end{center}

defines a map $\alpha_M\colon M(2)\to M(4)$ constructed as follows:

\vspace{.2cm}

\begin{center}
\psfrag{Alp}[0][0][1][0]{$\alpha_M(m)$}
\psfrag{m}[0][0][1][0]{$m$}
\includegraphics[width=10cm]{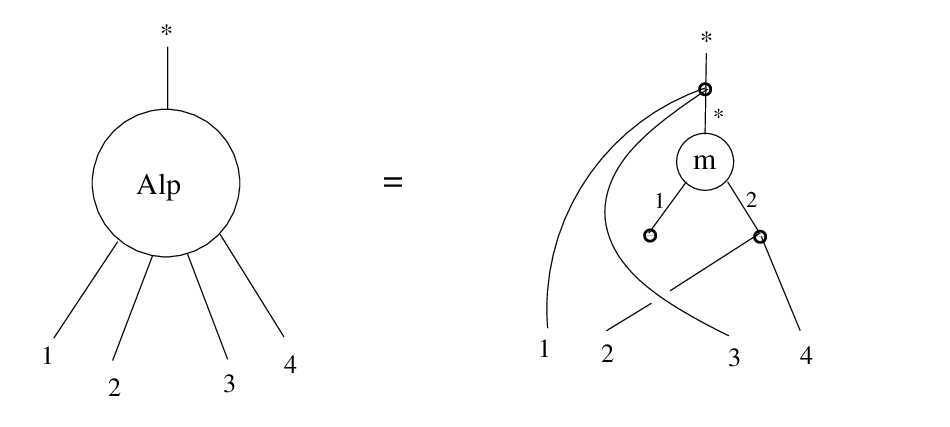}
\end{center}

Equivalently $\alpha_M(m)(x_1,x_2,x_3,x_4)=x_1 x_3 m(1,x_2x_4)$.
\end{proof}

We have seen that weak bimodules over $\COMM$ and weak non-$\Sigma$ bimodules over $\ASS$ have a very natural interpretation. Actually weak bimodules over $\ASS$ are also very common objects. As example the Hochschild complex of any (not necessarily symmetric) algebra $A$ with coefficients in its any bimodule has this structure. The category that encodes this structure is also well-known it is the category of finite pointed non-commutative sets~\cite{PR}.

The following definition is not standard, but will be convenient for the language of the paper.

\begin{definition}\label{d43}
(a) A weak bimodule over ($\Sigma$ operad) $\ASS$ will be called {\it $\Sigma$-cosimplicial space}.

(b) A weak bimodule over $\COMM$ or equivalently $\Gamma^{op}$-module will be called {\it commutative $\Sigma$-cosimplicial space}.
\end{definition}

 Roughly speaking $\Sigma$-cosimplicial spaces are cosimplicial spaces with $S_n$ action on each component, and   commutative $\Sigma$-cosimplicial spaces  are $\Sigma$-cosimplicial spaces for which this $S_n$ action is nicely compatible with the face maps.

\subsection{In the category of topological spaces}\label{ss41}
Similar constructions (Lemmas~\ref{l41}-\ref{l42}, Definition~\ref{d43}) work well for any symmetric monoidal category with associative coproducts. Again abusing the notation $\COMM=\{\COMM(n),\, n\geq 0\}=\{*,\, n\geq 0\}$ will denote the topological commutative operad and $\ASS$ will denote both $\Sigma$ and non-$\Sigma$ associative operads:
\begin{gather*}
\text{non-$\Sigma$ case: } \ASS=\{\ASS(n)=*,\, n\geq 0\}.\\
\text{$\Sigma$ case: }  \ASS=\{\ASS(n)=S_n,\, n\geq 0\}.
\end{gather*}

\subsection{Subcomplex of alternative multiderivations}\label{ss42}
There is one important case when the Hochschild cohomology is easy to compute. This is when $A$ is a smooth algebra, and $M$ is its flat symmetric bimodule. In this case the cohomology $HH^*(A,M)$ is described by the space of alternative multiderivations.

\begin{definition}\label{d_multiderivations} Let $Y^\bullet$ be a commutative $\Sigma$-cosimplicial $dg$-vector space.

(a) An element $y\in Y^\bullet$ is called alternative if for any $\sigma\in S_n$
$$
y(x_{\sigma_1},x_{\sigma_2},\ldots,x_{\sigma_n})=(-1)^{|\sigma|} y(x_1,x_2,\ldots,x_n).
\eqno(\numb)\label{eq_alternative}
$$

(b) An element $y\in Y^n$ is called a multiderivation if for any $i=1\ldots n$ one has
\begin{multline}
y(x_1,\ldots,x_{i-1},x_ix_{i+1},x_{i+2},\ldots,x_{n+1})=\\
x_iy(x_1,\ldots,\hat x_i,\ldots,x_{n+1})+y(x_1,\ldots,\hat x_{i+1},\ldots,x_{n+1})x_{i+1}.
\label{eq_multi}
\end{multline}
\end{definition}

Let $AM(Y^n)$ denote the subspace of alternative multiderivations in $Y^n$.

\begin{lemma}\label{l45}
 The space $AM(Y^n)$, $n\geq 0$, is a subcomplex of $\Tot Y^\bullet$.
\end{lemma}

\begin{proof} Let $y\in AM(Y^n)$. For any $1\leq i\le n$ one has
\begin{multline*}
s_iy=y(x_1,\ldots,x_{i-1},1,x_{i},\ldots,x_{n-1})=y(x_1,\ldots,x_{i-1},1\cdot 1,x_{i},\ldots,x_{n-1})=\\
1\cdot y(x_1,\ldots,x_{i-1},1,x_{i},\ldots,x_{n-1})+y(x_1,\ldots,x_{i-1},1,x_{i},\ldots,x_{n-1})\cdot 1= 2s_iy.
\end{multline*}
Therefore $s_iy=0$.

On the other hand the external part of the differential, which is the alternative sum of cofaces $d_i$, takes any multiderivation to zero, which can be checked by similar computations. One has also that the internal part of the differential preserves the alternative and multiderivation properties since it commutes with $\Gamma^{op}$ structure maps.
\end{proof}

The complex $\oplus_{n\geq 0}AM(Y^n)$ will be also denoted by $AM(Y^\bullet)$. The following definition is inspired by the example from the beginning of the section.

\begin{definition}\label{d_smooth} A commutative $\Sigma$-cosimplicial $dg$-vector space will be called {\it smooth} if the inclusion
$$
AM(Y^\bullet)\hookrightarrow \Tot Y^\bullet
$$
is a quasi-isomorphism.\footnote{This definition is a weaker version of~\cite[Definition~4.3]{Pirashvili} given by Pirashvili. \cite[Theorem~4.6]{Pirashvili} implies that smooth $\Gamma$-modules in the sense of Pirashvili are always smooth in our definition.}
\end{definition}

\begin{remark}\label{r46}
For a smooth commutative $\Sigma$-cosimplicial $dg$-vectror space, the Hodge decomposition in the homology can be easily understood:
$$
H_*^{(i)}(\Tot Y^\bullet)=H_*(AM(Y^i)).
$$
This follows from the fact that $AM(Y^i)$ lies in the image of $e_i^{(i)}=\frac 1{i!}\sum_{\sigma\in S_i}(-1)^{|\sigma|}\sigma.$
\end{remark}

\section{Hodge decomposition in the homology and homotopy of long knots}\label{s5}
Let $Top$ denote the category of topological spaces, and $hoTop$ denote the category with the same objects $Ob(hoTop)=Ob(Top)$, but with the morphisms being the homotopy classes of maps. One has a forgetful functor:
$$
Top \stackrel h\longrightarrow hoTop.
\eqno(\numb)\label{eq51}
$$

\begin{proposition}\label{p51}
The cosimplicial space ${C_d}^\bullet$ (see Section~\ref{s2}) has the property that $h\circ {C_d}^\bullet$ factors through $\Gamma^{op}$:
$$
\xymatrix{
\Delta\ar[r]^-{{C_d}^\bullet}\ar[rd]_\rho&Top\ar[r]^-h&hoTop\\
&\Gamma^{op}\ar[ru]&
}
\eqno(\numb)\label{eq:homot_factor}
$$
or in other words ${C_d}^\bullet$ is a commutative $\Sigma$-cosimplicial space in $hoTop$.
\end{proposition}

\begin{proof}
D.~Sinha constructed another cosimplicial model for $\overline{Emb}_d$, which is homotopy equivalent to the one that we briefly described in Section~2. The $n$-th stage of this model is the $n$-th component of some operad $\calK_d$ which is homotopy equivalent to the operad of little $d$-cubes. Using a natural inclusion
$$
\ASS\hookrightarrow \calK_d,
\eqno(\numb)\label{eq53}
$$
$\calK_d$ becomes a cosimplicial space (being a bimodule over $\ASS$, see Lemma~\ref{l41}). But~\eqref{eq53} is a morphism of $\Sigma$-operads. Therefore ${\calK_d}^\bullet$ is a $\Sigma$-cosimplicial space. But notice that all the components $K_d(n),$ $n\geq 0$ are connected, therefore in the category $hoTop$ the morphism~\eqref{eq53} factors though $\COMM$:
$$
\xymatrix{
\ASS\ar[r]\ar[d]&K_d\\
\COMM\ar@{.>}[ru]&
}
\eqno(\numb)\label{eq54}
$$
As a consequence in $hoTop$ the operad $\calK_d$ is a weak bimodule over $\COMM$, or equivalently is a commutative $\Sigma$-cosimplicial space.
\end{proof}

\begin{cor}\label{c52} Since the homology and homotopy functors factor through $hoTop$, both the homology ${A_d}^\bullet=H_*({C_d}^\bullet,\BK)$ and homotopy ${L_d}^\bullet=\pi_*({C_d}^\bullet)\otimes\BK$ of ${C_d}^\bullet$ are commutative $\Sigma$-cosimplicial graded spaces.
\end{cor}

There is another but similar way to see that ${A_d}^\bullet$ is a weak bimodule over $\COMM$. The vector spaces $\{{A_d}^n,\, n\geq 0\}$ form an operad, which is the homology operad of the little $d$-cubes. By~\cite{Coh} it is the operad of $(d-1)$-Poisson algebras: graded commutative algebras with a Lie bracket of degree $(d-1)$ compatible with the product:
$$
[x_1,x_2x_3]=[x_1,x_2]x_3+(-1)^{|x_2|+(d-1)}x_2[x_1,x_3].
$$
This operad contains $\COMM$ and therefore is a weak bimodule over it.

\begin{cor-def}\label{c53} The complexes $\Tot {A_d}^\bullet$, $\Tot {L_d}^\bullet$ admit Hodge splitting
$$
\Tot {A_d}^\bullet=\oplus_i \Tot^{(i)}{A_d}^\bullet.
\eqno(\numb)\label{eq55}
$$
$$
\Tot {L_d}^\bullet=\oplus_i \Tot^{(i)}{L_d}^\bullet.
\eqno(\numb)\label{eq56}
$$
Since the above complexes compute the rational homology and homotopy of $\overline{Emb}_d$, this gives the Hodge decomposition in its homology and homotopy:
$$
H_*(\overline{Emb}_d,\Q)=\oplus_i H_*^{(i)}(\overline{Emb}_d,\Q):=\oplus_i H_*(\Tot^{(i)}{A_d}^\bullet),
\eqno(\numb)\label{eq57}
$$
$$
\pi_*(\overline{Emb}_d)\otimes\Q=\oplus_i\pi_*^{(i)}(\overline{Emb}_d,\Q):=\oplus_i H_*(\Tot^{(i)}{L_d}^\bullet).
\eqno(\numb)\label{eq58}
$$
\end{cor-def}

\section{Geometric interpretation of the Hodge decomposition. Cabling maps}\label{s6}
Studying knot spaces is in many aspects similar to studying loop spaces. Recall that the real cohomology of the loop space $\Omega M$ of a 1-connected variety $M$ is naturally isomorphic to the Hochschild homology of the De Rham algebra $\Omega^*(M)$ (of differential forms on $M$) with trivial coefficients:
$$
H^*(\Omega M,\R)\simeq HH_*(\Omega^*(M),\R).
\eqno(\numb)\label{eq61}
$$
Similarly the cohomology of the free loop space $\Lambda M$ can be expressed as the Hochschild cohomology of $\Omega^*(M)$ with coefficients in itself:
$$
H^*(\Lambda M,\R)\simeq HH_*(\Omega^*(M),\Omega^*(M)).
\eqno(\numb)\label{eq62}
$$
The differential algebra $\Omega^*(M)$ is graded commutative, and both bimodules $\R$, and $\Omega^*(M)$ are graded commutative over it. Hence it makes sense to consider the Hodge decomposition of~(\ref{eq61}-\ref{eq62}). Let $\Phi_n$ denote the power maps
$$
\Omega M \stackrel{\Phi_n}\longrightarrow\Omega M, \qquad \Lambda M\stackrel{\Phi_n}\longrightarrow\Lambda M.
$$
induced by a degree $n$ map of a circle into itself $S^1\stackrel{n}\longrightarrow S^1$. It was shown in~\cite{BFG_AdamsOp} that the Hodge decomposition in $HH_*(\Omega^*(M),\R)$ (resp. $HH_*(\Omega^*(M),\Omega^*(M))$) corresponds via isomorphisms~\eqref{eq61}, \eqref{eq62} to the decomposition into eigenspaces of the power maps in cohomology:
$$
H^*(\Omega M,\R)\stackrel{\Phi_n^*}\longrightarrow H^*(\Omega M,\R), \qquad H^*(\Lambda M,\R)\stackrel{\Phi_n^*}\longrightarrow H^*(\Lambda M,\R).
$$
More precisely the $HH_*^{(i)}$ term is always the eigenspace of $\Phi_n^*$ with eigenvalue $n^i$.

We believe that a similar assertion holds for the homotopy and homology of $\overline{Emb}_d$. But instead of power maps one has to use cabling maps (abusing notation we will denote them similarly):
$$
\Phi_n\colon \overline{Emb}_d\to\overline{Emb}_d.
$$
The construction of $\Phi_n$ is similar to the proof of Proposition~\ref{p11} and goes in 2 steps. First one constructs such maps for the spaces of long framed embeddings $Emb_d^{framed}$, and long framed immersions $Imm_d^{framed}$:
$$
\xymatrix{
Emb_d^{framed}\ar[r]^{\Phi_n}\ar@{^{(}->}[d]&Emb_d^{framed}\ar@{^{(}->}[d]\\
Imm_d^{framed}\ar[r]^{\Phi_n}&Imm_d^{framed}.
}
\eqno(\numb)\label{eq63}
$$
Below we show how to construct a 2-cabling of a framed long knot (framing is necessary to make the construction work
on the level of spaces):

\vspace{.2cm}

\begin{center}
\psfrag{Phi}[0][0][1][0]{$\Phi_2$}
\includegraphics[width=15cm]{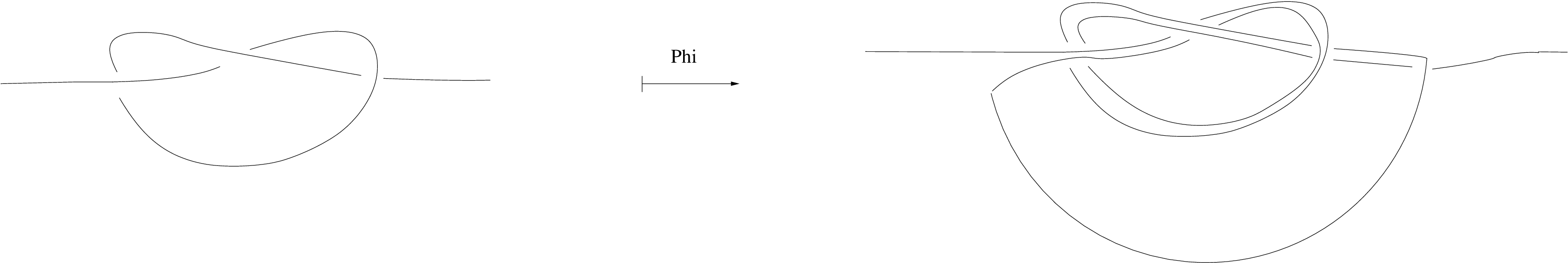}
\end{center}

\vspace{.2cm}

This construction works nicely if one uses Budney's model for the space of framed long knots~\cite{Budney} (the space denoted by $EC(1,D^{d-1})$), and its analogue for the space of framed immersions. Since the diagram~\eqref{eq63} commutes, one obtains the induced \lq\lq cabling" map $\Phi_n$ on the homotopy fiber of the inclusion $Emb_d^{framed}\hookrightarrow Imm_d^{framed}$. But this homotopy fiber is exactly $\overline{Emb}_d$.

\section{Relation between the Hodge decomposition in the homology and homotopy}\label{s7}
Rational homotopy of $\overline{Emb}_d$ is isomorphic to the primitive part of the rational homology, see Corollary~\ref{c12}. By Theorem~\ref{t21},
$$
H_*(\overline{Emb}_d,\Q)=H_*(\Tot {A_d}^\bullet), \qquad \pi_*(\overline{Emb}_d)\otimes \Q =H_*(\Tot {L_d}^\bullet).
$$
Therefore the space of primitives of $H_*(\Tot {A_d}^\bullet)$ is isomorphic to $H_*(\Tot {L_d}^\bullet)$:
$$
Prim(H_*(\Tot {A_d}^\bullet))\simeq H_*(\Tot {L_d}^\bullet)
\eqno(\numb)\label{eq72}
$$
A pure algebraic proof for this isomorphism is given in~\cite{LT}.

\begin{proposition}\label{p71}
(i) The coproduct structure of $\Tot A^\bullet$ respects the Hodge degree:
$$
\Delta\Tot^{(k)}{A_d}^\bullet\subset\bigoplus_{i=0}^k\Tot^{(i)}{A_d}^\bullet\otimes\Tot^{(k-i)}{A_d}^\bullet.
$$
(ii) The isomorphism~\eqref{eq72} respects the Hodge decomposition.
\end{proposition}

The above proposition means that the study of the Hodge decomposition in the homology or in the homotopy are two equivalent problems, since the homology $H_*(\Tot {A_d}^\bullet)$ is a free graded cocommutative coalgebra over $Prim(H_*(\Tot {A_d}^\bullet))\simeq H_*(\Tot {L_d}^\bullet)$.

\begin{proof}
The proof of (i) follows from the following lemma:

\begin{lemma}\label{l72}
Let $B^\bullet$ and $C^\bullet$ be two commutative $\Sigma$-cosimplicial spaces, then $B^\bullet\otimes C^\bullet$ is also a commutative $\Sigma$-cosimplicial space and moreover  the Eilenberg-Mac Lane map
$$
\Tot(B^\bullet\otimes C^\bullet)\to \Tot B^\bullet\otimes\Tot C^\bullet
$$
respects the Hodge degree.
\end{lemma}

\begin{proof}
This fact is well known. For example, one can see it from the geometric approach of F.~Patras to describe Adams operations~\cite{Patras1,Patras2} (see in particular Proposition~1.3 of~\cite{Patras2} and the remark that follows.)
\end{proof}

To prove (i) we notice that the coproduct on $\Tot {A_d}^\bullet$ is obtained as a composition of two maps:
$$
\Tot {A_d}^\bullet\stackrel{\Delta^\bullet}\longrightarrow\Tot({A_d}^\bullet\otimes {A_d}^\bullet)\stackrel{EM}\longrightarrow
\Tot {A_d}^\bullet\otimes \Tot {A_d}^\bullet.
$$
The first map $\Delta^\bullet$ is induced by a degree-wise coproduct (in the homology of configuration spaces), which is a morphism of commutative $\Sigma$-cosimplicial spaces. The second one respects the Hodge degree by Lemma~\ref{l72}.

(ii) Recall that the proof of~\eqref{eq72} was given by the following sequence of quasi-isomorphisms:
$$
 \xymatrix{
 \Tot {A_d}_\bullet/(\Tot {A_d}_\bullet)^2&\calL(\Tot
 {A_d}_\bullet)\ar@{->>}[l]_-\simeq^-\alpha\ar[r]^-\simeq_-{EM_\calL}&
 \Tot\calL({A_d}_\bullet)&\,\Tot({L_d}_\bullet),\ar@{_{(}->}[l]_-\simeq
 }\eqno(\numb)\label{eq_zig-zag}
 $$
where ${A_d}_\bullet$ is the simplicial graded commutative algebra dual to ${A_d}^\bullet$  (i.e. ${A_d}_n=H^*({C_d}^n,\Q)$), and ${L_d}_\bullet$ is
the simplicial graded Lie coalgebra dual to ${L_d}^\bullet$ (i.e. ${L_d}_n=Mor(\pi_*({C_d}^n),\Q)$). The algebra $\Tot {A_d}_\bullet$ is polynomial,
$\Tot {A_d}_\bullet/(\Tot {A_d}_\bullet)^2$ describes the quotient complex by all the non-generators. The homology of $\Tot {A_d}_\bullet/(\Tot {A_d}_\bullet)^2$ is
$Prim(H_*(\Tot {A_d}_\bullet))$. The functor $\calL(-)$ is the cobar construction that assigns to any commutative $dg$-algebra a free $dg$-Lie coalgebra. The first, and the second quasi-isomorphisms, $\alpha$, and $EM_\calL$, respect the Hodge decomposition by Lemma~\ref{l72}. The last one is induced by a morphism of differential graded $\Gamma$-modules and therefore  respects the Hodge degree.
\end{proof}

\part{Graph-complexes}\label{p2}

\section{Bialgebra of chord diagrams. Generalizing a result of Bar-Natan}\label{s8}
The bialgebra of chord diagrams is a well-known object in Low Dimensional Topology which encodes the so called Vassiliev invariants of knots. Bar-Natan has shown that this bialgebra is isomorphic to the bialgebra of unitrivalent graphs attached to a line modulo $STU$, $IHX$, and $AS$ relations~\cite{BarNatan}. Theorem~8.6 of~\cite{LT} generalizes and also gives another proof of this result using graph-complexes. In this paper we will go a step further and will generalize the following result of Bar-Natan:

\begin{thm}[Bar-Natan~\cite{BarNatan}]\label{t81}
The space of primitive elements of the bialgebra of chord diagrams is isomorphic to the space of connected unitrivalent graphs (with at least one univalent vertex) modulo $IHX$, and $AS$ relations.
\end{thm}

\begin{figure}[h]
\includegraphics[width=6cm]{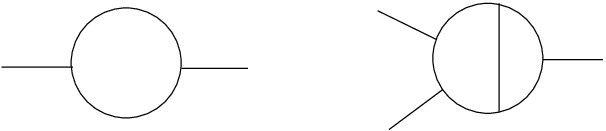}
\caption{Examples of uni-trivalent graphs}\label{fig2}
\end{figure}

The bialgebra of chord diagrams is naturally graded by {\it complexity} - number of chords. The complexity of a unitrivalent graph is the Betti number of the graph obtained by gluing together all univalent vertices, see Figure~\ref{f:gluing}.

\begin{figure}[h]
\includegraphics[width=8cm]{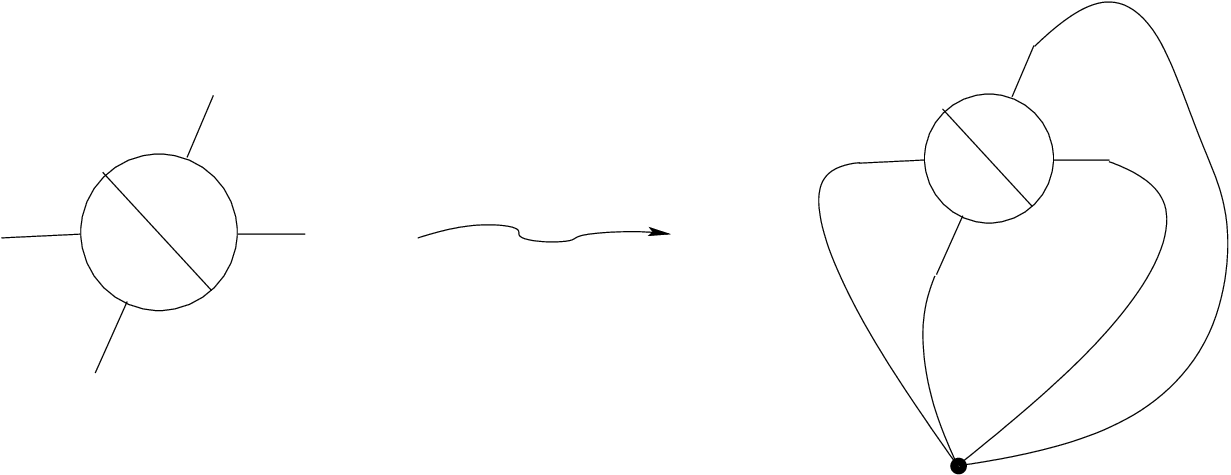}
\caption{}\label{f:gluing}
\end{figure}

From this theorem we see that the space of primitives has another grading which is the number of univalent vertices.

\begin{thm}\label{t82}
The homotopy $\pi_*(Emb_d)\otimes\Q$, resp. the homology $H_*(\overline{Emb}_d,\Q)$ is computed by the graph-complex $AM({P_d}^\bullet)$, resp. $AM({D_d}^\bullet)$ (that we define below) of connected, resp. possibly disconnected or empty uni-$\geq 3$-valent graphs. Moreover via this isomorphism the complexity grading is the first Betti number of the graph obtained by gluing together all the univalent vertices, and the Hodge degree is the number of univalent vertices.
\end{thm}

This theorem is proved in Section~\ref{s10}.

Below we describe explicitly the complexes $AM({P_d}^\bullet)$, $AM({D_d}^\bullet)$.

\subsection{Definition of $AM({P_d}^\bullet)$, $AM({D_d}^\bullet)$}\label{ss81}
Let us first define the complex $AM({P_d}^\bullet)$. By a connected uni-$\geq 3$-valent graph we mean a connected graph with a non-empty set of univalent vertices (which are called external) and some (possibly empty) set of vertices of valence at least 3 (those vertices are called external). The graphs are allowed to have both multiple edges and loops --- edges connecting a vertex to itself:

\begin{figure}[h]
\includegraphics[width=12cm]{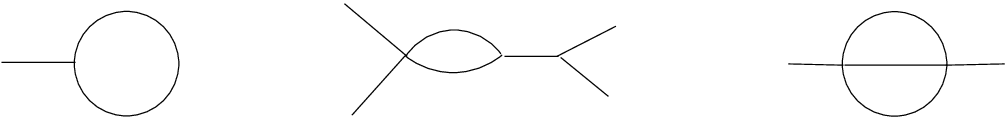}
\caption{Examples of connected uni-$\geq 3$-valent graphs}\label{fig3}
\end{figure}

Neither external, nor internal vertices are labeled. Two graphs are considered to be equivalent if there is a bijection between their sets of vertices and edges respecting the adjacency structure of the graphs.

The orientation set  of a graph is the set of its external (univalent) vertices (those vertices are considered to have degree $-1$), its internal vertices (which are considered to have degree $-d$), and its edges (considered to have degree $(d-1)$). The total grading of a graph is the sum of gradings of its edges, and vertices. An {\it orientation} of a graph is an ordering of its orientation set together with fixing of orientation of each edge.

\begin{definition}\label{d83}
The space of $AM({P_d}^\bullet)$ is a graded vector space over $\BK$ spanned by the above connected uni-$\geq 3$-valent graphs modulo the orientation relations:

(1) If $\Gamma_1$ and $\Gamma_2$ differ only by an orientation of an edge, then
$$
\Gamma_1=(-1)^d\Gamma_2.
$$

(2) If $\Gamma_1$ and $\Gamma_2$ differ only by a permutation of an orientation set, then
$$
\Gamma_1=\pm\Gamma_2,
$$
where the sign is a Koszul sign of permutation (taking into account the degrees of elements).

The differential in $AM({P_d}^\bullet)$ is a sum of expansions of internal vertices, see example below.
\end{definition}

\vspace{.1cm}

\begin{center}
\psfrag{d}[0][0][1][0]{$\partial$}
\psfrag{p}[0][0][1][0]{$\pm$}
\includegraphics[width=14cm]{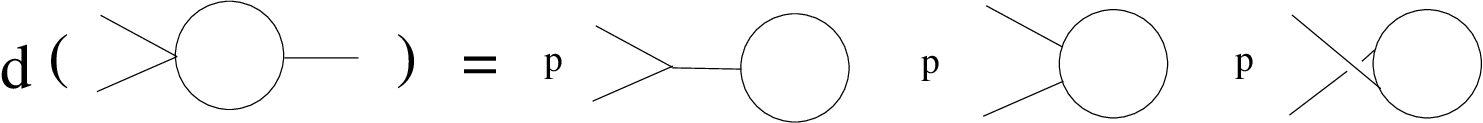}
\end{center}

\vspace{.2cm}

Together with the main grading which is the sum of gradings of edges and vertices, we will consider two additional gradings. The first one, {\it complexity}, is defined as the first Betti number of the graph obtained by gluing all univalent vertices together, see Figure~\ref{f:gluing}. The second one, {\it Hodge degree}, is defined as the number of external vertices (we will later on see that indeed it corresponds to to the Hodge degree of some Hochschild complex). The differential in $AM({P_d}^\bullet)$ preserves both the complexity and the Hodge degree, therefore the complex $AM({P_d}^\bullet)$ splits into a direct sum of complexes:
$$
AM({P_d}^\bullet)=\oplus_{i,j}AM_j({P_d}^i)
$$
by complexity $j$, and Hodge degree $i$.

\begin{remark}\label{r83}
Similar graph-complexes appear in the study of the homology of the outer spaces~\cite{CV:ThKon,CullerVogtmann,Kon:FSG,LazVor}. The difference is that in our case the graphs  have univalent vertices.
\end{remark}

The second complex $AM({D_d}^\bullet)$ can be defined as a free differential graded cocommutative coalgebra cogenerated by the complex $AM({P_d}^\bullet)$. This complex can also be viewed as a graph-complex of possibly empty or disconnected uni-$\geq 3$-valent graphs, each connected component being from the space $AM({P_d}^\bullet)$.

\vspace{0.3cm}

\begin{figure}[h]
\includegraphics[width=5cm]{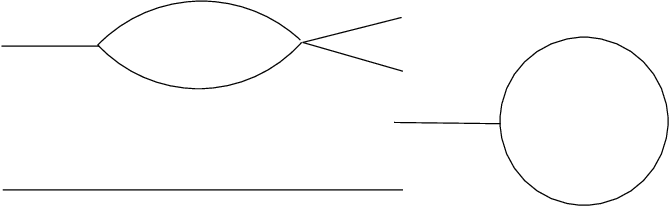}
\caption{Example of a graph from $AM({D_d}^\bullet)$. It has 3 connected components and has in total 6 external, and 3 internal vertices.}\label{fig4}
\end{figure}

The notation $AM(-)$ comes from \lq\lq Alternative Multiderivations" (see Section~\ref{ss42}). We will see in Section~\ref{s10} that they are indeed  isomorphic to the complexes of alternative multiderivations of the commutative $\Sigma$-cosimplicial $dg$-vector spaces ${P_d}^\bullet$, and ${D_d}^\bullet$.

\section{Operadic graph-complexes}\label{s9}
To recall the components of the cosimplicial space ${A_d}^\bullet=\{{A_d}^n,\, n\geq 0\}=\{H_*({C_d}^n),\, n\geq 0\}$ form the operad of $(d-1)$-Poisson algebras, that we also denote by $A_d$.\footnote{In~\cite{LTV,T-HLN,T-OS} this operad was denoted by $\Poiss_{d-1}$. We switched the notation for a convenience of presentation.} In this section we define a series of graph-complexes $D_d=\{D_d(n),\, n\geq 0\}$, that form an operad quasi-isomorphic to the operad $A_d$. Moreover the quasi-isomorphism is given by a natural inclusion. The construction of $D_d$ takes roots in~\cite{Kon} and had a number of remakes~\cite{LT,LambVol-FLBO}. The difference of our construction is that we obtain an operad, and not a cooperad like in~\cite{Kon,LT,LambVol-FLBO}. So our construction is dual to the previous ones.  Another difference is that we also allow a little bit more admissible graphs: we permit graphs to have multiple edges (contrary to~\cite{Kon,LambVol-FLBO}, but similarly to~\cite{LT}), and loops --- edges connecting a vertex to itself (contrary to all previous versions). The $n$-th component $D_d(n)$ is spanned by {\it admissible graphs} that have $n$ external vertices labeled by $1,2,\ldots,n$, and some number of non-labeled internal vertices. The external vertices can be of any non-negative valence, the internal ones should be of valence at least 3. The only condition we put on graphs -- each connected component should contain at least one external vertex (no pieces flying in air).

\vspace{0.3cm}

\begin{figure}[h]
\includegraphics[width=8cm]{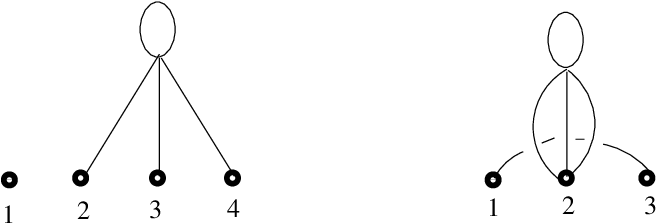}
\caption{Examples of graphs in $D_d(4)$ and $D_d(3)$.}\label{fig5}
\end{figure}


The {\it orientation set} of such a graph is the union of the set of internal vertices (those elements are considered to be of degree $-d$), and the set of edges (such elements are of degree $(d-1)$). Similarly to Section~\ref{s8} we will say that a graph is {\it oriented} if one fixes orientation of all its edges, and an ordering of its orientation set. The space $D_d(n)$ is defined to be spanned by all oriented admissible graphs  (with $n$ external vertices) modulo the orientation relations (1)-(2) similar to those of Definition~\ref{d83}.

\begin{remark}\label{r81}
For even $d$ the orientation relations kill the graphs with multiple edges, and for odd $d$  --- the graphs with loops.
\end{remark}

The differential is the sum of expansions of vertices (being dual to to the sum of contractions of edges). It diminishes the total degree by~1. One should be a little bit careful with the external vertices. An expansion of an external vertex produces one external vertex (with the same label) and one internal one.

\vspace{0.3cm}

\begin{center}
\psfrag{d}[0][0][1][0]{$\partial$}
\psfrag{p}[0][0][1][0]{$\pm$}
\includegraphics[width=12cm]{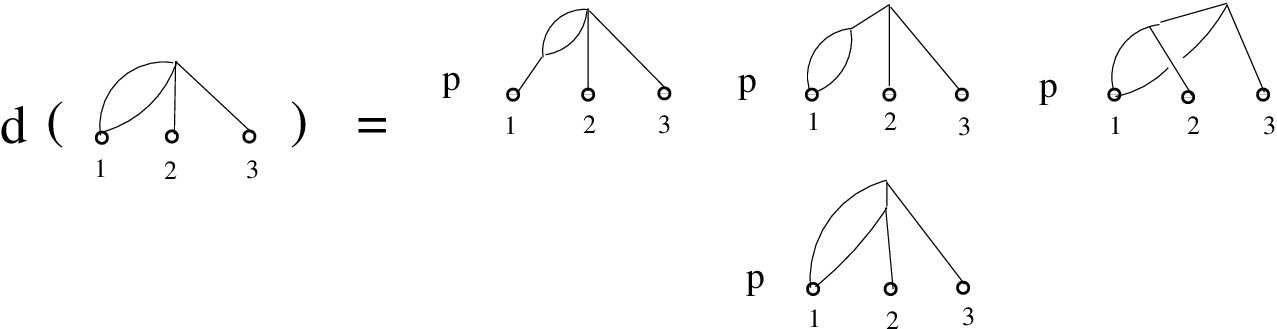}
\end{center}

\vspace{.2cm}

For the rule of signs see~\cite{LambVol-FLBO}, or define any reasonable one by yourself.

The operadic structure defined by compositions
$$
\circ_i\colon D_d(n)\otimes D_d(m)\to D_d(n+m-1)
$$
is dual to the cooperadic structure in~\cite{LambVol-FLBO}. If $\Gamma_1\in D_d(n)$, $\Gamma_2\in D_d(m)$ two graphs, then $\Gamma_1\circ_i\Gamma_2$ is the sum of graphs obtained by making $\Gamma_2$ very small and inserting it in the $i$-th external point of $\Gamma_1$. The edges adjacent to the external vertex $i$ in $\Gamma_1$ become adjacent to one of the vertices of $\Gamma_2$. The sum is taken by all such insertions, see Figure~\ref{fig6}.

\vspace{0.3cm}

\begin{figure}[h]
\psfrag{c2}[0][0][1][0]{$\circ_2$}
\psfrag{c3}[0][0][1][0]{$\circ_3$}
\psfrag{p}[0][0][1][0]{$\pm$}
\includegraphics[width=14cm]{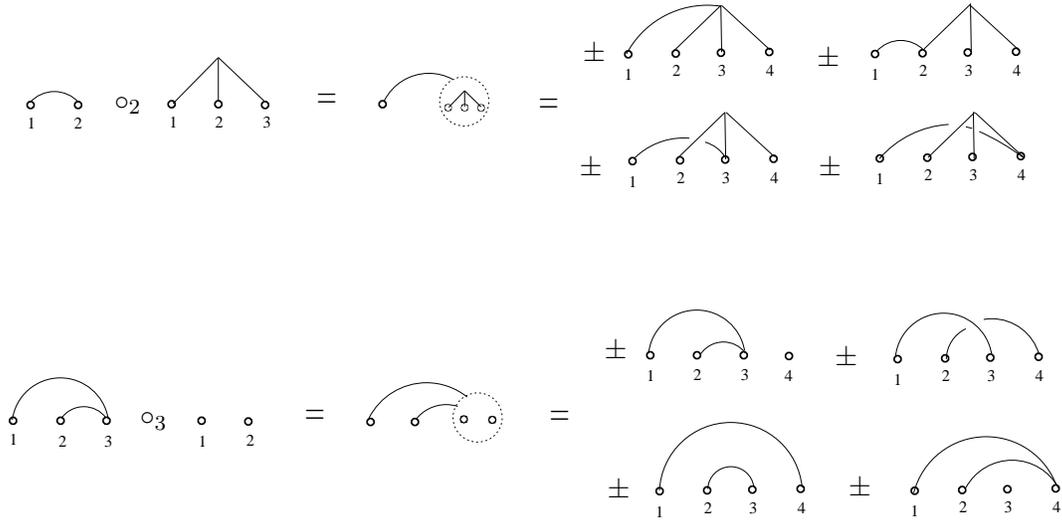}
\caption{Examples of composition}\label{fig6}
\end{figure}
\begin{proposition}\label{p91}
The assignment
$$
x_1x_2\mapsto \,
\begin{picture}(20,15)
\put(2,5){\circle*{3}}
\put(18,5){\circle*{3}}
\put(0,-1){$_1$}
\put(16,-1){$_2$}
\end{picture}\, , \qquad [x_1,x_2]\mapsto \,
\begin{picture}(20,15)
\put(2,5){\circle*{3}}
\put(18,5){\circle*{3}}
\qbezier(2,5)(10,15)(18,5)
\put(0,-1){$_1$}
\put(16,-1){$_2$}
\end{picture}\, ,
$$
where $x_1x_2,$ $[x_1,x_2]\in A_d(2)$ are the product and the bracket of the operad $A_d$ of $(d-1)$-Poisson algebras, defines an inclusion of operads
$$
\xymatrix{
A_d\,\ar@{^{(}->}[r]^\simeq&D_d
}
\eqno(\numb)\label{eq91}
$$
that turns out to be a quasi-isomorphism ($A_d$ is considered to have a zero differential).
\end{proposition}

\begin{proof} This assertion is dual to~\cite[Theorem~9.1]{LambVol-FLBO}. The graph-complexes used in~\cite{LambVol-FLBO} are slightly different, but the proof is  the same.
\end{proof}

Notice that the operad $A_d$ contains the operad $\COMM$. Therefore the inclusion~\eqref{eq91} can be viewed as a morphism of weak bimodules over $\COMM$. Denote by ${D_d}^\bullet=\{{D_d}^n,\, n\geq 0\}=\{D_d(n),\, n\geq 0\}$ the corresponding commutative $\Sigma$-cosimplicial space.

\begin{cor}\label{c92}
Inclusion~\eqref{eq91} defines a quasi-isomorphism of complexes
$$
\xymatrix{
\Tot {A_d}^\bullet\,\ar@{^{(}->}[r]^\simeq&\Tot {D_d}^\bullet
}
$$
that respects the Hodge splitting.
\end{cor}

The complex $\Tot {D_d}^\bullet$ is a graph-complex spanned by admissible graphs whose external vertices are of positive valence. The differential is the sum of expansions of vertices. An expansion of an external vertex can produce either two external vertices, or one external vertex and one internal one:

\vspace{0.3cm}

\begin{center}
\psfrag{d}[0][0][1][0]{$\partial_{\Tot {D_d}^\bullet}$}
\psfrag{p}[0][0][1][0]{$\pm$}
\includegraphics[width=13cm]{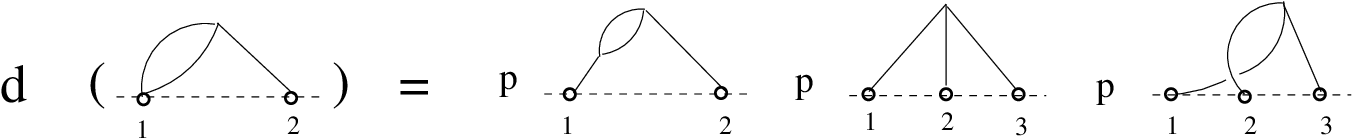}
\end{center}

\vspace{.3cm}

\begin{cor}\label{c93}
The graph-complex $\Tot {D_d}^\bullet$ computes the rational homology of the space $\overline{Emb}_d$ together with its Hodge splitting: for $\BK=\Q$ one has
$$
H_*(\Tot^{(i)}{D_d}^\bullet)=H_*^{(i)}(\overline{Emb}_d,\Q).
$$
\end{cor}

The operad $D_d$ is actually an operad in the category of graded cocommutative coalgebras.\footnote{Moreover~\eqref{eq91} is a morphism of such operads.} The coalgebra structure in each component $D_d(n)$ is given by a cosuperimposing:

\vspace{0.3cm}

\begin{center}
\psfrag{D}[0][0][1][0]{$\Delta$}
\psfrag{p}[0][0][1][0]{$\pm$}
\psfrag{o}[0][0][1][0]{$\otimes$}
\includegraphics[width=13cm]{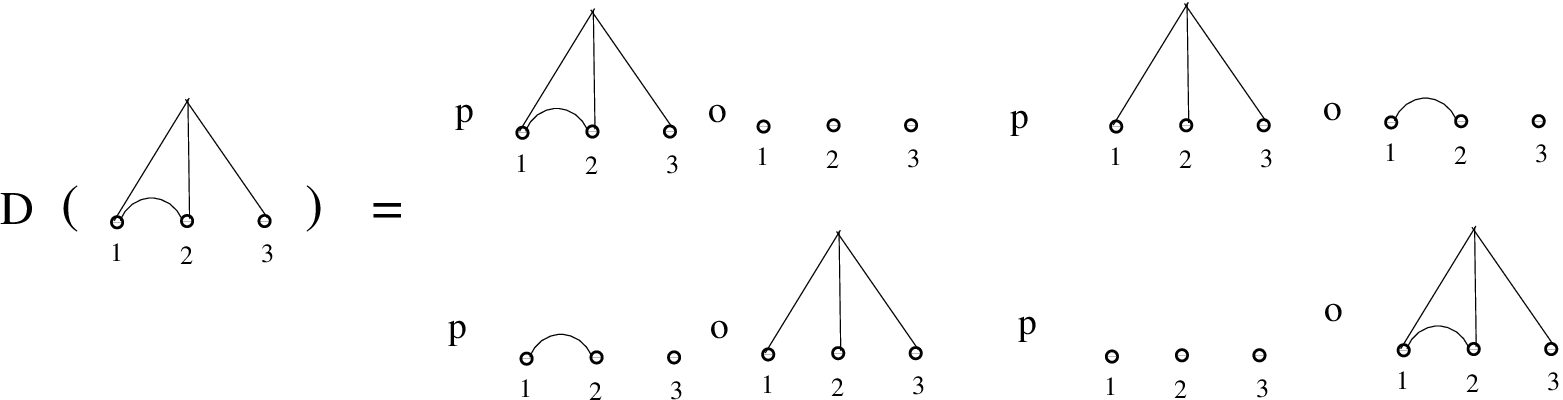}
\end{center}

\vspace{.3cm}

Let ${P_d}^n$ denote the primitive part of ${D_d}^n=D_d(n)$. The space ${P_d}^n$ is spanned by the graphs with $n$ external vertices that become connected if one removes all the external vertices together with their small vicinities.

The family of spaces ${P_d}^\bullet=\{{P_d}^n,\, n\geq 0\}$ is preserved by the commutative $\Sigma$-cosimplicial structure maps, simply because these maps respect the coalgebra structure of ${D_d}^n$, $n\geq 0$.

\begin{proposition}\label{p94}
The weak $\Sigma$-cosimplicial $dg$-vector spaces ${L_d}^\bullet$ and ${P_d}^\bullet$ are quasi-isomorphic (by a zig-zag of quasi-isomorphisms). As a consequence $\Tot {P_d}^\bullet$ computes the rational homotopy of $\overline{Emb}_d$ together with its Hodge splitting: for $\BK=\Q$ one has
$$
H_*(\Tot^{(i)}{P_d}^\bullet)=\pi_*^{(i)}(\overline{Emb}_d,\Q).
$$
\end{proposition}

\begin{proof} The proof of Theorem~9.4 in~\cite{LT} gives the necessary zig-zag.
\end{proof}

\section{${D_d}^\bullet$ and ${P_d}^\bullet$ are smooth}\label{s10}
Recall Definition~\ref{d_smooth} of a smooth commutative $\Sigma$-cosimplicial $dg$-vector space.

\begin{thm}\label{t101}
The commutative $\Sigma$-cosimplicial spaces ${D_d}^\bullet$ and ${P_d}^\bullet$ are smooth. Moreover their complexes of alternative multiderivations are isomorphic to the graph-complexes described in Section~\ref{ss81}.~\footnote{We abused the notation: in Section~\ref{ss81} we denoted the defined complexes  by $AM({D_d}^\bullet)$, $AM({P_d}^\bullet)$. We did it deliberately since they are naturally isomorphic to the complexes of alternative multiderivations that we  consider in this section.}
\end{thm}

An immediate corollary of this theorem (and of Corollary~\ref{c93}, and Proposition~\ref{p94}) is Theorem~\ref{t82}.

\begin{proof}
Let us first show that the complexes from Section~\ref{ss81} are indeed isomorphic to $AM({D_d}^\bullet)$, $AM({P_d}^\bullet)$. An element of ${D_d}^\bullet$ (resp. ${P_d}^\bullet$) is a multiderivation if and only if it is a linear combination of graphs whose all external vertices are univalent. Indeed, the equation~\eqref{eq_multi} implies that the valence of the $i$-th vertex for all the graphs in the sum is one.  On the other hand {\it alternative} means that we anti-symmetrize the external vertices, which exactly means that we put the weight $-1$ in each of them and then forget their labeling.

Now let us show that ${D_d}^\bullet$, ${P_d}^\bullet$ are indeed smooth. Denote by $M({D_d}^\bullet)=\oplus_nM({D_d}^n)$, $M({P_d}^\bullet)=\oplus_nM({P_d}^n)$
their subspaces of multiderivations. By a previous argument these subspaces are spanned by the graphs whose all external vertices are univalent. The graphs that span $M({P_d}^\bullet)$ are connected, those that span $M({D_d}^\bullet)$ might have any number of connected components.

Consider a free commutative algebra generated by $x_1,x_2,\ldots,x_n$. Let us take its normalized Hochschild complexes with coefficients in a trivial bimodule $\BK$. Let $K_n$ denote a subcomplex of this Hochschild complex spanned by the elements of degree 1 in each variable $x_i$, $i=1\ldots n$. Notice that $K_n$ is finite-dimensional. For example, $K_2$ is spanned by three elements: $x_1\otimes x_2$, $x_2\otimes x_1$, and $x_1x_2$. Let $K_n^\vee$ denote the dual of $K_n$.

In $\Tot {P_d}^\bullet$, $\Tot {D_d}^\bullet$ (and also in their subcomplexes $AM({P_d}^\bullet)$, $AM({P_d}^\bullet)$) consider the filtration by the number of internal vertices. The differential $d_0$ of the associated spectral sequence is the alternated sum of faces (external part of the differential in totalization). One can easily see that the term $E_0$ of the associated spectral sequence is isomorphic as a complex respectively to
$\oplus_{n\geq 0}K_n^\vee\otimes_{S_n}M({P_d}^n)$, $\oplus_{n\geq 0}K_n^\vee\otimes_{S_n}M(D_d^n)$, where $M({P_d}^n)$, $M({D_d}^n)$ are taken with zero differential.
%
The idea of these isomorphism is that any graph in $\Tot {P_d}^\bullet$ (resp. $\Tot {D_d}^\bullet$) can be obtained from a graph with all external vertices of valence 1 by gluing together consecutive external vertices. This correspondence is not unique that's why we take the tensor product over the symmetric group $S_n$.

The homology of $K_n^\vee$ is concentrated in the top degree and is isomorphic to the sign representation $sign_n$ of $S_n$. One can define an explicit inclusion
$$
sign_n\hookrightarrow K_n^\vee
$$
that defines a quasi-isomorphism of $dg$-$S_n$-modules. As a consequence the $E_1$ term is isomorphic to
$$
\oplus_nM({P_d}^n)\otimes_{S_n} sign_n,\qquad
\oplus_nM({D_d}^n)\otimes_{S_n} sign_n.
$$
respectively. But the above complexes are exactly $AM(D_d^\bullet)$, $AM(P_d^\bullet)$. Therefore the inclusions $AM({P_d}^\bullet)\hookrightarrow\Tot{P_d}^\bullet$, $AM({D_d}^\bullet)\hookrightarrow\Tot{D_d}^\bullet$ induce an isomorphism of the spectral sequences (associated with the above filtration) after the first term. Therefore these inclusions are quasi-isomorphisms.
\end{proof}

\begin{remark}\label{r102} Recall Remark~\ref{r81} that the loops are possible only if $d$ is even. But even in this case if we quotient out ${P_d}^\bullet$ by the graphs with loops, then ${P_d}^\bullet$ is no more smooth only in complexity~1. Indeed, the isomorphism
 $E_0\simeq \oplus_{n\geq 0}K_n^\vee\otimes_{S_n}M({P_d}^n)$  fails to be true only in complexity~1, since there is only one graph
\begin{picture}(20,15)
\put(2,5){\circle*{3}}
\put(18,5){\circle*{3}}
\qbezier(2,5)(10,15)(18,5)
\put(0,-1){$_1$}
\put(16,-1){$_2$}
\end{picture}
  in $M({P_d}^\bullet)$ that can produce a loop by gluing consecutive external vertices.
\end{remark}

\section{$AM({P_d}^\bullet)$ in small complexities}\label{s11}
In this section we present the results of computations of the homology of $AM({P_d}^\bullet)$ for complexities $j=1$, 2, and 3.

\vspace{.3cm}

\centerline{\bf Complexity $j=1$}


\subsubsection*{Odd $d$}
There is only one graph which is not canceled by the orientation relations:

$$
\begin{picture}(30,15)
\put(0,5){\line(2,1){30}}
\end{picture}
\eqno(\numb)\label{eq111}
$$

It defines a rational homotopy of $\overline{Emb}_d$ of dimension $d-3$. Its Hodge degree is 2.

\subsubsection*{Even $d$}
Again we have only one non-trivial graph:
$$
\begin{picture}(30,25)
\put(0,0){\line(2,1){15}}
\put(21,11.45){\circle{15}}
\end{picture}
\eqno(\numb)\label{eq112}
$$
that defines a rational homotopy of dimension $d-3$. Its Hodge degree is 1.

\vspace{.3cm}

\centerline{\bf Complexity $j=2$}


\subsubsection*{Odd $d$}
There are only 2 non-trivial graphs
$$
\includegraphics[width=2cm]{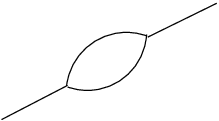},
\eqno(\numb)\label{eq113}
$$
and
$$
\includegraphics[width=1.5cm]{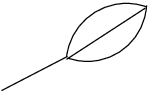}.
\eqno(\numb)\label{eq114}
$$
The first one has the degree $2d-6$ (its Hodge degree is 2), the second one has the degree $2d-5$ (its Hodge degree is 1).

\begin{remark}\label{r111}
The cycles~\eqref{eq111}, \eqref{eq112}, and \eqref{eq114} are exactly those that arise from the rational homotopy  of the factor $\Omega^2S^{d-1}$ in $\overline{Emb}_d=
Emb_d\times\Omega^2S^{d-1}$.
\end{remark}

\subsubsection*{Even $d$}
Recall from Remark~\ref{r102}, that starting from complexity 2 one can consider a quasi-isomorphic complex spanned by the graphs in $AM({D_{even}}^\bullet)$ without loops. In complexity 2 one has only one such graph:
$$
\includegraphics[width=1.5cm]{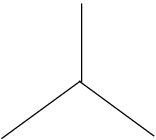}
\eqno(\numb)\label{eq115}
$$
It defines a cycle of dimension $2d-6$. its Hodge degree is 3.

\vspace{.3cm}

\centerline{\bf Complexity $j=3$}


\subsubsection*{Odd $d$}
The case of odd $d$ is harder because of the presence of multiple edges. In the Hodge degrees 4 and 3 all the graphs are trivial modulo the orientation relations. In the Hodge degree 2 one has a complex spanned by 5 graphs:

\vspace{.3cm}
\begin{center}
\psfrag{d}[0][0][1][0]{$\partial$}
\includegraphics[width=7cm]{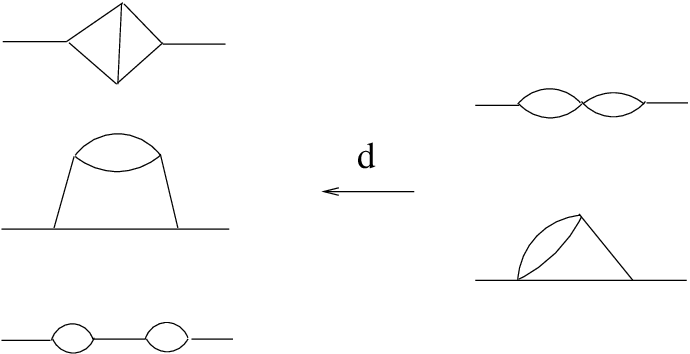}
\end{center}
\vspace{.3cm}

Its only cycle is concentrated in the lowest degree $3d-9$ and is given by any of the 3 graphs:
$$
\includegraphics[width=9cm]{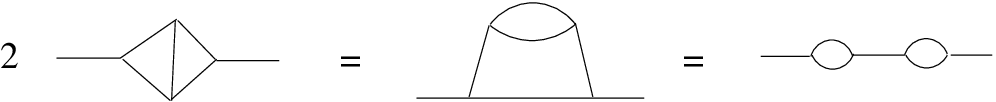}
\eqno(\numb)\label{eq116}
$$
We bring attention of the reader that pairing with the dual graph-complex (whose differential is a sum of contractions of edges) is given by the following rule: if a non-trivial graph in the dual graph complex is not isomorphic to a graph $\Gamma$ in $AM({P_d}^\bullet)$ then the pairing is zero, otherwise it is the order of the group of symmetries of $\Gamma$. In particular a cycle dual to~\eqref{eq116} is given by the sum:

$$
\includegraphics[width=9cm]{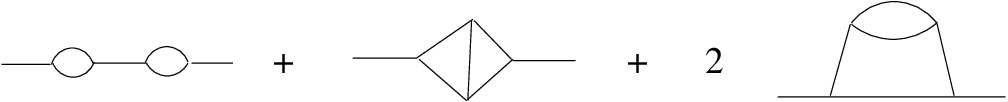}
$$

The order of the group of symmetries for the first graph is 8, for the second and the third ones it is 4. Therefore the pairing between~\eqref{eq116} and the above cycle is~8.

\vspace{.5cm}

In Hodge degree 1,  one has a complex spanned by 9 graphs:

\vspace{.3cm}
\begin{center}
\psfrag{d}[0][0][1][0]{$\partial$}
\includegraphics[width=10cm]{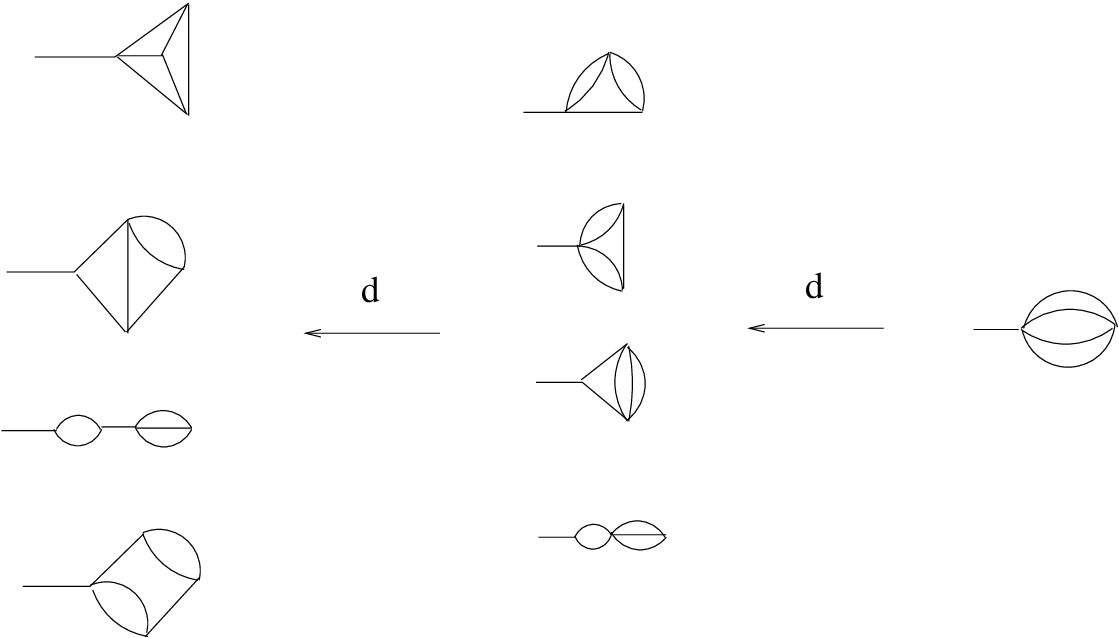}
\end{center}
\vspace{.3cm}

Easy computations show that there is only one non-trivial cycle which lies in the lowest degree $3d-8$. It is given by any of 2 graphs:
$$
\includegraphics[width=5cm]{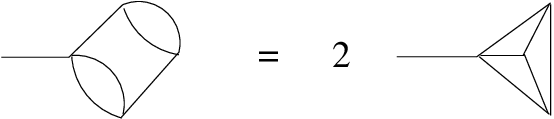}
\eqno(\numb)\label{eq117}
$$
The other 2 graphs in this degree lie in the boundary. The dual cycle is the sum:
$$
\includegraphics[width=5cm]{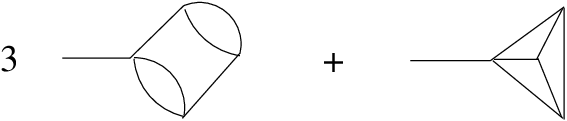}
$$

\subsubsection*{Even $d$}
The case of even $d$ is easier since the graphs are without multiple edges (and also without loops due to Remark~\ref{r102}).

In Hodge degree 4 there are only 2 non-trivial graphs that cancel each other:

\vspace{.3cm}
\begin{center}
\psfrag{d}[0][0][1][0]{$\partial$}
\includegraphics[width=5cm]{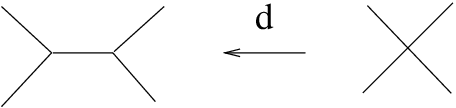}
\end{center}
\vspace{.3cm}

In Hodge degree 3 there is no non-trivial graphs. In Hodge degree 2 there is only 1 graph:

$$
\includegraphics[width=2.5cm]{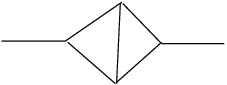}
\eqno(\numb)\label{eq118}
$$
that defines a cycle of degree 3d-9.

In Hodge degree 1  one has only 1 non-trivial graph:
$$
\includegraphics[width=2cm]{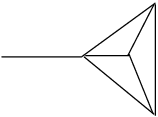}
\eqno(\numb)\label{eq119}
$$
It defines a cycle of degree $3d-8$.

\begin{remark}\label{r112}
In complexity 3 the rational homotopy for both even and odd $d$ are given by the same graphs:
$$
\includegraphics[width=2.5cm]{complexity3_even_hodge2.eps}\quad \text{\raisebox{10pt}[1pt][20pt]{and} }
\quad
\includegraphics[width=1.2cm]{complexity3_even_hodge1.eps}
$$
\end{remark}

\begin{remark}\label{r113}
The computations in this section are consistent with the computations of the Euler characteristics of the Hodge splitting in rational homotopy, see Tables~1 and~3, and also with the previous computation of the rational homotopy of $\overline{Emb}_d$, see~\cite{SinhaSc,T-OS}.
\end{remark}

\section{Homotopy graph-complexes for spaces of knotted planes}\label{s12}

Theorem~\ref{t81} seems unnatural since it forgets the natural \lq\lq linear" topology of the real line $\R^1$. From this point of view Theorem~\ref{t82} seems to be even more strange and unbelievable. But the surprise is that  these complexes $AM({P_d}^\bullet)$ are naturally related to the topology of embedding spaces of higher dimensional affine spaces for which the complexes in question do seem naturally appropriate.

Let $Emb(\R^k,\R^d)$  denote the space of smooth non-singular embeddings $\R^k\hookrightarrow\R^d$ with a fixed linear behavior outside some compact subset of $\R^k$. Similarly let $Imm(\R^k,\R^d)$ denote the space of immersions with the same behavior at infinity, and let $\overline{Emb}(\R^k,\R^d)$ denote the homotopy fiber of the inclusion $Emb(\R^k,\R^d)\hookrightarrow Imm(\R^k,\R^d)$.

Let $\calE_{k,d}$ denote a graph-complex whose definition is the same as the one given for $AM({P_d}^\bullet)$ in Section~\ref{s8} with the only exception that we assign degree $-k$ instead of $-1$ to the external vertices of the uni-$\geq 3$-valent graphs. In particular $\calE_{1,d}=
AM({P_d}^\bullet)$. Up to a shift of gradings $\calE_{k,d}$ depends on the parities of $k$, and $d$ only.


\begin{conj}\label{c121}
The rational homotopy of $\overline{Emb}(\R^k,\R^d)$, $d\geq 2k+2$, is naturally isomorphic to the homology of $\calE_{k,d}$:
$$
\pi_*(\overline{Emb}(\R^k,\R^d))\otimes\Q\simeq H_*(\calE_{k,d}).
$$
\end{conj}

This conjecture would imply that up to a shift of gradings the rational homotopy and homology of  $\overline{Emb}(\R^k,\R^d)$, $d\geq 2k+2$, depends on the parities of $k$ and $d$ only. We stress again the fact that surprisingly this biperiodicity starts already from $k=1$!

The above conjecture is equivalent to a collapse at $E^2$ of some spectral sequence arising naturally from the Goodwillie-Weiss embedding calculus and computing the rational homology of $\overline{Emb}(\R^k,\R^d)$. To be precise $H_*(\calE_{k,d})$ is exactly the primitive part of $E^2$. Recall that besides the long knots (Theorem~\ref{t21}) this rational homology collapse result holds also for the spaces of embeddings modulo immersions 
$\overline{Emb}(M,\R^d)$ of any compact manifold into an affine space of a sufficiently high dimension $d$~\cite{ALV:foHE:X}. Unfortunately neither the method for long knots~\cite{LTV}, nor the one of~\cite{ALV:foHE:X} can be applied to the case of knotted planes.

\part{Euler characteristics of the Hodge splitting}\label{p3}

\section{Generating function of Euler characteristics}\label{s13}
Recall that the rational homology of $\overline{Emb}_d$  is computed by the complex
$$
\Tot {A_d}^\bullet=(\oplus_{n\geq 0}s^{-n}N{A_d}^n,\partial)=(\oplus_{n\geq 0}s^{-n}H_*^{Norm}({C_d}^n,\Q),\partial).
$$
The homology of ${C_d}^n$ is concentrated in the gradings that are multiples of $(d-1)$:
$$
*=j\cdot (d-1),\qquad 0\leq j\leq n-1.
$$
For the normalized part one has the restrictions:
$$
\frac n2\leq j\leq n-1.
$$
The lower bound happens for the so called \lq\lq chord diagrams", see~\cite{T-HLN}. The differential $\partial$ preserves this grading $j$ that will be called {\it complexity}. Since $\partial$ also preserves the Hodge degree, one has a double splitting:
$$
\Tot {A_d}^\bullet=\oplus_{i,j}\Tot^{(i,j)}{A_d}^\bullet=\oplus_{i,j}(\oplus_{n}s^{-n}e_i(H_{j(d-1)}^{Norm}({C_d}^n,\Q)),\partial).
$$
Here $e_i$ is the $i$-th projector of the Hodge decomposition, see Section~\ref{s3}, $s^{-n}$ denotes the $n$-fold desuspension.

Denote by $H_*^{(i,j)}(\overline{Emb}_d,\Q)=H_*(\Tot^{(i,j)}{A_d}^\bullet)$. Let $\chi_{i,j}$ denote the Euler characteristics in bigrading $(i,j)$:
$$
\chi_{i,j}=\sum_{n=j(d-3)}^{j(d-2)}(-1)^n \rank( H_n^{(i,j)}(Emb_d,\Q)),
$$
and let $F_d(x,u)$ be the corresponding generating function
$$
F_d(x,u)=\sum_{i,j}\chi_{i,j}x^iu^j.
$$
The variable $x$ is responsible for the Hodge degree and the variable $u$ is responsible for the complexity. Since up to an (even) shift of gradings the complexes $\Tot {A_d}^\bullet$ depend on the parity of $d$ only, one has that $F_d(x,u)$ are the same for $d$'s of the same parity. We will denote the corresponding generating functions by $F_{odd}(x,u)$, and $F_{even}(x,u)$.

Let $E_\ell(y)$ be the polynomial $\frac 1\ell\sum_{d|\ell}\mu(d)y^{\ell/d}$ (where $\mu(-)$ is the standard M\"obius function), and let $\Gamma(y)$ be the usual Gamma function $(y-1)!$.

\begin{thm}\label{t131}
(a)
$$
F_{odd}(x,u)=\prod_{\ell\geq 1}\frac{\Gamma(E_\ell(\frac 1u)-E_\ell(x))}{(\ell u^\ell)^{E_\ell(x)}\Gamma(E_\ell(\frac 1u))},
\eqno(\numb)\label{eq131}
$$
where each factor in the product is understood as the asymptotic expansion (see Definition~\ref{d141}) of the underlying function when $u\to +0$, and $x$ is considered as a fixed parameter.

\vspace{0.5cm}

(b)
$$
F_{even}(x,u)=\prod_{\ell\geq 1}\frac{\Gamma(-E_\ell(\frac 1u)-E_\ell(x))}{(-\ell u^\ell)^{E_\ell(x)}\Gamma(-E_\ell(\frac 1u))},
\eqno(\numb)\label{eq132}
$$
where each factor in the product is understood as the asymptotic expansion of the underlying function when $u$ is  complex  and $u^\ell\to -0$. Again $x$ is considered as a fixed parameter.
\end{thm}

In the next section we will give a better understanding of this formula. We want to warn the reader that the series corresponding to each factor can be divergent depending on $x$. We mention that both $F_{odd}(x,u)$, $F_{even}(x,u)$ have the form $\sum_{j=0}^{+\infty}P_j(x)u^j$, for $P_j(x)$, $j=0,1,2,\ldots$, being a sequence of polynomials. For small complexities $j$ one has
$$
F_{odd}(x,u)=1+x^2u+(x^4+x^2-x)u^2+(x^6+x^4-x^3+x^2-x)u^3+(x^8+x^6-x^5+3x^4-3x^3+x^2-x)u^4+\ldots,
$$
$$
F_{even}(x,u)=1-xu+x^3u^2+(-x^4-x^2+x)u^3+(x^6+x^3-x^2)u^4+\ldots,
$$
see Tables 2 and 4.

\section{Understanding formulae~\eqref{eq131}-\eqref{eq132}}\label{s14}

\subsection{Looking at the first factor}\label{s141}

\begin{definition}\label{d141}
A function $f(u)$ is said to have an {\it asymptotic expansion} $\sum_{j=0}^{+\infty}a_ju^j$ when $u\to +0$ if for any $n\geq 0$ one has
$$
f(u)=\sum_{j=0}^na_ju^j + o(u^n),
$$
when $u\to +0$.
\end{definition}

Notice that the series $\sum_{j=0}^{+\infty}a_ju^j$ is considered as a formal series which is not necessary convergent, or even if it is convergent it is not supposed to converge to $F(u)$.

\vspace{.5cm}

Now consider the first factor
$$
\frac{\Gamma(\frac 1u-x)}{u^x\Gamma(\frac 1u)}
\eqno(\numb)\label{eq141}
$$
of the product~\eqref{eq131}. Variable $x$ is a parameter. Let $x$ be a positive integer $n$. Applying the identity $\Gamma(z+1)=z\Gamma(z)$, we obtain:
$$
\frac{\Gamma(\frac 1u - n)}{u^n\Gamma(\frac 1u)}=\frac 1{(1-u)(1-2u)\ldots (1-nu)}=\sum_j\gamma_j(n)u^j,
\eqno(\numb)\label{eq142}
$$
where $\gamma_j(n)=\sum_{1\leq i_1\leq\ldots\leq i_j\leq n}i_1i_2\ldots i_j$. It is easy to see that $\gamma_j(n)$ is a polynomial function of $n$. Now let $x$ be a negative integer $-n$. Similarly we obtain
$$
\frac{\Gamma(\frac 1u +n)}{u^{-n}\Gamma(\frac 1u)}=(1+u)(1+2u)\ldots (1+(n-1)u)=\sum_{j=0}^{+\infty}\tilde \gamma_j(n) u^j,
\eqno(\numb)\label{eq143}
$$
where $\tilde\gamma_j(n)=\sum_{1\leq i_1<\ldots<i_j\leq n-1}i_1\ldots i_j$ again are  polynomials of $n$.

\begin{lemma}\label{l142}
The polynomials $\gamma_j(x)$, $\tilde\gamma_j(x)$, $j=0,1,2,\ldots$, are related to each other:
$$
\tilde \gamma_j(x)=\gamma_j(-x).
$$
Moreover the function~\eqref{eq141} for any real parameter $x$ has the  asymptotic expansion
$$
\sum_{j=0}^{+\infty}\gamma_j(x)u^j,
$$
when $u\to +0$.
\end{lemma}

\begin{proof}
It is sufficient to show that the function~\eqref{eq141} has an asymptotic expansion of the form $\sum_{j=0}^{+\infty}\lambda_j(x)u^j$ when $u\to +0$, where $\lambda_j(x)$ are polynomials of $x$. Since a polynomial is uniquely determined by its values on the set of positive (resp. negative) integers, the result will follow.

By the generalized Stirling formula~\cite[p.~24]{Artin}
$$
\Gamma(z)=\sqrt{2\pi}z^{z-\frac 12}e^{-z+\nu(z)},
\eqno(\numb)\label{eq144}
$$
where $\nu(z)$ has the asymptotic expansion at $z\to+\infty$:
$$
\sum_{j=1}^{+\infty}h_j\frac 1{z^{2j-1}}
\eqno(\numb)\label{eq145},
$$
where $h_j=2(-1)^{j-1}(2j-2)!\sum_{i=1}^{+\infty}\frac 1{(2\pi i)^{2j-2}}=(-1)^{j-1}\frac{B_j}{2j(2j-1)}$, where $B_j$ are the Bernoulli numbers:
 $$
 B_1=1/6,\quad B_2=1/30,\quad B_3=1/42,\quad B_4=1/30,\quad B_5=5/66,\quad \text{etc.}
 $$
 The formal series~\eqref{eq145} is divergent for any complex $z$, since the coefficients $h_j$ has a faster then exponential growth.

Applying~\eqref{eq144}, we get
\begin{multline}
\frac{\Gamma(\frac 1u-x)}{u^x\Gamma(\frac 1u)}=e^x\cdot (1-ux)^\frac 1u\cdot (1-xu)^{-x-\frac 12}\cdot e^{\nu(\frac 1u-x)-\nu(\frac 1u)}=\\
e^{\left(\frac{x^2u}2-\frac{x^3u^2}3+\frac{x^4u^3}4-\ldots\right)}\cdot (1-xu)^{-x-\frac 12}\cdot e^{\nu(\frac 1u-x)-\nu(\frac 1u)}.
\label{eq145'}
\end{multline}
Both the first and the second factors have the asymptotic expansion of the form $\sum_jf_j(x)u^j$ with $f_j(x)$ being some polynomials of $x$. In such situation we will say that a function (of two variables $x$, $u$) has a  {\it polynomial asymptotic expansion}. The asymptotic expansion~\eqref{eq145} of $\nu(z)$ implies the following asymptotic expansion of $\nu(\frac 1u-x)-\nu(\frac 1u)$:
$$
\sum_{j=1}^{+\infty}h_j\frac 1{(\frac 1u-x)^{2j-1}}-\sum_{j=1}^{+\infty}h_j\frac 1{\left(\frac 1u\right)^{2j-1}}=
\sum_{j=1}^{+\infty}h_j\frac{u^{2j-1}}{(1-xu)^{2j-1}}-\sum_{j=1}^{+\infty}h_ju^{2j-1},
$$
which can be rewritten as $\sum_{j=1}^{+\infty}\rho_j(x)u^j$ for some polynomials $\rho_j(x)$. As a consequence $e^{\nu(\frac 1u-x)-\nu(\frac 1u)}$ has also a polynomial asymptotic expansion when $u\to +0$.
\end{proof}

\begin{notation}\label{n143}
We will denote by
$$
\Gamma(x,u)=\sum_{j=0}^{+\infty}\gamma_j(x)u^j
\eqno(\numb)\label{eq146}
$$
the asymptotic expansion  of $\frac{\Gamma(\frac 1u-x)}{u^x\Gamma(\frac 1u)}$ when $u\to +0$.
\end{notation}

Our next goal is to show that the series $\Gamma(x,u)$ does not have nice convergency properties when $x$ is not an integer.

\begin{proposition}\label{P:divergence}
For any $x\in\BC\setminus\BZ$ the series $\Gamma(x,u)$ has zero radius of convergence in $u$.
\end{proposition}

This result is classical: it is mentioned in~\cite{Frame}. The argument below is a modification of an argument given in~\cite{Frame}, which we give  for a completeness of exposition. Notice that the main reason for it is that the function $\frac{\Gamma(\frac 1u-x)}{u^x\Gamma(\frac 1u)}$ is not holomorphic in a neighborhood of 0 if $x$ is not an integer. Indeed, this function has poles at $u=\frac 1{x-n},$ $n\in\BN$, concentrating at $u=0$, and moreover its Riemann surface has a ramification at $u=0$ due to the factor $u^x$ in the denominator.

\begin{proof}
Consider the series
$$
\ln\Gamma(x,u)=\sum_{k=1}^{+\infty} \frac{S_k(x)}ku^k.
$$
It has zero radius of convergence if and only if $\Gamma(x,u)$ has zero radius of convergence. By Lemma~\ref{l142}, $S_k(x)$ are polynomials in $x$. We have to show that
$$
\varlimsup_{k\to +\infty}\sqrt[k]{S_k(x)}=+\infty
\eqno(\numb)\label{eq:radius}
$$
for any $x\in\BC/\BZ$. It follows from~\eqref{eq142} that for a positive integer $x=n$, one has $S_k(n)=\sum_{i=0}^ni^k$ for $k\geq 1$, which in its turn implies:
$$
S(x,t)=\sum_{k=0}^{+\infty}\frac{S_k(x)t^k}{k!}=\frac{e^{(x+1)t}-1}{e^t-1}.
$$
When $x$ is not integer the above function has poles at $t=\pm 2\pi i$, which means that the radius of convergence of $S(x,t)$ in $t$ is $2\pi$. Therefore
$$
\varlimsup_{k\to +\infty}\sqrt[k]{\frac{S_k(x)}{k!}}=\frac 1{2\pi}.
$$
And equation~\eqref{eq:radius} is proved.
\end{proof}

\subsection{Understanding all factors of \eqref{eq131}-\eqref{eq132}}\label{ss142}
It turns out that all the factors of~\eqref{eq131}-\eqref{eq132} can be easily expressed via the first one considered in the previous subsection.
Consider an arbitrary factor of~\eqref{eq131}. We are interested in its asymptotic expansion when $u\to +0$. We can rewrite it as follows:
$$
\frac{\Gamma(E_\ell(\frac 1u)-E_\ell(x))}{(\ell u^\ell)^{E_\ell(x)}\Gamma(E_\ell(\frac 1u))}=
\frac{\Gamma(E_\ell(\frac 1u)-E_\ell(x))}{(E_\ell(\frac 1u))^{-E_\ell(x)}\Gamma(E_\ell(\frac 1u))} \cdot
\frac 1 {\left(\ell u^\ell E_\ell\left(\frac 1u\right)\right)^{E_\ell(x)}}
\eqno(\numb)\label{eq1414}
$$
The first factor has the asymptotic expansion $\Gamma\left(E_\ell(x),\frac 1 {E_\ell(1/u)}\right)$ (see Notation~\ref{n143}). The second factor is holomorphic in a neighborhood of $u=0$ (for all complex $x$).

Let $F_\ell(y)$ denote the polynomial
$$
F_\ell(u)=\ell u^\ell E_\ell\left(1/u\right)= \sum_{d|\ell}\mu(d)u^{\ell-\ell/d}=1-u^{\ell-\ell/p_1}-u^{\ell-\ell/p_2}+u^{\ell-\ell/{p_1p_2}}+\ldots,
$$
where $p_1$, $p_2$ are first two prime factors of $\ell$.

We can rewrite~\eqref{eq131} and~\eqref{eq132} as follows:

\begin{lemma}\label{l147}
$$
\text{(i) }\quad
F_{odd}(x,u)=\frac{\prod_{\ell\geq 1}\Gamma\left(E_\ell(x),\frac 1{E_\ell\left(\frac 1u\right)}\right)}{\prod_{\ell\geq 1}\left[F_\ell(u)\right]^
{E_\ell(x)}}=\frac{\prod_{\ell\geq 1}\Gamma\left(E_\ell(x),\frac{\ell u^\ell}{F_\ell\left(u\right)}\right)}{\prod_{\ell\geq 1}\left[F_\ell(u)\right]^
{E_\ell(x)}},
$$

$$
\text{(ii) }\quad
F_{even}(x,u)=\frac{\prod_{\ell\geq 1}\Gamma\left(E_\ell(x),-\frac 1{E_\ell\left(\frac 1u\right)}\right)}{\prod_{\ell\geq 1}\left[F_\ell(u)\right]^
{E_\ell(x)}}=\frac{\prod_{\ell\geq 1}\Gamma\left(E_\ell(x),-\frac{\ell u^\ell}{F_\ell\left(u\right)}\right)}{\prod_{\ell\geq 1}\left[F_\ell(u)\right]^
{E_\ell(x)}}.
$$
\end{lemma}

\begin{proof}
(i) follows from~\eqref{eq1414}. (ii) is analogous --- we only mention that we use that $u^\ell\to -0$ implies $-\frac 1{E_\ell\left(\frac 1u\right)}\to +0$.
\end{proof}

\section{Proof of Theorem~\ref{t131}}\label{s15}

The proof of Theorem~\ref{t131} is based on the character computations of the symmetric group action on the homology of configuration spaces~\cite{Lehrer}, and on the components of the Hodge decomposition~\cite{Hanlon}.

\subsection{Character computations for symmetric sequences}\label{ss151}
The main field, which we denote by $\BK$, is as usual of characteristic zero. In this section we introduce some standard notation which will be used in the sequel.

For each permutation $\sigma\in S_n$ define $Z(\sigma)$, the {\it cycle indicator} of $\sigma$, by
$$
Z(\sigma)=\prod_\ell a_\ell^{j_\ell(\sigma)},
$$
where $j_\ell(\sigma)$  is the number of $\ell$-cycles of $\sigma$ and where $a_1$, $a_2$, $a_3$, $\ldots$ is an infinite family of commuting variables.

\begin{remark}\label{rk:cycle_ind}
Notice that $Z(\sigma')=Z(\sigma)$ for $\sigma'\in S_{n'}$, $\sigma\in S_n$ if and only if $n'=n$ and moreover $\sigma'$ is conjugate to $\sigma$. For $\sigma$ with $Z(\sigma)=\prod_\ell a_\ell^{j_\ell(\sigma)}$, there are exactly $\frac{n!}{\prod_\ell(\ell^{j_\ell}j_\ell !)}$ elements $\sigma'$ conjugate to $\sigma$.
\end{remark}

Let $\rho^V\colon S_n\to GL(V)$ be a representation of $S_n$. Define $Z_V(a_1,a_2,\ldots)$ the {\it cycle index} of $V$, by
$$
Z_V(a_1,a_2,\ldots)=\frac 1{n!}\sum_{\sigma\in S_n}\tr \rho^V(\sigma)\cdot Z(\sigma).
$$
Similarly for a symmetric sequence $W=\{W(n),n\geq 0\}$ --- sequence of $S_n$-modules $W(n)$, $n=0,1,2,\ldots$, we define its {\it cycle index sum} $Z_W$ by
$$
Z_W(a_1,a_2,\ldots)=\sum_{n=0}^{+\infty}Z_{W(n)}(a_1,a_2,\ldots).
$$

\begin{definition}\label{def:tensor}
The external tensor product of two symmetric sequences $V$, and $W$ is  a symmetric sequence $V\hat\otimes W$ given by
$$
V\hat\otimes W(n):=\bigoplus_{i=0}^n\mathrm{Ind}^{S_n}_{S_i\times S_{n-i}}V(i)\otimes W(n-i)=\bigoplus_{i=0}^n\bigl(V(i)\otimes W(n-i)\bigr)\otimes_{S_i\times S_{n-i}}\BK[S_n].
$$
\end{definition}

\begin{prop}\label{prop:tensor}
For any finite symmetric sequences (finite in each component) $V$ and $W$, one has:
$$
Z_{V\hat\otimes W}(a_1,a_2,\ldots)=Z_V(a_1,a_2,\ldots)\cdot Z_W(a_1,a_2,\ldots).
$$
\end{prop}
\begin{proof}
This result is standard, the idea is that if $N$ is a subgroup of $S_n$, and $V$ is a representation of $N$, then
$$
\frac 1{|N|}\sum_{\sigma\in N} \tr\rho^V(\sigma)Z(\sigma)=\frac 1{|S_n|}\sum_{\sigma\in S_n}\tr\rho^{\mathrm{Ind}^{S_n}_NV}(\sigma)Z(\sigma),
$$
see~\cite{Feit}. The right hand-side is exactly $Z_{\mathrm{Ind}^{S_n}_N V}(a_1,a_2,\ldots)$.
\end{proof}

Another important property is given by the following lemma:

\begin{lemma}\label{l:hom}
Let $V$ and $W$ be two $S_n$-modules, then
$$
\dim { Hom}(V,W)^{S_n}=\bigl(Z_V(a_\ell\leftarrow\partial/{\partial a_\ell},\,\ell\in\BN)\,\,Z_W(a_\ell\leftarrow \ell a_\ell,\,\ell\in\BN)\bigr)\Bigr|_{\substack{a_\ell=0,\\ \ell\in\BN}}.
$$
\end{lemma}

In the above formula $Z_V(a_\ell\leftarrow\partial/{\partial a_\ell},\,\ell\in\BN)$ is a differential operator, which is applied to
$Z_W(a_\ell\leftarrow \ell a_\ell,\,\ell\in\BN)$. And at the end we take all the variables $a_\ell,$ $\ell\in\BN$, to be zero.\footnote{It is easy to see that by a linear change of variables the right-hand side of the above formula is equal to
$$
\bigl(Z_V(a_\ell\leftarrow\ell\partial/{\partial a_\ell},\,\ell\in\BN)\,\,Z_W(a_1,a_2,\ldots)\bigr)\Bigl|_{\substack{a_\ell=0,\\ \ell\in\BN}},
$$
or more symmetrically to
$$
\left(Z_V(a_\ell\leftarrow\sqrt{\ell}\partial/{\partial a_\ell},\,\ell\in\BN)\,\,Z_W(a_\ell\leftarrow \sqrt{\ell} a_\ell,\,\ell\in\BN)\right)\Bigr|_{\substack{a_\ell=0,\\ \ell\in\BN}}.
$$
}

\begin{proof}
This follows from the formula
$$
\dim\,{Hom}(V,W)^G=\frac 1{|G|}\sum_{g\in G} \tr\rho^V(\sigma)\cdot \tr\rho^W(\sigma^{-1}),
$$
that holds for any finite group $G$ and any its finite-dimensional representations $V$, $W$.

Since in the symmetric group any element is conjugate to its inverse,  one has
$$
\dim\,{Hom}(V,W)^{S_n}=\frac 1{n!} \sum_{g\in S_n} tr\rho^V(\sigma)\cdot tr\rho^W(\sigma).
$$
The rest follows by direct computation from Remark~\ref{rk:cycle_ind}.
\end{proof}

In case $V=\oplus_iV_i$, $W=\oplus_iW_i$ are graded $S_n$-modules, and $\dim{ Hom}(V,W)^{S_n}$ is the graded dimension:
$$
\dim{Hom}(V,W)^{S_n}=\sum_{i,j\in\BZ}\dim {Hom}(V_i,W_j)^{S_n}z^{j-i},
$$
and $Z_V$, $Z_W$ are {\it graded} cycle indices:
\begin{align*}
Z_V(z;a_1,a_2,\ldots)=&\sum_{i\in \BZ}Z_{V_i}(a_1,a_2,\ldots)z^i,\\
Z_W(z;a_1,a_2,\ldots)=&\sum_{i\in \BZ}Z_{W_i}(a_1,a_2,\ldots)z^i.
\end{align*}
Then
$$
\dim\,{ Hom}(V,W)^{S_n}=Z_V(1/z;a_\ell\leftarrow\partial/{\partial a_\ell},\,\ell\in\BN)\,\,Z_W(z;a_\ell\leftarrow \ell a_\ell,\,\ell\in\BN)\Bigl|_{\substack{a_\ell=0,\\ \ell\in\BN}}.
$$

\begin{cor}\label{c:hom} Let  $V=\{V(n),n\geq 0\}$, $W=\{W(n),n\geq 0\}$ be a pair of symmetric sequences of graded $S_n$-modules. Then
\begin{multline}
\dim Hom(V,W)= \dim \left(\oplus_nHom(V(n),W(n))^{S_n}\right) =\\
=Z_V(1/z;a_\ell\leftarrow\partial/{\partial a_\ell},\,\ell\in\BN)\,\,Z_W(z;a_\ell\leftarrow \ell a_\ell,\,\ell\in\BN)\Bigl|_{\substack{a_\ell=0,\\ \ell\in\BN}}.
\end{multline}
\end{cor}

Later on we will also consider bigraded symmetric sequences. Similarly in our computations we will add one more variable $x$ or $u$ responsible for the second grading.

\subsection{Symmetric sequences ${A_d}^\bullet$, $N{A_d}^\bullet$}\label{ss152}

The symmetric group action on the homology of configuration spaces $C_d(n)=F(n,\R^d)$ is well studied~\cite{CT:RCCS,Lehrer,LO:ASGCH}.

\begin{prop}\label{p:Sn_conf_spaces}
The graded cycle index sum for the symmetric sequence
$$
{A_d}^\bullet=\left\{{A_d}^n|\, n\geq 0\right\}=\left\{H_*(C_d(n),\BK)|\, n\geq 0\right\}
$$
is given by the following formula:
$$
Z_{{A_d}^\bullet}(z;a_1,a_2,\ldots)=\prod_{\ell=1}^{+\infty}\left(1+(-1)^{d}(-z)^{(d-1)\ell}a_\ell\right)^{(-1)^{d}E_\ell\left(\frac 1{(-z)^{d-1}}\right)},
$$
where $E_\ell(y)=\frac 1\ell\sum_{d|\ell}\mu(d)y^{\frac d\ell}.$
\end{prop}

\begin{proof} It is an easy consequence of~\cite[Theorem~B]{Lehrer} and Remark~\ref{rk:cycle_ind}.
\end{proof}

We failed to find this formula in the literature, however one can find similar formulae in the study of $S_n$-modules closely related to the homology of configuration spaces~\cite{CHR:PEO,Hanlon}.

\vspace{.2cm}

Let us add another variable $u$ that will be responsible for the complexity, which is the homology degree divided by $(d-1)$. The $u$-degree is the $z$-degree divided by $(d-1)$:
$$
Z_{{A_d}^\bullet}(z,u;a_1,a_2,\ldots)=\prod_{\ell=1}^{+\infty}\left(1+(-1)^{d}((-z)^{(d-1)}u)^\ell a_\ell\right)^{(-1)^{d}E_\ell\left(\frac 1{(-z)^{d-1}u}\right)}.
$$

Consider the symmetric sequence $N{A_d}^\bullet=\left\{N{A_d}^n|\, n\geq 0\right\}=\{H_*^{Norm}(C_d(n),\BK)|\, n\geq 0\}$. It is easy to see that
\begin{equation}
{A_d}^n\simeq\bigoplus_{i=0}^n\mathrm{Ind}^{S_n}_{S_i\times S_{n-i}}{NA_d}^i,
\label{eq:norm}
\end{equation}
Where the $S_i$-module $N{A_d}^i$ is considered as an $S_i\times S_{n-i}$-module being acted on trivially by the second factor $S_{n-i}$.

According to Definition~\ref{def:tensor} formula~\eqref{eq:norm} means that
$$
{A_d}^\bullet\simeq N{A_d}^\bullet\hat \otimes I,
$$
where $I=\{I(n),n\geq 0\}$ is a sequence of trivial 1-dimensional representations. Using Remark~\ref{rk:cycle_ind} it is easy to check that
$$
Z_I(a_1,a_2,\ldots)=\prod_{\ell\geq 1}e^\frac {a_\ell}\ell.
$$
It follows from Proposition~\ref{prop:tensor}
\begin{equation}
Z_{{NA_d}^\bullet}(z,u;a_1,a_2,\ldots)=\prod_{\ell=1}^{+\infty}e^{-\frac {a_\ell}\ell}\left(1+(-1)^{d}((-z)^{(d-1)}u)^\ell a_\ell\right)^{(-1)^{d}E_\ell\left(\frac 1{(-z)^{d-1}u}\right)}.
\label{eq:h_norm}
\end{equation}

\subsection{Symmetric sequence of the Hodge decomposition}\label{ss153}
To recall in the $n$-th component the Hodge decomposition is described by means of the projectors $e_n^{(i)}\in\BK[S_n],$ $i=1\ldots n$ ($i=0$ if $n=0$). Consider the symmetric sequence $\chi(-)=\{\chi(n)|\, n\geq 0\}$:
$$
\chi(n)=\oplus_ie_n^{(i)}\cdot\BK[S_n]
\eqno(\numb)\label{eq:sym_seq_Hodge}
$$
of graded (by $i$) $S_n$-modules. It was shown by Hanlon~\cite[equation~(6.1)]{Hanlon} that the graded cycle index sum of $\chi(-)$ is given by the following formula:
$$
Z_{\chi(-)}=\prod_\ell(1+(-1)^\ell a_\ell)^{-E_\ell(x)},
\eqno(\numb)\label{eq:index_sum_hodge}
$$
where the variable $x$  is responsible for the Hodge degree $i$, and $E_\ell(x)=\frac 1\ell\sum_{d|\ell}\mu(d)x^{\ell/d}$.

\subsection{Proof of Theorem~\ref{t131}}\label{s:proof}
First notice that Corollary~\ref{c:hom} together with the formulae~(\ref{eq:h_norm}), and~(\ref{eq:index_sum_hodge})  produce the following formula for the generating function $\Phi(x,u,z)$ of the dimensions of the complex computing $H_*(\overline{Emb}_d,\Q)$, $d\geq 4$:
\begin{multline}
\Phi(x,u,z)=\\
\scriptstyle
\left.\left(\prod_{\ell=1}^{+\infty}\left(1+(-1/z)^{\ell}\partial/\partial a_\ell\right)^{-E_\ell\left(x\right)}
\prod_{\ell=1}^{+\infty}e^{-a_\ell}\left(1+(-1)^{d}\ell((-z)^{(d-1)}u)^\ell a_\ell\right)^{(-1)^{d}E_\ell\left(\frac 1{(-z)^{d-1}u}\right)}\right)\right|_{\substack{a_\ell=0\\ \ell\in\BN}}=\\
\scriptstyle
\prod_{\ell=1}^{+\infty}\left.\left(\left(1+(-1/z)^{\ell}\partial/\partial a_\ell \right)^{-E_\ell\left(x\right)}
e^{-a_\ell}\left(1+(-1)^{d}\ell((-z)^{(d-1)}u)^\ell a_\ell\right)^{(-1)^{d}E_\ell\left(\frac 1{(-z)^{d-1}u}\right)}\right)\right|_{a_\ell=0}.
\end{multline}

Since $F(x,u)=\Phi(x,u,-1)$ we get
$$
F(x,u)=
\prod_{\ell=1}^{+\infty}\left.\left(\left(1+\partial/\partial a \right)^{-E_\ell(x)}
e^{-a}\left(1+(-1)^{d}\ell u^\ell a\right)^{(-1)^{d}E_\ell\left(\frac 1{u}\right)}\right)\right|_{a=0}.
$$
Notice that in the above formula we replaced $a_\ell$ by $a$. We could do so because each factor uses only  one variable $a_\ell$ which is anyway taken to be zero.

Theorem~\ref{t131} follows immediately from the following proposition:

\begin{prop}\label{p:factor}
\begin{multline}
\left.\left(\left(1+\partial/\partial a \right)^{-E_\ell(x)}
e^{-a}\left(1+(-1)^{d}\ell u^\ell a\right)^{(-1)^{d}E_\ell\left(\frac 1{u}\right)}\right)\right|_{a=0}=\\
=\frac{\Gamma((-1)^{d-1}E_\ell(\frac 1u)-E_\ell(x))}{\bigl( (-1)^{d-1} \ell u^\ell\bigr)^{E_\ell(x)}\Gamma((-1)^{d-1}E_\ell(\frac 1u))},
\end{multline}
where each factor of the right-hand side is understood as its formal asymptotic behavior when $(-1)^{d-1}u^\ell\to +0$.
\end{prop}

\begin{proof}
For simplicity consider the case of odd $d$, and $\ell=1$. Other cases are absolutely analogous. In this situation the left-hand side becomes
$$
\left.\left(\left(1+\partial/\partial a \right)^{-x}
e^{-a}\left(1- u a\right)^{-\frac 1u}\right)\right|_{a=0},
\eqno(\numb)\label{eq15*1}
$$
the right-hand side is
$$
\frac{\Gamma\left(\frac 1u-x\right)}{u^x\Gamma\left(\frac 1u\right)}=\Gamma(x,u).
\eqno(\numb)\label{eq15*2}
$$
Notice that both~\eqref{eq15*1}, and~\eqref{eq15*2} have the form $\sum_jf_j(x)u^j$, where $f_j(x)$ are some polynomials. Indeed, \eqref{eq15*1} has this form because the normalized complex $\Tot {A_d}^\bullet$ is finite in each complexity $j$, the expression~\eqref{eq15*2} has this form by Lemma~\ref{l142}. We conclude that it suffices to check the equality when $x$ is any negative integer number: $x=-n$. In this case
$$
\Gamma(-n,u)=(1+u)(1+2u)\ldots(1+(n-1)u).
$$
One can also prove by induction over $n$ that
$$
\left(\left(1+\partial/\partial a \right)^{n}
e^{-a}\left(1- u a\right)^{-\frac 1u}\right)=\Gamma(-n,u)\cdot e^{-a}(1-ua)^{-\frac 1u-n}.
$$
Taking $a=0$ implies the result.
\end{proof}

\subsection{Alternative proof of Proposition~\ref{p:factor}}\label{ss155}
There is another proof which makes more natural the appearance of the Gamma function. The proof is more technical, so we give only its idea. It uses the following lemma:

\begin{lemma}\label{l:dif_operator}
Let $X$ be a complex number with a positive real part, and $f(a)$ be any polynomial, then
$$
\left.(1+\partial/\partial a)^{-X}f(a)\right|_{a=0}=\frac 1{\Gamma(X)}\int_{-\infty}^0(-a)^{X-1}e^af(a)\,da.\footnote{This formula was obtained by using Fourier transform which permitted to rewrite the differential operator $(1+\partial/\partial a)^{-X}$ as an integral operator.}
$$

\end{lemma}
To prove the lemma we notice that for $f(a)=a^n$ both sides are equal to $(-1)^nX(X+1)(X+2)\ldots (X+n-1)$.

\vspace{.5cm}

Now we apply the above lemma for $X=E_\ell(x)$ and taking instead of $f(a)$ the generating function of a sequence of polynomials
$\sum_jf_j(a)u^j=e^{-a}\left(1-(-1)^{d-1}\ell u^\ell a\right)^{-(-1)^{d-1}E_\ell\left(\frac 1{u}\right)}$:
$$
\left.\left(\left(1+\partial/\partial a \right)^{-E_\ell(x)}
e^{-a}\left(1-(-1)^{d-1}\ell u^\ell a\right)^{-(-1)^{d-1}E_\ell\left(\frac 1{u}\right)}\right)\right|_{a=0}=
$$
$$
\frac 1{\Gamma(E_\ell(x))}\int_{-\infty}^0(-a)^{E_\ell(x)-1}\left(1-(-1)^{d-1}\ell u^\ell a\right)^{-(-1)^{d-1}E_\ell\left(\frac 1{u}\right)}\, da.
\eqno(\numb)\label{eq15:integral}
$$
Assuming that $u$ is a small complex number such that $(-1)^{d-1}u^\ell$ is real positive we make a  change of variables
$$
a=-\frac{(-1)^{d-1}}{\ell u^\ell}\cdot\frac t{1-t},
$$
that gives that the above expression is equal to the following:
$$
\frac 1{\left((-1)^{d-1}\ell u^\ell\right)^{E_\ell(x)}\Gamma(E_\ell(x))}\int_0^1t^{E_\ell(x)-1}
(1-t)^{(-1)^{d-1}E_\ell\left(\frac 1{u}\right)-E_\ell(x)-1}\, dt=
$$
$$
=\frac 1{\left((-1)^{d-1}\ell u^\ell\right)^{E_\ell(x)}\Gamma(E_\ell(x))}\times\frac{\Gamma(E_\ell(x))\cdot
\Gamma((-1)^{d-1}E_\ell(\frac 1u)-E_\ell(x))}
{\Gamma((-1)^{d-1}E_\ell(\frac 1u))}=
$$
$$
=
\frac{\Gamma((-1)^{d-1}E_\ell(\frac 1u)-E_\ell(x))}{\bigl( (-1)^{d-1} \ell u^\ell\bigr)^{E_\ell(x)}\Gamma((-1)^{d-1}E_\ell(\frac 1u))}.
$$
However the proof has a serious analytical gap since the series $\sum_jf_j(a)u^j$ does not converge to $e^{-a}\left(1-(-1)^{d-1}\ell u^\ell a\right)^{-(-1)^{d-1}E_\ell\left(\frac 1{u}\right)}$ when $|a|>\frac 1{|\ell u^\ell|}$. To make this proof work we need to split the integral~\eqref{eq15:integral} as a sum $\int_{-\infty}^{-\frac{(-1)^{d-1}}{\ell u^\ell}}+\int_{-\frac{(-1)^{d-1}}{\ell u^\ell}}^0$. It is easy to see that the first integral has zero asymptotic expansion with respect to $u$ when $(-1)^{d-1}u^\ell\to +0$. The second integral can now be replaced by a series of integrals, whose expansion in $u$ has to be studied.

\section{Results of computations}\label{s16}
In the Appendix the results of computations of the Euler characteristics are presented. Let $h_{ijk}$ denote the rank of the $(i,j)$-component of $H_k(\overline{Emb}_d,\Q)$. Similarly let $\pi_{ijk}$ denote the rank of the $(i,j)$-component of $\pi_k(\overline{Emb}_d)\otimes\Q$. The homotopy Euler characteristics will be denoted by $\chi_{ij}^\pi$:
$$
\chi_{ij}^\pi=\sum_{k}(-1)^k\pi_{ijk}.
$$

The following lemma describes how the homotopy Euler characteristics can be obtained from the homology Euler characteristics.

\begin{lemma}\label{l161}
$$
F_d(x,u)=\sum_{ij}\chi_{ij}x^iu^j=\prod_{ij}\frac 1{(1-x^iu^j)^{\chi_{ij}^\pi}}.
\eqno(\numb)\label{eq161}
$$
\end{lemma}

\begin{proof}
By Proposition~\ref{p71} one has
$$
\sum_{ijk}h_{ijk}x^iu^jz^k=\prod_{{ij}\atop {k \text{ odd}}} (1+x^iu^jz^k)^{\pi_{ijk}}\left/
{\prod_{{ij}\atop {k \text{ even}}} (1-x^iu^jz^k)^{\pi_{ijk}}}\right..
$$
Taking $z=-1$, we obtain that the left-hand side is $F(x,u)=\sum_{ij}\chi_{ij}x^iu^j$, and the right-hand side is
$$
\prod_{{ij}\atop {k \text{ odd}}} (1-x^iu^j)^{\pi_{ijk}}\left/
{\prod_{{ij}\atop {k \text{ even}}} (1-x^iu^j)^{\pi_{ijk}}}=\prod_{ij}\frac 1{(1-x^iu^j)^{\chi_{ij}^\pi}}\right. .
$$
\end{proof}

We used the formula~\eqref{eq161} to fill the Tables~1 and~3 in the Appendix.
The column \lq\lq total" in the tables states for the sum of absolute values of $\chi^\pi_{ij}$ for a fixed complexity $j$:
$$
total=\sum_i |\chi^\pi_{ij}|.
$$
It gives a lower bound estimate for the rank of rational homotopy in a given complexity $j$. We did not make this column for the homology tables since a better lower bound estimate of the homology rank in a given complexity can be obtained using the Hodge decomposition in homotopy.

It is interesting to compare the first table with the table of primitive elements in the bialgebra of chord diagrams~\cite{Kneissler} which we copy below:
\vspace{.1cm}

\begin{center}
\small
\begin{tabular}{||c|c|c|c|c|c|c|c|}
\hline
$\mathrm{rk}_{ij}$&$i=2$&$i=4$&$i=6$&$i=8$&$i=10$&$i=12$&total\\
\hline
$j=1$&1&&&&&&1\\
\hline
$j=2$&1&&&&&&1\\
\hline
$j=3$&1&&&&&&1\\
\hline
$j=4$&1&1&&&&&2\\
\hline
$j=5$&2&1&&&&&3\\
\hline
$j=6$&2&2&1&&&&5\\
\hline
$j=7$&3&3&2&&&&8\\
\hline
$j=8$&4&4&3&1&&&12\\
\hline
$j=9$&5&6&5&2&&&18\\
\hline
$j=10$&6&8&8&4&1&&27\\
\hline
$j=11$&8&10&11&8&2&&39\\
\hline
$j=12$&9&13&15&12&5&1&55\\
\hline
\end{tabular}
\end{center}

\vspace{.2cm}

To recall Bar-Natan~\cite{BarNatan} described the space of primitives of the bialgebra of chord diagrams as the space of uni-trivalent graphs modulo $STU$ and $IHX$ relations (Theorem~\ref{t81}). The latter  space has a natural double grading. The first grading  complexity $j$ - for any graph it is the first Betti  number of the graph obtained by gluing all the univalent vertices together - this grading corresponds to the number of chords in  chord diagrams. The second grading is the number $i$ of univalent vertices. It turns out that the last grading is exactly our Hodge degree, see Theorem~\ref{t82}. In our terms the above table describes the rank of the $j(d-3)$-dimensional rational homotopy
$\pi_{j(d-3)}^{(i,j)}(\overline{Emb}_d,\Q)$ ($d$ being odd) in complexity $j$ and Hodge degree $i$.
Notice that there is no non-trivial generators in odd Hodge degree. From the point of view of knot theory this means that up to the order 12  Vassiliev invariants are orientation insensitive. It rises the question wether Vassiliev invariants can distinguish a knot from its inverse. More generally looking at Table~1 we can ask wether even cycles are all in even Hodge degrees and all odd cycles are all in odd Hodge degrees? Comparing the above table with Table~1 we can see that there must be at least one odd cycle in complexity 10 and of Hodge degree 2. Indeed, in this bigrading one has $\chi^\pi_{2,10}=5$, but the rank of the primitives of the bialgebra of chord diagrams is 6. Even more dramatically it turns out that the sign of $\chi^\pi_{ij}$ can be different from $(-1)^i$. The first counter example appears in complexity 20:
$$
\chi^\pi_{1,20}=12>0,
$$
see Table~1.

It would be interesting to understand the geometrical reason for this phenomenon of sign alternation for small complexities. Notice that this happens only when $d$ is odd (in other Tables~3-4 the signs of entries look rather random). One should also try to compute the Table~1 for higher complexities $j$ to check wether this almost alternation of signs keeps take place or completely disappears. (Our computer equipment could do it only up to $j=23$). As a conclusion one should say that these results give some optimism for finding Vassiliev invariants that can distinguish orientation of a knot. Personally I would try with the complexity 22 and Hodge degree 5!

\section{Exponential growth of the homology and homotopy  of $\overline{Emb}_d$ or taking $F(\pm 1,u)$}\label{s17}
 One can get a lower bound estimation for the rank of the homology groups in a given complexity $j$ by taking $x=\pm 1$ in the formula for $F(x,u)=\sum_{i,j}\chi_{ij}x^iu^j$. We notice first that
\begin{center}
\begin{tabular}{ccc}
$E_\ell(1)=
\begin{cases} 1,& \text{if $\ell=1$}\\
0,& \ell\geq 2;
\end{cases}
$
&\qquad\qquad
&
$
E_\ell(-1)=
\begin{cases} (-1)^\ell,& \text{if $\ell=1$ or $2$}\\
0,& \ell\geq 3.
\end{cases}
$
\end{tabular}
\end{center}

This means that for $x=1$ only the first factor of the product~\eqref{eq131}-\eqref{eq132} can be different from 1; and for $x=-1$ only first two factors can differ from 1.


 Easy computations show that
 $$F_{odd}(1,u)=\frac 1{1-u},\qquad F_{odd}(-1,u)=\frac 1{1-u-2u^2}.$$

 $$F_{even}(1,u)=\frac 1{1+u},\qquad F_{even}(-1,u)=\frac 1{1-u+2u^2}.$$

From the above formulas we will derive the following result.

\begin{thm}\label{th:exp_compl_homol}
The rank of the rational homology of $\overline{Emb}_d$ in a given complexity $j$ grows at least exponentially with $j$.
\end{thm}

\begin{proof}
Consider first the case when $d$ is odd. The formula $F_{odd}(-1,u)=\frac 1{1-u-2u^2}=\frac 1{(1+u)(1-2u)}=\frac {1/3}{1+u}+\frac{2/3}{1-2u}$ implies at least exponential growth $\approx \frac 23 2^j$ of the rank of the homology groups in complexity $j$ in this case.

Similarly in the case of even $d$ we have $F_{even}(-1,u)=\frac 1{1-u+2u^2}=\frac 1{(1-\frac{1+\sqrt{-7}}{2}u)(1-\frac{1-\sqrt{-7}}{2}u)}=\sum_ja_ju^j$ with $a_j=\frac 2{\sqrt{7}}\mathrm{Im} \left(\frac{1+\sqrt{-7}}{2}\right)^{j+1}$. Using Baker's theorem~\cite{Baker1,Baker2} one can get that $a_j$ has also exponential growth,\footnote{I am gratefull to C.~Pinner for the argument that follows.} more precisely for some $C_1>0$ and $C_2>0$ one has
$$
|a_j|>C_1\left|\frac{1+\sqrt{-7}}{2}\right|^{j}\Bigr/j^{C_2}=C_1\frac{2^{j/2}}{j^{C_2}}.
$$ 
Indeed, Baker's theorem can be formulated as follows~\cite[III]{Baker1}.

\begin{theorem*}
{\rm(A. Baker~\cite[III]{Baker1})}
If $\alpha_1$, $\alpha_2$, $\ldots$, $\alpha_n$, and $\beta_0$, $\beta_1$, $\beta_2$, $\ldots$, $\beta_n$ are algebraic of degree at most $D$ and heights at most $A$ and $B$ (assuming that $B\geq 2$) respectively, then
$$
\Lambda=\beta_0+\beta_1\log\alpha_1+\ldots+\beta_n\log\alpha_n
$$
has $\Lambda=0$  or $|\Lambda|>B^{-C}$ where $C$ is a constant depending only on $n$, $D$, and $A$.
\end{theorem*}

To recall the {\it height} of an algebraic number is the maximum of the absolute values of the relatively prime integer coefficients in the minimal  defining polynomial.

\vspace{.2cm}

We will need this theorem only when $n=2$ (and also when $n=3$ for the proof of Theorem~\ref{th:exp_compl_homot}).

Take $\alpha_1=\frac{1+\sqrt{7}i}{2\sqrt{2}}=e^{i\pi\theta}$, where $0<\theta<\frac 12$, and $\alpha_2=-1=e^{i\pi}$. So one has $\log\alpha_1=i\pi\theta$, $\log\alpha_2=i\pi$. One has
\begin{multline*}
|a_j|=\left|\frac 2{\sqrt{7}}\mathrm{Im} \left(\frac{1+i\sqrt{7}}{2}\right)^{j+1}\right|=\frac 2{\sqrt{7}} \sqrt{2}^{\,j+1}|\sin(\pi (j+1)\theta)|=\\
=\frac 2{\sqrt{7}} \sqrt{2}^{\,j+1} \sin\pi||(j+1)\theta||\geq\frac 2{\sqrt{7}} \sqrt{2}^{\,j+1}\frac 2\pi \pi||(j+1)\theta||,
\end{multline*}
where $||x||$ denote the distance to the nearest integer. Let $n=\lfloor (j+1)\theta\rfloor$ or $\lfloor (j+1)\theta\rfloor+1$ be this nearest integer. Obviously, $n\leq j+1$ since $\theta<\frac 12$. The right-hand side is
$$
\frac 4{\pi\sqrt{7}} \sqrt{2}^{\,j+1} |\pi((j+1)\theta-n)|=  \frac 4 {\pi\sqrt{7}} \sqrt{2}^{\,j+1} |(j+1)\log\alpha_1-n\log\alpha_2|>
\frac 4 {\sqrt{7}\pi} \sqrt{2}^{\,j+1} (j+1)^{-C}.
$$
The last inequality uses Baker's theorem. The only thing we need to check is that $(j+1)\theta\neq n$, or in other words that $\alpha_1$ is not a root of unity. But $\alpha_1$ is not even an algebraic integer since its minimal polynomial is
$$
\prod \left( x\pm\left(\frac{1\pm \sqrt{7}i}{2\sqrt{2}}\right)\right)=x^4+\frac 32x^2+1.
$$
\end{proof}

Similar result holds for the rational homotopy of these spaces.

\begin{thm}\label{th:exp_compl_homot}
The rank of the rational homotopy of $\overline{Emb}_d$ in a given complexity $j$ grows at least exponentially with $j$.
\end{thm}

The proof is based on a few observations. Let
$$
\chi_j=\sum_i \chi_{ij}, \qquad \chi_j^\pi=\sum_i\chi_{ij}^\pi
$$
denote the Euler characteristic of the homology, resp. homotopy of $\overline{Emb}_d$ in complexity $j$. One has
$$
F_d(1,u)=\sum_j\chi_ju^j=\prod_j\frac 1{(1-u^j)^{\chi_j^\pi}}.
$$
The last equality is a consequence of Lemma~\ref{l161}. For the proof of Theorem~\ref{th:exp_compl_homot} we will need the following lemma.

\begin{lemma}[Scannell, Sinha~\cite{SinhaSc}]\label{l:euler_homot_compl}
The Euler characteristic of the rational homotopy of  $\overline{Emb}_d$ in complexity $j$ for odd $d$ is
$$
\chi_j^\pi=
\begin{cases}
1,&\text{if $j=1$;}\\
0,&j\geq 2;
\end{cases}
$$
for even $d$ is
$$
\chi_j^\pi=
\begin{cases}
(-1)^j,&\text{if $j=1$ or $2$;}\\
0,&j\geq 3.
\end{cases}
$$
\end{lemma}
\begin{proof}
The first assertion is true since
$$
F_{odd}(1,u)=\frac 1{1-u}=\frac 1{(1-u)^1}.
$$
The second one is true since
$$
F_{even}(1,u)=\frac 1{1+u}=\frac 1{(1-u)^{-1}}\cdot \frac 1{(1-u^2)^1}.
$$
\end{proof}

\begin{proof}[Proof of Theorem~\ref{th:exp_compl_homot}]

Denote by $\chi_{E,j}^\pi$, resp. $\chi_{O,j}^\pi$ the sum of Euler characteristics over even, resp. odd Hodge degrees:
$$
\chi_{E,j}^\pi=\sum_{i \text{ even}}\chi_{i,j}^\pi, \qquad \chi_{O,j}^\pi=\sum_{i \text{ odd}}\chi_{i,j}^\pi.
$$
It follows from Lemma~\ref{l161} that
$$
F_d(-1,u)=\prod_i\frac{(1+u^j)^{-\chi_{O,j}^\pi}}{(1-u^j)^{\chi_{E,j}^\pi}}.
$$
Since $\chi_{E,j}^\pi+\chi_{O,j}^\pi=\chi_j^\pi=0$ for $j\geq 2$ in case of odd $d$ (and for $j\geq 3$ in case of even $d$, see Lemma~\ref{l:euler_homot_compl}), one has:
$$
F_{odd}(-1,u)=\frac 1{1-u}\prod_{j\geq 2}\left(\frac{1+u^j}{1-u^j}\right)^{\chi_{E,j}^\pi}=\frac 1{1-u-2u^2}=\frac 1{(1+u)(1-2u)}.
$$
(We used that $\chi_{E,1}^\pi=1$ and $\chi_{O,1}^\pi=0$, see the first row of Table~1.)
Or, equivalently,
$$
\prod_{j\geq 1}\left(\frac{1+u^j}{1-u^j}\right)^{\chi_{E,j}^\pi}=\frac 1{1-2u}.
$$
Applying logarithmic derivative, and multiplying each side by $u$, one has:
$$
\sum_{j\geq 1} 2j\chi_{E,j}^\pi\left(\frac {u^j}{1-u^{2j}}\right)=\frac {2u}{1-2u}.
$$
Therefore,
$$
\sum_{{k|n}\atop {\text{$k$ odd}}}2\frac nk\chi_{E,\frac nk}^\pi=2^n.
$$
Using M\"obius transformation one obtains:
$$
\chi_{E,j}^\pi=\frac 1{2j}\sum_{{k|j}\atop {\text{$k$ odd}}}\mu(k)2^{j/k}=\frac 1{2j}2^j+\frac 1{2j}\sum_{{k|j}\atop {\text{$k>1$ odd}}}\mu(k)2^{j/k}.
$$
The above formula implies that, in case of odd $d$, $\chi_{E,j}^\pi$ has asymptotics ${2^j}/{2j}$. Indeed, since the number of divisors is always less then the number itself, the absolute value of the right summand  is less then $\frac 12 2^{j/2}$ which is infinitely small compared to ${2^j}/{2j}$. Since the rank of rational homotopy in complexity $j$ is greater then $\chi_{E,j}^\pi$, one gets the result of the theorem for odd $d$.

The case of even $d$ is obtained in a similar way. We have
$$
F_{even}(-1,u)=\frac {1+u}1\cdot \frac {(1+u^2)^{-1}}{1}\cdot \prod_{j\geq 3}\left(\frac{1+u^j}{1-u^j}\right)^{\chi_{E,j}^\pi}=\frac 1{1-u+2u^2}
$$
(see the first and the second rows of Table~3.).
Equivalently,
$$
\prod_{j\geq 1}\left(\frac{1+u^j}{1-u^j}\right)^{\chi_{E,j}^\pi}=\frac{1+u^2}{(1+u)(1-u+2u^2)}.
$$
Taking logarithmic derivative and multiplying each side by $u$, one obtains:
$$
\sum_{j\geq 1} 2j\chi_{E,j}^\pi\left(\frac {u^j}{1-u^{2j}}\right)=\frac{2u^2}{1+u^2}-\frac {u}{1+u}+\frac{u-4u^2}{1-u+2u^2}=\sum_jB_ju^j.
$$
Using an argument similar to the proof of Theorem~\ref{th:exp_compl_homol}, one can show that the sequence $|B_j|$ starting from some $j$ is bounded by
$$
\frac{C_12^{j/2}}{j^\alpha}<|B_j|<C_22^{j/2}
\eqno(\numb)\label{eq17:bounds}
$$
for some positive constants $C_1$, $C_2$, $\alpha$. Using M\"obius transformation, one has
$$
\chi_{E,j}^\pi=\frac 1{2j}\sum_{{k|j}\atop {\text{$k$ odd}}}\mu(k)B_{j/k}=\frac 1{2j}B_j+\frac 1{2j}\sum_{{k|j}\atop {\text{$k>1$ odd}}}\mu(k)B_{j/k}.
$$
Using the lower bound of~\eqref{eq17:bounds} for the first summand, and the upper bound of~\eqref{eq17:bounds} to estimate the second one, we see that $\chi_{E,j}^\pi$ in case of even $d$ again has an exponential growth.
\end{proof}

An immediate consequence of Theorems~\ref{th:exp_compl_homol}, \ref{th:exp_compl_homot} is the following.

\begin{thm}\label{th:cumulative}
The cumulative ranks of the rational homology and of rational homotopy
$$
\rank(H_{\leq n}(\overline{Emb}_d,\Q)), \qquad \rank(\pi_{\leq n} (\overline{Emb}_d)\otimes\Q)
$$
grow at least exponentially with $n$.
\end{thm}
\begin{proof}
The idea is that the homology/homotopy in complexity $j$ has the total degree less then $j(d-2)$, see Section~\ref{s13}. Therefore for a given $n$ the cumulative homology/homotopy is sure to contain the whole homology/homotopy of complexity $\lfloor \frac n{d-2}\rfloor$.
\end{proof}

\section*{Acknowledgement}
The author would like to thank first M.~Kontsevich, short discussions with whom were always the most enlightening. He is also grateful to G.~Arone, A.~Giaquinto, P.~Lambrechts, J.-L.~Loday, F.~Patras, M.~Ronco, P.~Salvatore, M.~Vigu\'e-Poirrier for discussions and communications, and to C.~Pinner who provided  a proof of Theorem~\ref{th:exp_compl_homol}. The author is also grateful to G.~Arone, P.~Lambrechts, P.~Salvatore and to their institutions University of Virginia, Universit\'e Catholique de Louvain, and Universita di Roma Tor Vergata for hospitality. This work was partially written during the stay at the Max Planck Institute of Mathematics, that the author also thanks for its hospitality.

\newpage

\appendix

\section{Tables}

\begin{table}[h!]

{\tiny \begin{center}
\begin{tabular}{|c|c|c|c|c|c|c|c|c|c|c|c|c|c|c|c|c|c|c|c|c|c|c|c|c|}
\hline
&\multicolumn{23}{|c|}{Hodge degree $i$}&\\
\hline
$j$ & 1 & 2 & 3 & 4 & 5 & 6 & 7 & 8 & 9 & 10 & 11 & 12 & 13 & 14 & 15 & 16 & 17 & 18 & 19 & 20 & 21 & 22 & 23 & total \\
\hline
1 & \begin{turn}{80}{}\end{turn} & \begin{turn}{80}{1}\end{turn} & \begin{turn}{80}{}\end{turn} & \begin{turn}{80}{}\end{turn} & \begin{turn}{80}{}\end{turn} & \begin{turn}{80}{}\end{turn} & \begin{turn}{80}{}\end{turn} & \begin{turn}{80}{}\end{turn} & \begin{turn}{80}{}\end{turn} & \begin{turn}{80}{}\end{turn} & \begin{turn}{80}{}\end{turn} & \begin{turn}{80}{}\end{turn} & \begin{turn}{80}{}\end{turn} & \begin{turn}{80}{}\end{turn} & \begin{turn}{80}{}\end{turn} & \begin{turn}{80}{}\end{turn} & \begin{turn}{80}{}\end{turn} & \begin{turn}{80}{}\end{turn} & \begin{turn}{80}{}\end{turn} & \begin{turn}{80}{}\end{turn} & \begin{turn}{80}{}\end{turn} & \begin{turn}{80}{}\end{turn} & \begin{turn}{80}{}\end{turn} & 1 \\
\hline
2 & \begin{turn}{80}{-1}\end{turn} & \begin{turn}{80}{1}\end{turn} & \begin{turn}{80}{}\end{turn} & \begin{turn}{80}{}\end{turn} & \begin{turn}{80}{}\end{turn} & \begin{turn}{80}{}\end{turn} & \begin{turn}{80}{}\end{turn} & \begin{turn}{80}{}\end{turn} & \begin{turn}{80}{}\end{turn} & \begin{turn}{80}{}\end{turn} & \begin{turn}{80}{}\end{turn} & \begin{turn}{80}{}\end{turn} & \begin{turn}{80}{}\end{turn} & \begin{turn}{80}{}\end{turn} & \begin{turn}{80}{}\end{turn} & \begin{turn}{80}{}\end{turn} & \begin{turn}{80}{}\end{turn} & \begin{turn}{80}{}\end{turn} & \begin{turn}{80}{}\end{turn} & \begin{turn}{80}{}\end{turn} & \begin{turn}{80}{}\end{turn} & \begin{turn}{80}{}\end{turn} & \begin{turn}{80}{}\end{turn} & 2 \\
\hline
3 & \begin{turn}{80}{-1}\end{turn} & \begin{turn}{80}{1}\end{turn} & \begin{turn}{80}{}\end{turn} & \begin{turn}{80}{}\end{turn} & \begin{turn}{80}{}\end{turn} & \begin{turn}{80}{}\end{turn} & \begin{turn}{80}{}\end{turn} & \begin{turn}{80}{}\end{turn} & \begin{turn}{80}{}\end{turn} & \begin{turn}{80}{}\end{turn} & \begin{turn}{80}{}\end{turn} & \begin{turn}{80}{}\end{turn} & \begin{turn}{80}{}\end{turn} & \begin{turn}{80}{}\end{turn} & \begin{turn}{80}{}\end{turn} & \begin{turn}{80}{}\end{turn} & \begin{turn}{80}{}\end{turn} & \begin{turn}{80}{}\end{turn} & \begin{turn}{80}{}\end{turn} & \begin{turn}{80}{}\end{turn} & \begin{turn}{80}{}\end{turn} & \begin{turn}{80}{}\end{turn} & \begin{turn}{80}{}\end{turn} & 2 \\
\hline
4 & \begin{turn}{80}{-1}\end{turn} & \begin{turn}{80}{1}\end{turn} & \begin{turn}{80}{-1}\end{turn} & \begin{turn}{80}{1}\end{turn} & \begin{turn}{80}{}\end{turn} & \begin{turn}{80}{}\end{turn} & \begin{turn}{80}{}\end{turn} & \begin{turn}{80}{}\end{turn} & \begin{turn}{80}{}\end{turn} & \begin{turn}{80}{}\end{turn} & \begin{turn}{80}{}\end{turn} & \begin{turn}{80}{}\end{turn} & \begin{turn}{80}{}\end{turn} & \begin{turn}{80}{}\end{turn} & \begin{turn}{80}{}\end{turn} & \begin{turn}{80}{}\end{turn} & \begin{turn}{80}{}\end{turn} & \begin{turn}{80}{}\end{turn} & \begin{turn}{80}{}\end{turn} & \begin{turn}{80}{}\end{turn} & \begin{turn}{80}{}\end{turn} & \begin{turn}{80}{}\end{turn} & \begin{turn}{80}{}\end{turn} & 4 \\
\hline
5 & \begin{turn}{80}{-2}\end{turn} & \begin{turn}{80}{2}\end{turn} & \begin{turn}{80}{-1}\end{turn} & \begin{turn}{80}{1}\end{turn} & \begin{turn}{80}{}\end{turn} & \begin{turn}{80}{}\end{turn} & \begin{turn}{80}{}\end{turn} & \begin{turn}{80}{}\end{turn} & \begin{turn}{80}{}\end{turn} & \begin{turn}{80}{}\end{turn} & \begin{turn}{80}{}\end{turn} & \begin{turn}{80}{}\end{turn} & \begin{turn}{80}{}\end{turn} & \begin{turn}{80}{}\end{turn} & \begin{turn}{80}{}\end{turn} & \begin{turn}{80}{}\end{turn} & \begin{turn}{80}{}\end{turn} & \begin{turn}{80}{}\end{turn} & \begin{turn}{80}{}\end{turn} & \begin{turn}{80}{}\end{turn} & \begin{turn}{80}{}\end{turn} & \begin{turn}{80}{}\end{turn} & \begin{turn}{80}{}\end{turn} & 6 \\
\hline
6 & \begin{turn}{80}{-1}\end{turn} & \begin{turn}{80}{2}\end{turn} & \begin{turn}{80}{-3}\end{turn} & \begin{turn}{80}{2}\end{turn} & \begin{turn}{80}{-1}\end{turn} & \begin{turn}{80}{1}\end{turn} & \begin{turn}{80}{}\end{turn} & \begin{turn}{80}{}\end{turn} & \begin{turn}{80}{}\end{turn} & \begin{turn}{80}{}\end{turn} & \begin{turn}{80}{}\end{turn} & \begin{turn}{80}{}\end{turn} & \begin{turn}{80}{}\end{turn} & \begin{turn}{80}{}\end{turn} & \begin{turn}{80}{}\end{turn} & \begin{turn}{80}{}\end{turn} & \begin{turn}{80}{}\end{turn} & \begin{turn}{80}{}\end{turn} & \begin{turn}{80}{}\end{turn} & \begin{turn}{80}{}\end{turn} & \begin{turn}{80}{}\end{turn} & \begin{turn}{80}{}\end{turn} & \begin{turn}{80}{}\end{turn} & 10 \\
\hline
7 & \begin{turn}{80}{-2}\end{turn} & \begin{turn}{80}{3}\end{turn} & \begin{turn}{80}{-4}\end{turn} & \begin{turn}{80}{4}\end{turn} & \begin{turn}{80}{-3}\end{turn} & \begin{turn}{80}{2}\end{turn} & \begin{turn}{80}{}\end{turn} & \begin{turn}{80}{}\end{turn} & \begin{turn}{80}{}\end{turn} & \begin{turn}{80}{}\end{turn} & \begin{turn}{80}{}\end{turn} & \begin{turn}{80}{}\end{turn} & \begin{turn}{80}{}\end{turn} & \begin{turn}{80}{}\end{turn} & \begin{turn}{80}{}\end{turn} & \begin{turn}{80}{}\end{turn} & \begin{turn}{80}{}\end{turn} & \begin{turn}{80}{}\end{turn} & \begin{turn}{80}{}\end{turn} & \begin{turn}{80}{}\end{turn} & \begin{turn}{80}{}\end{turn} & \begin{turn}{80}{}\end{turn} & \begin{turn}{80}{}\end{turn} & 18 \\
\hline
8 & \begin{turn}{80}{-2}\end{turn} & \begin{turn}{80}{4}\end{turn} & \begin{turn}{80}{-6}\end{turn} & \begin{turn}{80}{7}\end{turn} & \begin{turn}{80}{-6}\end{turn} & \begin{turn}{80}{4}\end{turn} & \begin{turn}{80}{-2}\end{turn} & \begin{turn}{80}{1}\end{turn} & \begin{turn}{80}{}\end{turn} & \begin{turn}{80}{}\end{turn} & \begin{turn}{80}{}\end{turn} & \begin{turn}{80}{}\end{turn} & \begin{turn}{80}{}\end{turn} & \begin{turn}{80}{}\end{turn} & \begin{turn}{80}{}\end{turn} & \begin{turn}{80}{}\end{turn} & \begin{turn}{80}{}\end{turn} & \begin{turn}{80}{}\end{turn} & \begin{turn}{80}{}\end{turn} & \begin{turn}{80}{}\end{turn} & \begin{turn}{80}{}\end{turn} & \begin{turn}{80}{}\end{turn} & \begin{turn}{80}{}\end{turn} & 32 \\
\hline
9 & \begin{turn}{80}{-2}\end{turn} & \begin{turn}{80}{5}\end{turn} & \begin{turn}{80}{-10}\end{turn} & \begin{turn}{80}{12}\end{turn} & \begin{turn}{80}{-11}\end{turn} & \begin{turn}{80}{9}\end{turn} & \begin{turn}{80}{-5}\end{turn} & \begin{turn}{80}{2}\end{turn} & \begin{turn}{80}{}\end{turn} & \begin{turn}{80}{}\end{turn} & \begin{turn}{80}{}\end{turn} & \begin{turn}{80}{}\end{turn} & \begin{turn}{80}{}\end{turn} & \begin{turn}{80}{}\end{turn} & \begin{turn}{80}{}\end{turn} & \begin{turn}{80}{}\end{turn} & \begin{turn}{80}{}\end{turn} & \begin{turn}{80}{}\end{turn} & \begin{turn}{80}{}\end{turn} & \begin{turn}{80}{}\end{turn} & \begin{turn}{80}{}\end{turn} & \begin{turn}{80}{}\end{turn} & \begin{turn}{80}{}\end{turn} & 56 \\
\hline
10 & \begin{turn}{80}{-1}\end{turn} & \begin{turn}{80}{5}\end{turn} & \begin{turn}{80}{-14}\end{turn} & \begin{turn}{80}{20}\end{turn} & \begin{turn}{80}{-22}\end{turn} & \begin{turn}{80}{19}\end{turn} & \begin{turn}{80}{-12}\end{turn} & \begin{turn}{80}{6}\end{turn} & \begin{turn}{80}{-2}\end{turn} & \begin{turn}{80}{1}\end{turn} & \begin{turn}{80}{}\end{turn} & \begin{turn}{80}{}\end{turn} & \begin{turn}{80}{}\end{turn} & \begin{turn}{80}{}\end{turn} & \begin{turn}{80}{}\end{turn} & \begin{turn}{80}{}\end{turn} & \begin{turn}{80}{}\end{turn} & \begin{turn}{80}{}\end{turn} & \begin{turn}{80}{}\end{turn} & \begin{turn}{80}{}\end{turn} & \begin{turn}{80}{}\end{turn} & \begin{turn}{80}{}\end{turn} & \begin{turn}{80}{}\end{turn} & 102 \\
\hline
11 & \begin{turn}{80}{-2}\end{turn} & \begin{turn}{80}{7}\end{turn} & \begin{turn}{80}{-17}\end{turn} & \begin{turn}{80}{30}\end{turn} & \begin{turn}{80}{-39}\end{turn} & \begin{turn}{80}{38}\end{turn} & \begin{turn}{80}{-29}\end{turn} & \begin{turn}{80}{16}\end{turn} & \begin{turn}{80}{-6}\end{turn} & \begin{turn}{80}{2}\end{turn} & \begin{turn}{80}{}\end{turn} & \begin{turn}{80}{}\end{turn} & \begin{turn}{80}{}\end{turn} & \begin{turn}{80}{}\end{turn} & \begin{turn}{80}{}\end{turn} & \begin{turn}{80}{}\end{turn} & \begin{turn}{80}{}\end{turn} & \begin{turn}{80}{}\end{turn} & \begin{turn}{80}{}\end{turn} & \begin{turn}{80}{}\end{turn} & \begin{turn}{80}{}\end{turn} & \begin{turn}{80}{}\end{turn} & \begin{turn}{80}{}\end{turn} & 186 \\
\hline
12 & \begin{turn}{80}{}\end{turn} & \begin{turn}{80}{5}\end{turn} & \begin{turn}{80}{-22}\end{turn} & \begin{turn}{80}{45}\end{turn} & \begin{turn}{80}{-66}\end{turn} & \begin{turn}{80}{72}\end{turn} & \begin{turn}{80}{-60}\end{turn} & \begin{turn}{80}{40}\end{turn} & \begin{turn}{80}{-20}\end{turn} & \begin{turn}{80}{7}\end{turn} & \begin{turn}{80}{-2}\end{turn} & \begin{turn}{80}{1}\end{turn} & \begin{turn}{80}{}\end{turn} & \begin{turn}{80}{}\end{turn} & \begin{turn}{80}{}\end{turn} & \begin{turn}{80}{}\end{turn} & \begin{turn}{80}{}\end{turn} & \begin{turn}{80}{}\end{turn} & \begin{turn}{80}{}\end{turn} & \begin{turn}{80}{}\end{turn} & \begin{turn}{80}{}\end{turn} & \begin{turn}{80}{}\end{turn} & \begin{turn}{80}{}\end{turn} & 340 \\
\hline
13 & \begin{turn}{80}{}\end{turn} & \begin{turn}{80}{4}\end{turn} & \begin{turn}{80}{-25}\end{turn} & \begin{turn}{80}{60}\end{turn} & \begin{turn}{80}{-104}\end{turn} & \begin{turn}{80}{133}\end{turn} & \begin{turn}{80}{-125}\end{turn} & \begin{turn}{80}{91}\end{turn} & \begin{turn}{80}{-52}\end{turn} & \begin{turn}{80}{24}\end{turn} & \begin{turn}{80}{-9}\end{turn} & \begin{turn}{80}{3}\end{turn} & \begin{turn}{80}{}\end{turn} & \begin{turn}{80}{}\end{turn} & \begin{turn}{80}{}\end{turn} & \begin{turn}{80}{}\end{turn} & \begin{turn}{80}{}\end{turn} & \begin{turn}{80}{}\end{turn} & \begin{turn}{80}{}\end{turn} & \begin{turn}{80}{}\end{turn} & \begin{turn}{80}{}\end{turn} & \begin{turn}{80}{}\end{turn} & \begin{turn}{80}{}\end{turn} & 630 \\
\hline
14 & \begin{turn}{80}{-1}\end{turn} & \begin{turn}{80}{2}\end{turn} & \begin{turn}{80}{-22}\end{turn} & \begin{turn}{80}{79}\end{turn} & \begin{turn}{80}{-155}\end{turn} & \begin{turn}{80}{221}\end{turn} & \begin{turn}{80}{-244}\end{turn} & \begin{turn}{80}{203}\end{turn} & \begin{turn}{80}{-130}\end{turn} & \begin{turn}{80}{68}\end{turn} & \begin{turn}{80}{-30}\end{turn} & \begin{turn}{80}{11}\end{turn} & \begin{turn}{80}{-3}\end{turn} & \begin{turn}{80}{1}\end{turn} & \begin{turn}{80}{}\end{turn} & \begin{turn}{80}{}\end{turn} & \begin{turn}{80}{}\end{turn} & \begin{turn}{80}{}\end{turn} & \begin{turn}{80}{}\end{turn} & \begin{turn}{80}{}\end{turn} & \begin{turn}{80}{}\end{turn} & \begin{turn}{80}{}\end{turn} & \begin{turn}{80}{}\end{turn} & 1170 \\
\hline
15 & \begin{turn}{80}{-1}\end{turn} & \begin{turn}{80}{3}\end{turn} & \begin{turn}{80}{-17}\end{turn} & \begin{turn}{80}{81}\end{turn} & \begin{turn}{80}{-217}\end{turn} & \begin{turn}{80}{368}\end{turn} & \begin{turn}{80}{-445}\end{turn} & \begin{turn}{80}{413}\end{turn} & \begin{turn}{80}{-308}\end{turn} & \begin{turn}{80}{186}\end{turn} & \begin{turn}{80}{-91}\end{turn} & \begin{turn}{80}{37}\end{turn} & \begin{turn}{80}{-12}\end{turn} & \begin{turn}{80}{3}\end{turn} & \begin{turn}{80}{}\end{turn} & \begin{turn}{80}{}\end{turn} & \begin{turn}{80}{}\end{turn} & \begin{turn}{80}{}\end{turn} & \begin{turn}{80}{}\end{turn} & \begin{turn}{80}{}\end{turn} & \begin{turn}{80}{}\end{turn} & \begin{turn}{80}{}\end{turn} & \begin{turn}{80}{}\end{turn} & 2182 \\
\hline
16 & \begin{turn}{80}{-4}\end{turn} & \begin{turn}{80}{3}\end{turn} & \begin{turn}{80}{-12}\end{turn} & \begin{turn}{80}{83}\end{turn} & \begin{turn}{80}{-275}\end{turn} & \begin{turn}{80}{549}\end{turn} & \begin{turn}{80}{-769}\end{turn} & \begin{turn}{80}{823}\end{turn} & \begin{turn}{80}{-685}\end{turn} & \begin{turn}{80}{455}\end{turn} & \begin{turn}{80}{-255}\end{turn} & \begin{turn}{80}{121}\end{turn} & \begin{turn}{80}{-45}\end{turn} & \begin{turn}{80}{13}\end{turn} & \begin{turn}{80}{-3}\end{turn} & \begin{turn}{80}{1}\end{turn} & \begin{turn}{80}{}\end{turn} & \begin{turn}{80}{}\end{turn} & \begin{turn}{80}{}\end{turn} & \begin{turn}{80}{}\end{turn} & \begin{turn}{80}{}\end{turn} & \begin{turn}{80}{}\end{turn} & \begin{turn}{80}{}\end{turn} & 4096 \\
\hline
17 & \begin{turn}{80}{-18}\end{turn} & \begin{turn}{80}{19}\end{turn} & \begin{turn}{80}{-12}\end{turn} & \begin{turn}{80}{79}\end{turn} & \begin{turn}{80}{-307}\end{turn} & \begin{turn}{80}{751}\end{turn} & \begin{turn}{80}{-1258}\end{turn} & \begin{turn}{80}{1528}\end{turn} & \begin{turn}{80}{-1422}\end{turn} & \begin{turn}{80}{1071}\end{turn} & \begin{turn}{80}{-672}\end{turn} & \begin{turn}{80}{351}\end{turn} & \begin{turn}{80}{-152}\end{turn} & \begin{turn}{80}{53}\end{turn} & \begin{turn}{80}{-14}\end{turn} & \begin{turn}{80}{3}\end{turn} & \begin{turn}{80}{}\end{turn} & \begin{turn}{80}{}\end{turn} & \begin{turn}{80}{}\end{turn} & \begin{turn}{80}{}\end{turn} & \begin{turn}{80}{}\end{turn} & \begin{turn}{80}{}\end{turn} & \begin{turn}{80}{}\end{turn} & 7710 \\
\hline
18 & \begin{turn}{80}{-20}\end{turn} & \begin{turn}{80}{59}\end{turn} & \begin{turn}{80}{-83}\end{turn} & \begin{turn}{80}{65}\end{turn} & \begin{turn}{80}{-257}\end{turn} & \begin{turn}{80}{964}\end{turn} & \begin{turn}{80}{-1943}\end{turn} & \begin{turn}{80}{2651}\end{turn} & \begin{turn}{80}{-2781}\end{turn} & \begin{turn}{80}{2369}\end{turn} & \begin{turn}{80}{-1666}\end{turn} & \begin{turn}{80}{969}\end{turn} & \begin{turn}{80}{-465}\end{turn} & \begin{turn}{80}{186}\end{turn} & \begin{turn}{80}{-62}\end{turn} & \begin{turn}{80}{16}\end{turn} & \begin{turn}{80}{-3}\end{turn} & \begin{turn}{80}{1}\end{turn} & \begin{turn}{80}{}\end{turn} & \begin{turn}{80}{}\end{turn} & \begin{turn}{80}{}\end{turn} & \begin{turn}{80}{}\end{turn} & \begin{turn}{80}{}\end{turn} & 14560 \\
\hline
19 & \begin{turn}{80}{-13}\end{turn} & \begin{turn}{80}{124}\end{turn} & \begin{turn}{80}{-188}\end{turn} & \begin{turn}{80}{59}\end{turn} & \begin{turn}{80}{-298}\end{turn} & \begin{turn}{80}{1234}\end{turn} & \begin{turn}{80}{-2646}\end{turn} & \begin{turn}{80}{4224}\end{turn} & \begin{turn}{80}{-5203}\end{turn} & \begin{turn}{80}{4983}\end{turn} & \begin{turn}{80}{-3850}\end{turn} & \begin{turn}{80}{2486}\end{turn} & \begin{turn}{80}{-1353}\end{turn} & \begin{turn}{80}{613}\end{turn} & \begin{turn}{80}{-228}\end{turn} & \begin{turn}{80}{70}\end{turn} & \begin{turn}{80}{-18}\end{turn} & \begin{turn}{80}{4}\end{turn} & \begin{turn}{80}{}\end{turn} & \begin{turn}{80}{}\end{turn} & \begin{turn}{80}{}\end{turn} & \begin{turn}{80}{}\end{turn} & \begin{turn}{80}{}\end{turn} & 27594 \\
\hline
20 & \begin{turn}{80}{12}\end{turn} & \begin{turn}{80}{115}\end{turn} & \begin{turn}{80}{-225}\end{turn} & \begin{turn}{80}{442}\end{turn} & \begin{turn}{80}{-807}\end{turn} & \begin{turn}{80}{1202}\end{turn} & \begin{turn}{80}{-3068}\end{turn} & \begin{turn}{80}{6527}\end{turn} & \begin{turn}{80}{-9208}\end{turn} & \begin{turn}{80}{9707}\end{turn} & \begin{turn}{80}{-8379}\end{turn} & \begin{turn}{80}{6075}\end{turn} & \begin{turn}{80}{-3672}\end{turn} & \begin{turn}{80}{1847}\end{turn} & \begin{turn}{80}{-781}\end{turn} & \begin{turn}{80}{278}\end{turn} & \begin{turn}{80}{-82}\end{turn} & \begin{turn}{80}{20}\end{turn} & \begin{turn}{80}{-4}\end{turn} & \begin{turn}{80}{1}\end{turn} & \begin{turn}{80}{}\end{turn} & \begin{turn}{80}{}\end{turn} & \begin{turn}{80}{}\end{turn} & 52452 \\
\hline
21 & \begin{turn}{80}{158}\end{turn} & \begin{turn}{80}{-281}\end{turn} & \begin{turn}{80}{-607}\end{turn} & \begin{turn}{80}{1998}\end{turn} & \begin{turn}{80}{-1171}\end{turn} & \begin{turn}{80}{378}\end{turn} & \begin{turn}{80}{-4068}\end{turn} & \begin{turn}{80}{9921}\end{turn} & \begin{turn}{80}{-14656}\end{turn} & \begin{turn}{80}{17558}\end{turn} & \begin{turn}{80}{-17437}\end{turn} & \begin{turn}{80}{14053}\end{turn} & \begin{turn}{80}{-9307}\end{turn} & \begin{turn}{80}{5204}\end{turn} & \begin{turn}{80}{-2483}\end{turn} & \begin{turn}{80}{999}\end{turn} & \begin{turn}{80}{-336}\end{turn} & \begin{turn}{80}{95}\end{turn} & \begin{turn}{80}{-22}\end{turn} & \begin{turn}{80}{4}\end{turn} & \begin{turn}{80}{}\end{turn} & \begin{turn}{80}{}\end{turn} & \begin{turn}{80}{}\end{turn} & 100736 \\
\hline
22 & \begin{turn}{80}{638}\end{turn} & \begin{turn}{80}{-457}\end{turn} & \begin{turn}{80}{-2294}\end{turn} & \begin{turn}{80}{2080}\end{turn} & \begin{turn}{80}{613}\end{turn} & \begin{turn}{80}{3026}\end{turn} & \begin{turn}{80}{-8531}\end{turn} & \begin{turn}{80}{11888}\end{turn} & \begin{turn}{80}{-20247}\end{turn} & \begin{turn}{80}{30923}\end{turn} & \begin{turn}{80}{-34589}\end{turn} & \begin{turn}{80}{30247}\end{turn} & \begin{turn}{80}{-22186}\end{turn} & \begin{turn}{80}{13900}\end{turn} & \begin{turn}{80}{-7367}\end{turn} & \begin{turn}{80}{3287}\end{turn} & \begin{turn}{80}{-1248}\end{turn} & \begin{turn}{80}{406}\end{turn} & \begin{turn}{80}{-110}\end{turn} & \begin{turn}{80}{24}\end{turn} & \begin{turn}{80}{-4}\end{turn} & \begin{turn}{80}{1}\end{turn} & \begin{turn}{80}{}\end{turn} & 194066 \\
\hline
23 & \begin{turn}{80}{480}\end{turn} & \begin{turn}{80}{1706}\end{turn} & \begin{turn}{80}{967}\end{turn} & \begin{turn}{80}{-7614}\end{turn} & \begin{turn}{80}{-6392}\end{turn} & \begin{turn}{80}{20835}\end{turn} & \begin{turn}{80}{-8447}\end{turn} & \begin{turn}{80}{5974}\end{turn} & \begin{turn}{80}{-31163}\end{turn} & \begin{turn}{80}{54026}\end{turn} & \begin{turn}{80}{-61977}\end{turn} & \begin{turn}{80}{60522}\end{turn} & \begin{turn}{80}{-50681}\end{turn} & \begin{turn}{80}{35146}\end{turn} & \begin{turn}{80}{-20337}\end{turn} & \begin{turn}{80}{10068}\end{turn} & \begin{turn}{80}{-4302}\end{turn} & \begin{turn}{80}{1570}\end{turn} & \begin{turn}{80}{-484}\end{turn} & \begin{turn}{80}{124}\end{turn} & \begin{turn}{80}{-25}\end{turn} & \begin{turn}{80}{4}\end{turn} & \begin{turn}{80}{}\end{turn} & 382844 \\
\hline
\end{tabular}
\end{center}}
\vspace{.4cm}
\caption{Table of Euler characteristics $\chi_{ij}^\pi$ by complexity $j$ and Hodge degree $i$ of $\pi_*(\overline{Emb}_d)\otimes\BQ$ for odd $d$.}
\end{table}

\begin{table}[h!]

\begin{center}
\tiny
\begin{tabular}{|c|c|c|c|c|c|c|c|c|c|c|c|c|c|c|c|c|c|c|c|c|c|c|c|}\hline &\multicolumn{23}{|c|}{Hodge degree $i$}\\
\hline
$j$ & 1 & 2 & 3 & 4 & 5 & 6 & 7 & 8 & 9 & 10 & 11 & 12 & 13 & 14 & 15 & 16 & 17 & 18 & 19 & 20 & 21 & 22 & 23 \\
\hline
1 & \begin{turn}{80}{}\end{turn} & \begin{turn}{80}{1}\end{turn} & \begin{turn}{80}{}\end{turn} & \begin{turn}{80}{}\end{turn} & \begin{turn}{80}{}\end{turn} & \begin{turn}{80}{}\end{turn} & \begin{turn}{80}{}\end{turn} & \begin{turn}{80}{}\end{turn} & \begin{turn}{80}{}\end{turn} & \begin{turn}{80}{}\end{turn} & \begin{turn}{80}{}\end{turn} & \begin{turn}{80}{}\end{turn} & \begin{turn}{80}{}\end{turn} & \begin{turn}{80}{}\end{turn} & \begin{turn}{80}{}\end{turn} & \begin{turn}{80}{}\end{turn} & \begin{turn}{80}{}\end{turn} & \begin{turn}{80}{}\end{turn} & \begin{turn}{80}{}\end{turn} & \begin{turn}{80}{}\end{turn} & \begin{turn}{80}{}\end{turn} & \begin{turn}{80}{}\end{turn} & \begin{turn}{80}{}\end{turn}  \\
\hline
2 & \begin{turn}{80}{-1}\end{turn} & \begin{turn}{80}{1}\end{turn} & \begin{turn}{80}{}\end{turn} & \begin{turn}{80}{1}\end{turn} & \begin{turn}{80}{}\end{turn} & \begin{turn}{80}{}\end{turn} & \begin{turn}{80}{}\end{turn} & \begin{turn}{80}{}\end{turn} & \begin{turn}{80}{}\end{turn} & \begin{turn}{80}{}\end{turn} & \begin{turn}{80}{}\end{turn} & \begin{turn}{80}{}\end{turn} & \begin{turn}{80}{}\end{turn} & \begin{turn}{80}{}\end{turn} & \begin{turn}{80}{}\end{turn} & \begin{turn}{80}{}\end{turn} & \begin{turn}{80}{}\end{turn} & \begin{turn}{80}{}\end{turn} & \begin{turn}{80}{}\end{turn} & \begin{turn}{80}{}\end{turn} & \begin{turn}{80}{}\end{turn} & \begin{turn}{80}{}\end{turn} & \begin{turn}{80}{}\end{turn}  \\
\hline
3 & \begin{turn}{80}{-1}\end{turn} & \begin{turn}{80}{1}\end{turn} & \begin{turn}{80}{-1}\end{turn} & \begin{turn}{80}{1}\end{turn} & \begin{turn}{80}{}\end{turn} & \begin{turn}{80}{1}\end{turn} & \begin{turn}{80}{}\end{turn} & \begin{turn}{80}{}\end{turn} & \begin{turn}{80}{}\end{turn} & \begin{turn}{80}{}\end{turn} & \begin{turn}{80}{}\end{turn} & \begin{turn}{80}{}\end{turn} & \begin{turn}{80}{}\end{turn} & \begin{turn}{80}{}\end{turn} & \begin{turn}{80}{}\end{turn} & \begin{turn}{80}{}\end{turn} & \begin{turn}{80}{}\end{turn} & \begin{turn}{80}{}\end{turn} & \begin{turn}{80}{}\end{turn} & \begin{turn}{80}{}\end{turn} & \begin{turn}{80}{}\end{turn} & \begin{turn}{80}{}\end{turn} & \begin{turn}{80}{}\end{turn}  \\
\hline
4 & \begin{turn}{80}{-1}\end{turn} & \begin{turn}{80}{1}\end{turn} & \begin{turn}{80}{-3}\end{turn} & \begin{turn}{80}{3}\end{turn} & \begin{turn}{80}{-1}\end{turn} & \begin{turn}{80}{1}\end{turn} & \begin{turn}{80}{}\end{turn} & \begin{turn}{80}{1}\end{turn} & \begin{turn}{80}{}\end{turn} & \begin{turn}{80}{}\end{turn} & \begin{turn}{80}{}\end{turn} & \begin{turn}{80}{}\end{turn} & \begin{turn}{80}{}\end{turn} & \begin{turn}{80}{}\end{turn} & \begin{turn}{80}{}\end{turn} & \begin{turn}{80}{}\end{turn} & \begin{turn}{80}{}\end{turn} & \begin{turn}{80}{}\end{turn} & \begin{turn}{80}{}\end{turn} & \begin{turn}{80}{}\end{turn} & \begin{turn}{80}{}\end{turn} & \begin{turn}{80}{}\end{turn} & \begin{turn}{80}{}\end{turn}  \\
\hline
5 & \begin{turn}{80}{-2}\end{turn} & \begin{turn}{80}{3}\end{turn} & \begin{turn}{80}{-4}\end{turn} & \begin{turn}{80}{3}\end{turn} & \begin{turn}{80}{-3}\end{turn} & \begin{turn}{80}{3}\end{turn} & \begin{turn}{80}{-1}\end{turn} & \begin{turn}{80}{1}\end{turn} & \begin{turn}{80}{}\end{turn} & \begin{turn}{80}{1}\end{turn} & \begin{turn}{80}{}\end{turn} & \begin{turn}{80}{}\end{turn} & \begin{turn}{80}{}\end{turn} & \begin{turn}{80}{}\end{turn} & \begin{turn}{80}{}\end{turn} & \begin{turn}{80}{}\end{turn} & \begin{turn}{80}{}\end{turn} & \begin{turn}{80}{}\end{turn} & \begin{turn}{80}{}\end{turn} & \begin{turn}{80}{}\end{turn} & \begin{turn}{80}{}\end{turn} & \begin{turn}{80}{}\end{turn} & \begin{turn}{80}{}\end{turn}  \\
\hline
6 & \begin{turn}{80}{-1}\end{turn} & \begin{turn}{80}{3}\end{turn} & \begin{turn}{80}{-8}\end{turn} & \begin{turn}{80}{8}\end{turn} & \begin{turn}{80}{-8}\end{turn} & \begin{turn}{80}{6}\end{turn} & \begin{turn}{80}{-3}\end{turn} & \begin{turn}{80}{3}\end{turn} & \begin{turn}{80}{-1}\end{turn} & \begin{turn}{80}{1}\end{turn} & \begin{turn}{80}{}\end{turn} & \begin{turn}{80}{1}\end{turn} & \begin{turn}{80}{}\end{turn} & \begin{turn}{80}{}\end{turn} & \begin{turn}{80}{}\end{turn} & \begin{turn}{80}{}\end{turn} & \begin{turn}{80}{}\end{turn} & \begin{turn}{80}{}\end{turn} & \begin{turn}{80}{}\end{turn} & \begin{turn}{80}{}\end{turn} & \begin{turn}{80}{}\end{turn} & \begin{turn}{80}{}\end{turn} & \begin{turn}{80}{}\end{turn}  \\
\hline
7 & \begin{turn}{80}{-2}\end{turn} & \begin{turn}{80}{6}\end{turn} & \begin{turn}{80}{-11}\end{turn} & \begin{turn}{80}{13}\end{turn} & \begin{turn}{80}{-17}\end{turn} & \begin{turn}{80}{13}\end{turn} & \begin{turn}{80}{-8}\end{turn} & \begin{turn}{80}{6}\end{turn} & \begin{turn}{80}{-3}\end{turn} & \begin{turn}{80}{3}\end{turn} & \begin{turn}{80}{-1}\end{turn} & \begin{turn}{80}{1}\end{turn} & \begin{turn}{80}{}\end{turn} & \begin{turn}{80}{1}\end{turn} & \begin{turn}{80}{}\end{turn} & \begin{turn}{80}{}\end{turn} & \begin{turn}{80}{}\end{turn} & \begin{turn}{80}{}\end{turn} & \begin{turn}{80}{}\end{turn} & \begin{turn}{80}{}\end{turn} & \begin{turn}{80}{}\end{turn} & \begin{turn}{80}{}\end{turn} & \begin{turn}{80}{}\end{turn}  \\
\hline
8 & \begin{turn}{80}{-2}\end{turn} & \begin{turn}{80}{7}\end{turn} & \begin{turn}{80}{-16}\end{turn} & \begin{turn}{80}{25}\end{turn} & \begin{turn}{80}{-30}\end{turn} & \begin{turn}{80}{25}\end{turn} & \begin{turn}{80}{-25}\end{turn} & \begin{turn}{80}{18}\end{turn} & \begin{turn}{80}{-8}\end{turn} & \begin{turn}{80}{6}\end{turn} & \begin{turn}{80}{-3}\end{turn} & \begin{turn}{80}{3}\end{turn} & \begin{turn}{80}{-1}\end{turn} & \begin{turn}{80}{1}\end{turn} & \begin{turn}{80}{}\end{turn} & \begin{turn}{80}{1}\end{turn} & \begin{turn}{80}{}\end{turn} & \begin{turn}{80}{}\end{turn} & \begin{turn}{80}{}\end{turn} & \begin{turn}{80}{}\end{turn} & \begin{turn}{80}{}\end{turn} & \begin{turn}{80}{}\end{turn} & \begin{turn}{80}{}\end{turn}  \\
\hline
9 & \begin{turn}{80}{-2}\end{turn} & \begin{turn}{80}{10}\end{turn} & \begin{turn}{80}{-25}\end{turn} & \begin{turn}{80}{41}\end{turn} & \begin{turn}{80}{-55}\end{turn} & \begin{turn}{80}{57}\end{turn} & \begin{turn}{80}{-51}\end{turn} & \begin{turn}{80}{34}\end{turn} & \begin{turn}{80}{-25}\end{turn} & \begin{turn}{80}{18}\end{turn} & \begin{turn}{80}{-8}\end{turn} & \begin{turn}{80}{6}\end{turn} & \begin{turn}{80}{-3}\end{turn} & \begin{turn}{80}{3}\end{turn} & \begin{turn}{80}{-1}\end{turn} & \begin{turn}{80}{1}\end{turn} & \begin{turn}{80}{}\end{turn} & \begin{turn}{80}{1}\end{turn} & \begin{turn}{80}{}\end{turn} & \begin{turn}{80}{}\end{turn} & \begin{turn}{80}{}\end{turn} & \begin{turn}{80}{}\end{turn} & \begin{turn}{80}{}\end{turn}  \\
\hline
10 & \begin{turn}{80}{-1}\end{turn} & \begin{turn}{80}{11}\end{turn} & \begin{turn}{80}{-36}\end{turn} & \begin{turn}{80}{67}\end{turn} & \begin{turn}{80}{-95}\end{turn} & \begin{turn}{80}{108}\end{turn} & \begin{turn}{80}{-107}\end{turn} & \begin{turn}{80}{86}\end{turn} & \begin{turn}{80}{-65}\end{turn} & \begin{turn}{80}{41}\end{turn} & \begin{turn}{80}{-25}\end{turn} & \begin{turn}{80}{18}\end{turn} & \begin{turn}{80}{-8}\end{turn} & \begin{turn}{80}{6}\end{turn} & \begin{turn}{80}{-3}\end{turn} & \begin{turn}{80}{3}\end{turn} & \begin{turn}{80}{-1}\end{turn} & \begin{turn}{80}{1}\end{turn} & \begin{turn}{80}{}\end{turn} & \begin{turn}{80}{1}\end{turn} & \begin{turn}{80}{}\end{turn} & \begin{turn}{80}{}\end{turn} & \begin{turn}{80}{}\end{turn}  \\
\hline
11 & \begin{turn}{80}{-2}\end{turn} & \begin{turn}{80}{15}\end{turn} & \begin{turn}{80}{-45}\end{turn} & \begin{turn}{80}{101}\end{turn} & \begin{turn}{80}{-166}\end{turn} & \begin{turn}{80}{207}\end{turn} & \begin{turn}{80}{-217}\end{turn} & \begin{turn}{80}{188}\end{turn} & \begin{turn}{80}{-150}\end{turn} & \begin{turn}{80}{102}\end{turn} & \begin{turn}{80}{-65}\end{turn} & \begin{turn}{80}{41}\end{turn} & \begin{turn}{80}{-25}\end{turn} & \begin{turn}{80}{18}\end{turn} & \begin{turn}{80}{-8}\end{turn} & \begin{turn}{80}{6}\end{turn} & \begin{turn}{80}{-3}\end{turn} & \begin{turn}{80}{3}\end{turn} & \begin{turn}{80}{-1}\end{turn} & \begin{turn}{80}{1}\end{turn} & \begin{turn}{80}{}\end{turn} & \begin{turn}{80}{1}\end{turn} & \begin{turn}{80}{}\end{turn}  \\
\hline
12 & \begin{turn}{80}{}\end{turn} & \begin{turn}{80}{14}\end{turn} & \begin{turn}{80}{-61}\end{turn} & \begin{turn}{80}{153}\end{turn} & \begin{turn}{80}{-267}\end{turn} & \begin{turn}{80}{367}\end{turn} & \begin{turn}{80}{-422}\end{turn} & \begin{turn}{80}{405}\end{turn} & \begin{turn}{80}{-340}\end{turn} & \begin{turn}{80}{244}\end{turn} & \begin{turn}{80}{-173}\end{turn} & \begin{turn}{80}{113}\end{turn} & \begin{turn}{80}{-65}\end{turn} & \begin{turn}{80}{41}\end{turn} & \begin{turn}{80}{-25}\end{turn} & \begin{turn}{80}{18}\end{turn} & \begin{turn}{80}{-8}\end{turn} & \begin{turn}{80}{6}\end{turn} & \begin{turn}{80}{-3}\end{turn} & \begin{turn}{80}{3}\end{turn} & \begin{turn}{80}{-1}\end{turn} & \begin{turn}{80}{1}\end{turn} & \begin{turn}{80}{}\end{turn}  \\
\hline
13 & \begin{turn}{80}{}\end{turn} & \begin{turn}{80}{15}\end{turn} & \begin{turn}{80}{-74}\end{turn} & \begin{turn}{80}{210}\end{turn} & \begin{turn}{80}{-421}\end{turn} & \begin{turn}{80}{648}\end{turn} & \begin{turn}{80}{-795}\end{turn} & \begin{turn}{80}{820}\end{turn} & \begin{turn}{80}{-743}\end{turn} & \begin{turn}{80}{584}\end{turn} & \begin{turn}{80}{-422}\end{turn} & \begin{turn}{80}{271}\end{turn} & \begin{turn}{80}{-173}\end{turn} & \begin{turn}{80}{113}\end{turn} & \begin{turn}{80}{-65}\end{turn} & \begin{turn}{80}{41}\end{turn} & \begin{turn}{80}{-25}\end{turn} & \begin{turn}{80}{18}\end{turn} & \begin{turn}{80}{-8}\end{turn} & \begin{turn}{80}{6}\end{turn} & \begin{turn}{80}{-3}\end{turn} & \begin{turn}{80}{3}\end{turn} & \begin{turn}{80}{-1}\end{turn}  \\
\hline
14 & \begin{turn}{80}{-1}\end{turn} & \begin{turn}{80}{12}\end{turn} & \begin{turn}{80}{-83}\end{turn} & \begin{turn}{80}{290}\end{turn} & \begin{turn}{80}{-633}\end{turn} & \begin{turn}{80}{1063}\end{turn} & \begin{turn}{80}{-1449}\end{turn} & \begin{turn}{80}{1629}\end{turn} & \begin{turn}{80}{-1557}\end{turn} & \begin{turn}{80}{1307}\end{turn} & \begin{turn}{80}{-1003}\end{turn} & \begin{turn}{80}{692}\end{turn} & \begin{turn}{80}{-460}\end{turn} & \begin{turn}{80}{286}\end{turn} & \begin{turn}{80}{-173}\end{turn} & \begin{turn}{80}{113}\end{turn} & \begin{turn}{80}{-65}\end{turn} & \begin{turn}{80}{41}\end{turn} & \begin{turn}{80}{-25}\end{turn} & \begin{turn}{80}{18}\end{turn} & \begin{turn}{80}{-8}\end{turn} & \begin{turn}{80}{6}\end{turn} & \begin{turn}{80}{-3}\end{turn}  \\
\hline
15 & \begin{turn}{80}{-1}\end{turn} & \begin{turn}{80}{13}\end{turn} & \begin{turn}{80}{-88}\end{turn} & \begin{turn}{80}{355}\end{turn} & \begin{turn}{80}{-919}\end{turn} & \begin{turn}{80}{1730}\end{turn} & \begin{turn}{80}{-2529}\end{turn} & \begin{turn}{80}{3065}\end{turn} & \begin{turn}{80}{-3190}\end{turn} & \begin{turn}{80}{2882}\end{turn} & \begin{turn}{80}{-2308}\end{turn} & \begin{turn}{80}{1673}\end{turn} & \begin{turn}{80}{-1152}\end{turn} & \begin{turn}{80}{736}\end{turn} & \begin{turn}{80}{-460}\end{turn} & \begin{turn}{80}{286}\end{turn} & \begin{turn}{80}{-173}\end{turn} & \begin{turn}{80}{113}\end{turn} & \begin{turn}{80}{-65}\end{turn} & \begin{turn}{80}{41}\end{turn} & \begin{turn}{80}{-25}\end{turn} & \begin{turn}{80}{18}\end{turn} & \begin{turn}{80}{-8}\end{turn}  \\
\hline
16 & \begin{turn}{80}{-4}\end{turn} & \begin{turn}{80}{14}\end{turn} & \begin{turn}{80}{-97}\end{turn} & \begin{turn}{80}{435}\end{turn} & \begin{turn}{80}{-1256}\end{turn} & \begin{turn}{80}{2628}\end{turn} & \begin{turn}{80}{-4264}\end{turn} & \begin{turn}{80}{5662}\end{turn} & \begin{turn}{80}{-6288}\end{turn} & \begin{turn}{80}{6027}\end{turn} & \begin{turn}{80}{-5164}\end{turn} & \begin{turn}{80}{3991}\end{turn} & \begin{turn}{80}{-2826}\end{turn} & \begin{turn}{80}{1862}\end{turn} & \begin{turn}{80}{-1211}\end{turn} & \begin{turn}{80}{758}\end{turn} & \begin{turn}{80}{-460}\end{turn} & \begin{turn}{80}{286}\end{turn} & \begin{turn}{80}{-173}\end{turn} & \begin{turn}{80}{113}\end{turn} & \begin{turn}{80}{-65}\end{turn} & \begin{turn}{80}{41}\end{turn} & \begin{turn}{80}{-25}\end{turn}  \\
\hline
17 & \begin{turn}{80}{-18}\end{turn} & \begin{turn}{80}{29}\end{turn} & \begin{turn}{80}{-105}\end{turn} & \begin{turn}{80}{510}\end{turn} & \begin{turn}{80}{-1646}\end{turn} & \begin{turn}{80}{3858}\end{turn} & \begin{turn}{80}{-6917}\end{turn} & \begin{turn}{80}{9960}\end{turn} & \begin{turn}{80}{-11959}\end{turn} & \begin{turn}{80}{12375}\end{turn} & \begin{turn}{80}{-11226}\end{turn} & \begin{turn}{80}{9101}\end{turn} & \begin{turn}{80}{-6791}\end{turn} & \begin{turn}{80}{4701}\end{turn} & \begin{turn}{80}{-3082}\end{turn} & \begin{turn}{80}{1930}\end{turn} & \begin{turn}{80}{-1211}\end{turn} & \begin{turn}{80}{758}\end{turn} & \begin{turn}{80}{-460}\end{turn} & \begin{turn}{80}{286}\end{turn} & \begin{turn}{80}{-173}\end{turn} & \begin{turn}{80}{113}\end{turn} & \begin{turn}{80}{-65}\end{turn}  \\
\hline
18 & \begin{turn}{80}{-20}\end{turn} & \begin{turn}{80}{72}\end{turn} & \begin{turn}{80}{-205}\end{turn} & \begin{turn}{80}{600}\end{turn} & \begin{turn}{80}{-2016}\end{turn} & \begin{turn}{80}{5418}\end{turn} & \begin{turn}{80}{-10793}\end{turn} & \begin{turn}{80}{16931}\end{turn} & \begin{turn}{80}{-21970}\end{turn} & \begin{turn}{80}{24380}\end{turn} & \begin{turn}{80}{-23612}\end{turn} & \begin{turn}{80}{20356}\end{turn} & \begin{turn}{80}{-15886}\end{turn} & \begin{turn}{80}{11412}\end{turn} & \begin{turn}{80}{-7761}\end{turn} & \begin{turn}{80}{5026}\end{turn} & \begin{turn}{80}{-3172}\end{turn} & \begin{turn}{80}{1960}\end{turn} & \begin{turn}{80}{-1211}\end{turn} & \begin{turn}{80}{758}\end{turn} & \begin{turn}{80}{-460}\end{turn} & \begin{turn}{80}{286}\end{turn} & \begin{turn}{80}{-173}\end{turn}  \\
\hline
19 & \begin{turn}{80}{-13}\end{turn} & \begin{turn}{80}{155}\end{turn} & \begin{turn}{80}{-348}\end{turn} & \begin{turn}{80}{726}\end{turn} & \begin{turn}{80}{-2645}\end{turn} & \begin{turn}{80}{7467}\end{turn} & \begin{turn}{80}{-15913}\end{turn} & \begin{turn}{80}{27513}\end{turn} & \begin{turn}{80}{-39089}\end{turn} & \begin{turn}{80}{46618}\end{turn} & \begin{turn}{80}{-48051}\end{turn} & \begin{turn}{80}{43930}\end{turn} & \begin{turn}{80}{-36222}\end{turn} & \begin{turn}{80}{27330}\end{turn} & \begin{turn}{80}{-19176}\end{turn} & \begin{turn}{80}{12706}\end{turn} & \begin{turn}{80}{-8187}\end{turn} & \begin{turn}{80}{5131}\end{turn} & \begin{turn}{80}{-3172}\end{turn} & \begin{turn}{80}{1960}\end{turn} & \begin{turn}{80}{-1211}\end{turn} & \begin{turn}{80}{758}\end{turn} & \begin{turn}{80}{-460}\end{turn}  \\
\hline
20 & \begin{turn}{80}{12}\end{turn} & \begin{turn}{80}{164}\end{turn} & \begin{turn}{80}{-461}\end{turn} & \begin{turn}{80}{1415}\end{turn} & \begin{turn}{80}{-3897}\end{turn} & \begin{turn}{80}{9561}\end{turn} & \begin{turn}{80}{-22616}\end{turn} & \begin{turn}{80}{43567}\end{turn} & \begin{turn}{80}{-66750}\end{turn} & \begin{turn}{80}{85597}\end{turn} & \begin{turn}{80}{-94842}\end{turn} & \begin{turn}{80}{92332}\end{turn} & \begin{turn}{80}{-80218}\end{turn} & \begin{turn}{80}{63394}\end{turn} & \begin{turn}{80}{-46401}\end{turn} & \begin{turn}{80}{31900}\end{turn} & \begin{turn}{80}{-20912}\end{turn} & \begin{turn}{80}{13236}\end{turn} & \begin{turn}{80}{-8322}\end{turn} & \begin{turn}{80}{5173}\end{turn} & \begin{turn}{80}{-3172}\end{turn} & \begin{turn}{80}{1960}\end{turn} & \begin{turn}{80}{-1211}\end{turn}  \\
\hline
21 & \begin{turn}{80}{158}\end{turn} & \begin{turn}{80}{-217}\end{turn} & \begin{turn}{80}{-956}\end{turn} & \begin{turn}{80}{3381}\end{turn} & \begin{turn}{80}{-5196}\end{turn} & \begin{turn}{80}{11975}\end{turn} & \begin{turn}{80}{-32592}\end{turn} & \begin{turn}{80}{66403}\end{turn} & \begin{turn}{80}{-109135}\end{turn} & \begin{turn}{80}{152157}\end{turn} & \begin{turn}{80}{-181121}\end{turn} & \begin{turn}{80}{187515}\end{turn} & \begin{turn}{80}{-172652}\end{turn} & \begin{turn}{80}{143782}\end{turn} & \begin{turn}{80}{-109800}\end{turn} & \begin{turn}{80}{78122}\end{turn} & \begin{turn}{80}{-52716}\end{turn} & \begin{turn}{80}{34181}\end{turn} & \begin{turn}{80}{-21600}\end{turn} & \begin{turn}{80}{13392}\end{turn} & \begin{turn}{80}{-8322}\end{turn} & \begin{turn}{80}{5173}\end{turn} & \begin{turn}{80}{-3172}\end{turn}  \\
\hline
22 & \begin{turn}{80}{638}\end{turn} & \begin{turn}{80}{-391}\end{turn} & \begin{turn}{80}{-2652}\end{turn} & \begin{turn}{80}{3573}\end{turn} & \begin{turn}{80}{-5265}\end{turn} & \begin{turn}{80}{20670}\end{turn} & \begin{turn}{80}{-48875}\end{turn} & \begin{turn}{80}{94966}\end{turn} & \begin{turn}{80}{-171890}\end{turn} & \begin{turn}{80}{262183}\end{turn} & \begin{turn}{80}{-334105}\end{turn} & \begin{turn}{80}{368926}\end{turn} & \begin{turn}{80}{-360657}\end{turn} & \begin{turn}{80}{316730}\end{turn} & \begin{turn}{80}{-253621}\end{turn} & \begin{turn}{80}{187909}\end{turn} & \begin{turn}{80}{-130721}\end{turn} & \begin{turn}{80}{86694}\end{turn} & \begin{turn}{80}{-55715}\end{turn} & \begin{turn}{80}{35034}\end{turn} & \begin{turn}{80}{-21798}\end{turn} & \begin{turn}{80}{13448}\end{turn} & \begin{turn}{80}{-8322}\end{turn}  \\
\hline
23 & \begin{turn}{80}{480}\end{turn} & \begin{turn}{80}{1619}\end{turn} & \begin{turn}{80}{1446}\end{turn} & \begin{turn}{80}{-5648}\end{turn} & \begin{turn}{80}{-17669}\end{turn} & \begin{turn}{80}{45575}\end{turn} & \begin{turn}{80}{-61584}\end{turn} & \begin{turn}{80}{130666}\end{turn} & \begin{turn}{80}{-271978}\end{turn} & \begin{turn}{80}{436507}\end{turn} & \begin{turn}{80}{-591664}\end{turn} & \begin{turn}{80}{702289}\end{turn} & \begin{turn}{80}{-731813}\end{turn} & \begin{turn}{80}{678875}\end{turn} & \begin{turn}{80}{-570539}\end{turn} & \begin{turn}{80}{441262}\end{turn} & \begin{turn}{80}{-318509}\end{turn} & \begin{turn}{80}{217419}\end{turn} & \begin{turn}{80}{-142337}\end{turn} & \begin{turn}{80}{90568}\end{turn} & \begin{turn}{80}{-56797}\end{turn} & \begin{turn}{80}{35263}\end{turn} & \begin{turn}{80}{-21798}\end{turn}  \\
\hline
\end{tabular}
\end{center}

\vspace{.4cm}
\caption{Table of Euler characteristics $\chi_{ij}$ by complexity $j$ and Hodge degree $i$ of $H_*(\overline{Emb}_d,\BQ)$ for odd $d$.}
\end{table}

\begin{table}[h!]

{\tiny \begin{center}
\begin{tabular}{|c|c|c|c|c|c|c|c|c|c|c|c|c|c|c|c|c|c|c|c|c|c|c|c|c|}\hline &\multicolumn{23}{|c|}{Hodge degree $i$}&\\
\hline
$j$ & 1 & 2 & 3 & 4 & 5 & 6 & 7 & 8 & 9 & 10 & 11 & 12 & 13 & 14 & 15 & 16 & 17 & 18 & 19 & 20 & 21 & 22 & 23 & total \\
\hline
1 & \begin{turn}{80}{-1}\end{turn} & \begin{turn}{80}{}\end{turn} & \begin{turn}{80}{}\end{turn} & \begin{turn}{80}{}\end{turn} & \begin{turn}{80}{}\end{turn} & \begin{turn}{80}{}\end{turn} & \begin{turn}{80}{}\end{turn} & \begin{turn}{80}{}\end{turn} & \begin{turn}{80}{}\end{turn} & \begin{turn}{80}{}\end{turn} & \begin{turn}{80}{}\end{turn} & \begin{turn}{80}{}\end{turn} & \begin{turn}{80}{}\end{turn} & \begin{turn}{80}{}\end{turn} & \begin{turn}{80}{}\end{turn} & \begin{turn}{80}{}\end{turn} & \begin{turn}{80}{}\end{turn} & \begin{turn}{80}{}\end{turn} & \begin{turn}{80}{}\end{turn} & \begin{turn}{80}{}\end{turn} & \begin{turn}{80}{}\end{turn} & \begin{turn}{80}{}\end{turn} & \begin{turn}{80}{}\end{turn} & 1 \\
\hline
2 & \begin{turn}{80}{}\end{turn} & \begin{turn}{80}{}\end{turn} & \begin{turn}{80}{1}\end{turn} & \begin{turn}{80}{}\end{turn} & \begin{turn}{80}{}\end{turn} & \begin{turn}{80}{}\end{turn} & \begin{turn}{80}{}\end{turn} & \begin{turn}{80}{}\end{turn} & \begin{turn}{80}{}\end{turn} & \begin{turn}{80}{}\end{turn} & \begin{turn}{80}{}\end{turn} & \begin{turn}{80}{}\end{turn} & \begin{turn}{80}{}\end{turn} & \begin{turn}{80}{}\end{turn} & \begin{turn}{80}{}\end{turn} & \begin{turn}{80}{}\end{turn} & \begin{turn}{80}{}\end{turn} & \begin{turn}{80}{}\end{turn} & \begin{turn}{80}{}\end{turn} & \begin{turn}{80}{}\end{turn} & \begin{turn}{80}{}\end{turn} & \begin{turn}{80}{}\end{turn} & \begin{turn}{80}{}\end{turn} & 1 \\
\hline
3 & \begin{turn}{80}{1}\end{turn} & \begin{turn}{80}{-1}\end{turn} & \begin{turn}{80}{}\end{turn} & \begin{turn}{80}{}\end{turn} & \begin{turn}{80}{}\end{turn} & \begin{turn}{80}{}\end{turn} & \begin{turn}{80}{}\end{turn} & \begin{turn}{80}{}\end{turn} & \begin{turn}{80}{}\end{turn} & \begin{turn}{80}{}\end{turn} & \begin{turn}{80}{}\end{turn} & \begin{turn}{80}{}\end{turn} & \begin{turn}{80}{}\end{turn} & \begin{turn}{80}{}\end{turn} & \begin{turn}{80}{}\end{turn} & \begin{turn}{80}{}\end{turn} & \begin{turn}{80}{}\end{turn} & \begin{turn}{80}{}\end{turn} & \begin{turn}{80}{}\end{turn} & \begin{turn}{80}{}\end{turn} & \begin{turn}{80}{}\end{turn} & \begin{turn}{80}{}\end{turn} & \begin{turn}{80}{}\end{turn} & 2 \\
\hline
4 & \begin{turn}{80}{}\end{turn} & \begin{turn}{80}{}\end{turn} & \begin{turn}{80}{}\end{turn} & \begin{turn}{80}{}\end{turn} & \begin{turn}{80}{}\end{turn} & \begin{turn}{80}{}\end{turn} & \begin{turn}{80}{}\end{turn} & \begin{turn}{80}{}\end{turn} & \begin{turn}{80}{}\end{turn} & \begin{turn}{80}{}\end{turn} & \begin{turn}{80}{}\end{turn} & \begin{turn}{80}{}\end{turn} & \begin{turn}{80}{}\end{turn} & \begin{turn}{80}{}\end{turn} & \begin{turn}{80}{}\end{turn} & \begin{turn}{80}{}\end{turn} & \begin{turn}{80}{}\end{turn} & \begin{turn}{80}{}\end{turn} & \begin{turn}{80}{}\end{turn} & \begin{turn}{80}{}\end{turn} & \begin{turn}{80}{}\end{turn} & \begin{turn}{80}{}\end{turn} & \begin{turn}{80}{}\end{turn} & 0 \\
\hline
5 & \begin{turn}{80}{1}\end{turn} & \begin{turn}{80}{}\end{turn} & \begin{turn}{80}{-1}\end{turn} & \begin{turn}{80}{1}\end{turn} & \begin{turn}{80}{-1}\end{turn} & \begin{turn}{80}{}\end{turn} & \begin{turn}{80}{}\end{turn} & \begin{turn}{80}{}\end{turn} & \begin{turn}{80}{}\end{turn} & \begin{turn}{80}{}\end{turn} & \begin{turn}{80}{}\end{turn} & \begin{turn}{80}{}\end{turn} & \begin{turn}{80}{}\end{turn} & \begin{turn}{80}{}\end{turn} & \begin{turn}{80}{}\end{turn} & \begin{turn}{80}{}\end{turn} & \begin{turn}{80}{}\end{turn} & \begin{turn}{80}{}\end{turn} & \begin{turn}{80}{}\end{turn} & \begin{turn}{80}{}\end{turn} & \begin{turn}{80}{}\end{turn} & \begin{turn}{80}{}\end{turn} & \begin{turn}{80}{}\end{turn} & 4 \\
\hline
6 & \begin{turn}{80}{-1}\end{turn} & \begin{turn}{80}{1}\end{turn} & \begin{turn}{80}{}\end{turn} & \begin{turn}{80}{}\end{turn} & \begin{turn}{80}{}\end{turn} & \begin{turn}{80}{}\end{turn} & \begin{turn}{80}{}\end{turn} & \begin{turn}{80}{}\end{turn} & \begin{turn}{80}{}\end{turn} & \begin{turn}{80}{}\end{turn} & \begin{turn}{80}{}\end{turn} & \begin{turn}{80}{}\end{turn} & \begin{turn}{80}{}\end{turn} & \begin{turn}{80}{}\end{turn} & \begin{turn}{80}{}\end{turn} & \begin{turn}{80}{}\end{turn} & \begin{turn}{80}{}\end{turn} & \begin{turn}{80}{}\end{turn} & \begin{turn}{80}{}\end{turn} & \begin{turn}{80}{}\end{turn} & \begin{turn}{80}{}\end{turn} & \begin{turn}{80}{}\end{turn} & \begin{turn}{80}{}\end{turn} & 2 \\
\hline
7 & \begin{turn}{80}{1}\end{turn} & \begin{turn}{80}{}\end{turn} & \begin{turn}{80}{-1}\end{turn} & \begin{turn}{80}{}\end{turn} & \begin{turn}{80}{1}\end{turn} & \begin{turn}{80}{-1}\end{turn} & \begin{turn}{80}{}\end{turn} & \begin{turn}{80}{}\end{turn} & \begin{turn}{80}{}\end{turn} & \begin{turn}{80}{}\end{turn} & \begin{turn}{80}{}\end{turn} & \begin{turn}{80}{}\end{turn} & \begin{turn}{80}{}\end{turn} & \begin{turn}{80}{}\end{turn} & \begin{turn}{80}{}\end{turn} & \begin{turn}{80}{}\end{turn} & \begin{turn}{80}{}\end{turn} & \begin{turn}{80}{}\end{turn} & \begin{turn}{80}{}\end{turn} & \begin{turn}{80}{}\end{turn} & \begin{turn}{80}{}\end{turn} & \begin{turn}{80}{}\end{turn} & \begin{turn}{80}{}\end{turn} & 4 \\
\hline
8 & \begin{turn}{80}{-1}\end{turn} & \begin{turn}{80}{1}\end{turn} & \begin{turn}{80}{1}\end{turn} & \begin{turn}{80}{-2}\end{turn} & \begin{turn}{80}{2}\end{turn} & \begin{turn}{80}{-1}\end{turn} & \begin{turn}{80}{}\end{turn} & \begin{turn}{80}{}\end{turn} & \begin{turn}{80}{}\end{turn} & \begin{turn}{80}{}\end{turn} & \begin{turn}{80}{}\end{turn} & \begin{turn}{80}{}\end{turn} & \begin{turn}{80}{}\end{turn} & \begin{turn}{80}{}\end{turn} & \begin{turn}{80}{}\end{turn} & \begin{turn}{80}{}\end{turn} & \begin{turn}{80}{}\end{turn} & \begin{turn}{80}{}\end{turn} & \begin{turn}{80}{}\end{turn} & \begin{turn}{80}{}\end{turn} & \begin{turn}{80}{}\end{turn} & \begin{turn}{80}{}\end{turn} & \begin{turn}{80}{}\end{turn} & 8 \\
\hline
9 & \begin{turn}{80}{}\end{turn} & \begin{turn}{80}{}\end{turn} & \begin{turn}{80}{2}\end{turn} & \begin{turn}{80}{-1}\end{turn} & \begin{turn}{80}{}\end{turn} & \begin{turn}{80}{}\end{turn} & \begin{turn}{80}{-1}\end{turn} & \begin{turn}{80}{1}\end{turn} & \begin{turn}{80}{-1}\end{turn} & \begin{turn}{80}{}\end{turn} & \begin{turn}{80}{}\end{turn} & \begin{turn}{80}{}\end{turn} & \begin{turn}{80}{}\end{turn} & \begin{turn}{80}{}\end{turn} & \begin{turn}{80}{}\end{turn} & \begin{turn}{80}{}\end{turn} & \begin{turn}{80}{}\end{turn} & \begin{turn}{80}{}\end{turn} & \begin{turn}{80}{}\end{turn} & \begin{turn}{80}{}\end{turn} & \begin{turn}{80}{}\end{turn} & \begin{turn}{80}{}\end{turn} & \begin{turn}{80}{}\end{turn} & 6 \\
\hline
10 & \begin{turn}{80}{-2}\end{turn} & \begin{turn}{80}{}\end{turn} & \begin{turn}{80}{3}\end{turn} & \begin{turn}{80}{-2}\end{turn} & \begin{turn}{80}{}\end{turn} & \begin{turn}{80}{3}\end{turn} & \begin{turn}{80}{-4}\end{turn} & \begin{turn}{80}{2}\end{turn} & \begin{turn}{80}{}\end{turn} & \begin{turn}{80}{}\end{turn} & \begin{turn}{80}{}\end{turn} & \begin{turn}{80}{}\end{turn} & \begin{turn}{80}{}\end{turn} & \begin{turn}{80}{}\end{turn} & \begin{turn}{80}{}\end{turn} & \begin{turn}{80}{}\end{turn} & \begin{turn}{80}{}\end{turn} & \begin{turn}{80}{}\end{turn} & \begin{turn}{80}{}\end{turn} & \begin{turn}{80}{}\end{turn} & \begin{turn}{80}{}\end{turn} & \begin{turn}{80}{}\end{turn} & \begin{turn}{80}{}\end{turn} & 16 \\
\hline
11 & \begin{turn}{80}{2}\end{turn} & \begin{turn}{80}{-3}\end{turn} & \begin{turn}{80}{1}\end{turn} & \begin{turn}{80}{3}\end{turn} & \begin{turn}{80}{-6}\end{turn} & \begin{turn}{80}{5}\end{turn} & \begin{turn}{80}{-1}\end{turn} & \begin{turn}{80}{-1}\end{turn} & \begin{turn}{80}{1}\end{turn} & \begin{turn}{80}{-1}\end{turn} & \begin{turn}{80}{}\end{turn} & \begin{turn}{80}{}\end{turn} & \begin{turn}{80}{}\end{turn} & \begin{turn}{80}{}\end{turn} & \begin{turn}{80}{}\end{turn} & \begin{turn}{80}{}\end{turn} & \begin{turn}{80}{}\end{turn} & \begin{turn}{80}{}\end{turn} & \begin{turn}{80}{}\end{turn} & \begin{turn}{80}{}\end{turn} & \begin{turn}{80}{}\end{turn} & \begin{turn}{80}{}\end{turn} & \begin{turn}{80}{}\end{turn} & 24 \\
\hline
12 & \begin{turn}{80}{}\end{turn} & \begin{turn}{80}{-1}\end{turn} & \begin{turn}{80}{}\end{turn} & \begin{turn}{80}{4}\end{turn} & \begin{turn}{80}{-6}\end{turn} & \begin{turn}{80}{4}\end{turn} & \begin{turn}{80}{2}\end{turn} & \begin{turn}{80}{-6}\end{turn} & \begin{turn}{80}{5}\end{turn} & \begin{turn}{80}{-3}\end{turn} & \begin{turn}{80}{1}\end{turn} & \begin{turn}{80}{}\end{turn} & \begin{turn}{80}{}\end{turn} & \begin{turn}{80}{}\end{turn} & \begin{turn}{80}{}\end{turn} & \begin{turn}{80}{}\end{turn} & \begin{turn}{80}{}\end{turn} & \begin{turn}{80}{}\end{turn} & \begin{turn}{80}{}\end{turn} & \begin{turn}{80}{}\end{turn} & \begin{turn}{80}{}\end{turn} & \begin{turn}{80}{}\end{turn} & \begin{turn}{80}{}\end{turn} & 32 \\
\hline
13 & \begin{turn}{80}{3}\end{turn} & \begin{turn}{80}{-2}\end{turn} & \begin{turn}{80}{-3}\end{turn} & \begin{turn}{80}{12}\end{turn} & \begin{turn}{80}{-10}\end{turn} & \begin{turn}{80}{-6}\end{turn} & \begin{turn}{80}{15}\end{turn} & \begin{turn}{80}{-13}\end{turn} & \begin{turn}{80}{5}\end{turn} & \begin{turn}{80}{1}\end{turn} & \begin{turn}{80}{-2}\end{turn} & \begin{turn}{80}{1}\end{turn} & \begin{turn}{80}{-1}\end{turn} & \begin{turn}{80}{}\end{turn} & \begin{turn}{80}{}\end{turn} & \begin{turn}{80}{}\end{turn} & \begin{turn}{80}{}\end{turn} & \begin{turn}{80}{}\end{turn} & \begin{turn}{80}{}\end{turn} & \begin{turn}{80}{}\end{turn} & \begin{turn}{80}{}\end{turn} & \begin{turn}{80}{}\end{turn} & \begin{turn}{80}{}\end{turn} & 74 \\
\hline
14 & \begin{turn}{80}{}\end{turn} & \begin{turn}{80}{3}\end{turn} & \begin{turn}{80}{-13}\end{turn} & \begin{turn}{80}{8}\end{turn} & \begin{turn}{80}{10}\end{turn} & \begin{turn}{80}{-20}\end{turn} & \begin{turn}{80}{19}\end{turn} & \begin{turn}{80}{-8}\end{turn} & \begin{turn}{80}{-4}\end{turn} & \begin{turn}{80}{10}\end{turn} & \begin{turn}{80}{-8}\end{turn} & \begin{turn}{80}{4}\end{turn} & \begin{turn}{80}{-1}\end{turn} & \begin{turn}{80}{}\end{turn} & \begin{turn}{80}{}\end{turn} & \begin{turn}{80}{}\end{turn} & \begin{turn}{80}{}\end{turn} & \begin{turn}{80}{}\end{turn} & \begin{turn}{80}{}\end{turn} & \begin{turn}{80}{}\end{turn} & \begin{turn}{80}{}\end{turn} & \begin{turn}{80}{}\end{turn} & \begin{turn}{80}{}\end{turn} & 108 \\
\hline
15 & \begin{turn}{80}{}\end{turn} & \begin{turn}{80}{5}\end{turn} & \begin{turn}{80}{-15}\end{turn} & \begin{turn}{80}{5}\end{turn} & \begin{turn}{80}{23}\end{turn} & \begin{turn}{80}{-36}\end{turn} & \begin{turn}{80}{23}\end{turn} & \begin{turn}{80}{11}\end{turn} & \begin{turn}{80}{-32}\end{turn} & \begin{turn}{80}{25}\end{turn} & \begin{turn}{80}{-11}\end{turn} & \begin{turn}{80}{1}\end{turn} & \begin{turn}{80}{3}\end{turn} & \begin{turn}{80}{-2}\end{turn} & \begin{turn}{80}{}\end{turn} & \begin{turn}{80}{}\end{turn} & \begin{turn}{80}{}\end{turn} & \begin{turn}{80}{}\end{turn} & \begin{turn}{80}{}\end{turn} & \begin{turn}{80}{}\end{turn} & \begin{turn}{80}{}\end{turn} & \begin{turn}{80}{}\end{turn} & \begin{turn}{80}{}\end{turn} & 192 \\
\hline
16 & \begin{turn}{80}{-7}\end{turn} & \begin{turn}{80}{18}\end{turn} & \begin{turn}{80}{-9}\end{turn} & \begin{turn}{80}{-30}\end{turn} & \begin{turn}{80}{59}\end{turn} & \begin{turn}{80}{-38}\end{turn} & \begin{turn}{80}{-21}\end{turn} & \begin{turn}{80}{68}\end{turn} & \begin{turn}{80}{-59}\end{turn} & \begin{turn}{80}{17}\end{turn} & \begin{turn}{80}{11}\end{turn} & \begin{turn}{80}{-16}\end{turn} & \begin{turn}{80}{11}\end{turn} & \begin{turn}{80}{-5}\end{turn} & \begin{turn}{80}{1}\end{turn} & \begin{turn}{80}{}\end{turn} & \begin{turn}{80}{}\end{turn} & \begin{turn}{80}{}\end{turn} & \begin{turn}{80}{}\end{turn} & \begin{turn}{80}{}\end{turn} & \begin{turn}{80}{}\end{turn} & \begin{turn}{80}{}\end{turn} & \begin{turn}{80}{}\end{turn} & 370 \\
\hline
17 & \begin{turn}{80}{-14}\end{turn} & \begin{turn}{80}{19}\end{turn} & \begin{turn}{80}{23}\end{turn} & \begin{turn}{80}{-82}\end{turn} & \begin{turn}{80}{64}\end{turn} & \begin{turn}{80}{34}\end{turn} & \begin{turn}{80}{-100}\end{turn} & \begin{turn}{80}{96}\end{turn} & \begin{turn}{80}{-43}\end{turn} & \begin{turn}{80}{-21}\end{turn} & \begin{turn}{80}{56}\end{turn} & \begin{turn}{80}{-49}\end{turn} & \begin{turn}{80}{21}\end{turn} & \begin{turn}{80}{-2}\end{turn} & \begin{turn}{80}{-3}\end{turn} & \begin{turn}{80}{2}\end{turn} & \begin{turn}{80}{-1}\end{turn} & \begin{turn}{80}{}\end{turn} & \begin{turn}{80}{}\end{turn} & \begin{turn}{80}{}\end{turn} & \begin{turn}{80}{}\end{turn} & \begin{turn}{80}{}\end{turn} & \begin{turn}{80}{}\end{turn} & 630 \\
\hline
18 & \begin{turn}{80}{-16}\end{turn} & \begin{turn}{80}{-1}\end{turn} & \begin{turn}{80}{52}\end{turn} & \begin{turn}{80}{-120}\end{turn} & \begin{turn}{80}{78}\end{turn} & \begin{turn}{80}{152}\end{turn} & \begin{turn}{80}{-268}\end{turn} & \begin{turn}{80}{122}\end{turn} & \begin{turn}{80}{85}\end{turn} & \begin{turn}{80}{-168}\end{turn} & \begin{turn}{80}{126}\end{turn} & \begin{turn}{80}{-45}\end{turn} & \begin{turn}{80}{-12}\end{turn} & \begin{turn}{80}{26}\end{turn} & \begin{turn}{80}{-16}\end{turn} & \begin{turn}{80}{6}\end{turn} & \begin{turn}{80}{-1}\end{turn} & \begin{turn}{80}{}\end{turn} & \begin{turn}{80}{}\end{turn} & \begin{turn}{80}{}\end{turn} & \begin{turn}{80}{}\end{turn} & \begin{turn}{80}{}\end{turn} & \begin{turn}{80}{}\end{turn} & 1294 \\
\hline
19 & \begin{turn}{80}{-12}\end{turn} & \begin{turn}{80}{-88}\end{turn} & \begin{turn}{80}{176}\end{turn} & \begin{turn}{80}{-8}\end{turn} & \begin{turn}{80}{-186}\end{turn} & \begin{turn}{80}{290}\end{turn} & \begin{turn}{80}{-265}\end{turn} & \begin{turn}{80}{-30}\end{turn} & \begin{turn}{80}{346}\end{turn} & \begin{turn}{80}{-339}\end{turn} & \begin{turn}{80}{100}\end{turn} & \begin{turn}{80}{76}\end{turn} & \begin{turn}{80}{-108}\end{turn} & \begin{turn}{80}{75}\end{turn} & \begin{turn}{80}{-33}\end{turn} & \begin{turn}{80}{5}\end{turn} & \begin{turn}{80}{3}\end{turn} & \begin{turn}{80}{-2}\end{turn} & \begin{turn}{80}{}\end{turn} & \begin{turn}{80}{}\end{turn} & \begin{turn}{80}{}\end{turn} & \begin{turn}{80}{}\end{turn} & \begin{turn}{80}{}\end{turn} & 2142 \\
\hline
20 & \begin{turn}{80}{7}\end{turn} & \begin{turn}{80}{-167}\end{turn} & \begin{turn}{80}{393}\end{turn} & \begin{turn}{80}{145}\end{turn} & \begin{turn}{80}{-937}\end{turn} & \begin{turn}{80}{558}\end{turn} & \begin{turn}{80}{327}\end{turn} & \begin{turn}{80}{-611}\end{turn} & \begin{turn}{80}{531}\end{turn} & \begin{turn}{80}{-312}\end{turn} & \begin{turn}{80}{-98}\end{turn} & \begin{turn}{80}{363}\end{turn} & \begin{turn}{80}{-294}\end{turn} & \begin{turn}{80}{101}\end{turn} & \begin{turn}{80}{17}\end{turn} & \begin{turn}{80}{-40}\end{turn} & \begin{turn}{80}{23}\end{turn} & \begin{turn}{80}{-7}\end{turn} & \begin{turn}{80}{1}\end{turn} & \begin{turn}{80}{}\end{turn} & \begin{turn}{80}{}\end{turn} & \begin{turn}{80}{}\end{turn} & \begin{turn}{80}{}\end{turn} & 4932 \\
\hline
21 & \begin{turn}{80}{168}\end{turn} & \begin{turn}{80}{-37}\end{turn} & \begin{turn}{80}{13}\end{turn} & \begin{turn}{80}{-108}\end{turn} & \begin{turn}{80}{-1151}\end{turn} & \begin{turn}{80}{1472}\end{turn} & \begin{turn}{80}{1007}\end{turn} & \begin{turn}{80}{-2404}\end{turn} & \begin{turn}{80}{871}\end{turn} & \begin{turn}{80}{718}\end{turn} & \begin{turn}{80}{-984}\end{turn} & \begin{turn}{80}{667}\end{turn} & \begin{turn}{80}{-238}\end{turn} & \begin{turn}{80}{-105}\end{turn} & \begin{turn}{80}{206}\end{turn} & \begin{turn}{80}{-132}\end{turn} & \begin{turn}{80}{45}\end{turn} & \begin{turn}{80}{-6}\end{turn} & \begin{turn}{80}{-3}\end{turn} & \begin{turn}{80}{2}\end{turn} & \begin{turn}{80}{-1}\end{turn} & \begin{turn}{80}{}\end{turn} & \begin{turn}{80}{}\end{turn} & 10338 \\
\hline
22 & \begin{turn}{80}{638}\end{turn} & \begin{turn}{80}{-241}\end{turn} & \begin{turn}{80}{-2676}\end{turn} & \begin{turn}{80}{1806}\end{turn} & \begin{turn}{80}{2506}\end{turn} & \begin{turn}{80}{-1378}\end{turn} & \begin{turn}{80}{-349}\end{turn} & \begin{turn}{80}{-2171}\end{turn} & \begin{turn}{80}{1510}\end{turn} & \begin{turn}{80}{2159}\end{turn} & \begin{turn}{80}{-2672}\end{turn} & \begin{turn}{80}{707}\end{turn} & \begin{turn}{80}{580}\end{turn} & \begin{turn}{80}{-748}\end{turn} & \begin{turn}{80}{492}\end{turn} & \begin{turn}{80}{-187}\end{turn} & \begin{turn}{80}{-6}\end{turn} & \begin{turn}{80}{53}\end{turn} & \begin{turn}{80}{-31}\end{turn} & \begin{turn}{80}{9}\end{turn} & \begin{turn}{80}{-1}\end{turn} & \begin{turn}{80}{}\end{turn} & \begin{turn}{80}{}\end{turn} & 20920 \\
\hline
23 & \begin{turn}{80}{468}\end{turn} & \begin{turn}{80}{-2644}\end{turn} & \begin{turn}{80}{-2607}\end{turn} & \begin{turn}{80}{12686}\end{turn} & \begin{turn}{80}{1016}\end{turn} & \begin{turn}{80}{-18755}\end{turn} & \begin{turn}{80}{5351}\end{turn} & \begin{turn}{80}{8867}\end{turn} & \begin{turn}{80}{-4274}\end{turn} & \begin{turn}{80}{1079}\end{turn} & \begin{turn}{80}{-2353}\end{turn} & \begin{turn}{80}{228}\end{turn} & \begin{turn}{80}{2389}\end{turn} & \begin{turn}{80}{-2042}\end{turn} & \begin{turn}{80}{537}\end{turn} & \begin{turn}{80}{261}\end{turn} & \begin{turn}{80}{-351}\end{turn} & \begin{turn}{80}{201}\end{turn} & \begin{turn}{80}{-66}\end{turn} & \begin{turn}{80}{8}\end{turn} & \begin{turn}{80}{3}\end{turn} & \begin{turn}{80}{-2}\end{turn} & \begin{turn}{80}{}\end{turn} & 66188 \\
\hline
\end{tabular}
\end{center}}
\vspace{.4cm}
\caption{Table of Euler characteristics $\chi_{ij}^\pi$ by complexity $j$ and Hodge degree $i$ of $\pi_*(\overline{Emb}_d)\otimes\BQ$ for even $d$.}
\end{table}

\begin{table}[h!]

\begin{center}
\tiny
\begin{tabular}{|c|c|c|c|c|c|c|c|c|c|c|c|c|c|c|c|c|c|c|c|c|c|c|c|}\hline &\multicolumn{23}{|c|}{Hodge degree $i$}\\
\hline
$j$ & 1 & 2 & 3 & 4 & 5 & 6 & 7 & 8 & 9 & 10 & 11 & 12 & 13 & 14 & 15 & 16 & 17 & 18 & 19 & 20 & 21 & 22 & 23 \\
\hline
1 & \begin{turn}{80}{-1}\end{turn} & \begin{turn}{80}{}\end{turn} & \begin{turn}{80}{}\end{turn} & \begin{turn}{80}{}\end{turn} & \begin{turn}{80}{}\end{turn} & \begin{turn}{80}{}\end{turn} & \begin{turn}{80}{}\end{turn} & \begin{turn}{80}{}\end{turn} & \begin{turn}{80}{}\end{turn} & \begin{turn}{80}{}\end{turn} & \begin{turn}{80}{}\end{turn} & \begin{turn}{80}{}\end{turn} & \begin{turn}{80}{}\end{turn} & \begin{turn}{80}{}\end{turn} & \begin{turn}{80}{}\end{turn} & \begin{turn}{80}{}\end{turn} & \begin{turn}{80}{}\end{turn} & \begin{turn}{80}{}\end{turn} & \begin{turn}{80}{}\end{turn} & \begin{turn}{80}{}\end{turn} & \begin{turn}{80}{}\end{turn} & \begin{turn}{80}{}\end{turn} & \begin{turn}{80}{}\end{turn}  \\
\hline
2 & \begin{turn}{80}{}\end{turn} & \begin{turn}{80}{}\end{turn} & \begin{turn}{80}{1}\end{turn} & \begin{turn}{80}{}\end{turn} & \begin{turn}{80}{}\end{turn} & \begin{turn}{80}{}\end{turn} & \begin{turn}{80}{}\end{turn} & \begin{turn}{80}{}\end{turn} & \begin{turn}{80}{}\end{turn} & \begin{turn}{80}{}\end{turn} & \begin{turn}{80}{}\end{turn} & \begin{turn}{80}{}\end{turn} & \begin{turn}{80}{}\end{turn} & \begin{turn}{80}{}\end{turn} & \begin{turn}{80}{}\end{turn} & \begin{turn}{80}{}\end{turn} & \begin{turn}{80}{}\end{turn} & \begin{turn}{80}{}\end{turn} & \begin{turn}{80}{}\end{turn} & \begin{turn}{80}{}\end{turn} & \begin{turn}{80}{}\end{turn} & \begin{turn}{80}{}\end{turn} & \begin{turn}{80}{}\end{turn}  \\
\hline
3 & \begin{turn}{80}{1}\end{turn} & \begin{turn}{80}{-1}\end{turn} & \begin{turn}{80}{}\end{turn} & \begin{turn}{80}{-1}\end{turn} & \begin{turn}{80}{}\end{turn} & \begin{turn}{80}{}\end{turn} & \begin{turn}{80}{}\end{turn} & \begin{turn}{80}{}\end{turn} & \begin{turn}{80}{}\end{turn} & \begin{turn}{80}{}\end{turn} & \begin{turn}{80}{}\end{turn} & \begin{turn}{80}{}\end{turn} & \begin{turn}{80}{}\end{turn} & \begin{turn}{80}{}\end{turn} & \begin{turn}{80}{}\end{turn} & \begin{turn}{80}{}\end{turn} & \begin{turn}{80}{}\end{turn} & \begin{turn}{80}{}\end{turn} & \begin{turn}{80}{}\end{turn} & \begin{turn}{80}{}\end{turn} & \begin{turn}{80}{}\end{turn} & \begin{turn}{80}{}\end{turn} & \begin{turn}{80}{}\end{turn}  \\
\hline
4 & \begin{turn}{80}{}\end{turn} & \begin{turn}{80}{-1}\end{turn} & \begin{turn}{80}{1}\end{turn} & \begin{turn}{80}{}\end{turn} & \begin{turn}{80}{}\end{turn} & \begin{turn}{80}{1}\end{turn} & \begin{turn}{80}{}\end{turn} & \begin{turn}{80}{}\end{turn} & \begin{turn}{80}{}\end{turn} & \begin{turn}{80}{}\end{turn} & \begin{turn}{80}{}\end{turn} & \begin{turn}{80}{}\end{turn} & \begin{turn}{80}{}\end{turn} & \begin{turn}{80}{}\end{turn} & \begin{turn}{80}{}\end{turn} & \begin{turn}{80}{}\end{turn} & \begin{turn}{80}{}\end{turn} & \begin{turn}{80}{}\end{turn} & \begin{turn}{80}{}\end{turn} & \begin{turn}{80}{}\end{turn} & \begin{turn}{80}{}\end{turn} & \begin{turn}{80}{}\end{turn} & \begin{turn}{80}{}\end{turn}  \\
\hline
5 & \begin{turn}{80}{1}\end{turn} & \begin{turn}{80}{}\end{turn} & \begin{turn}{80}{-1}\end{turn} & \begin{turn}{80}{2}\end{turn} & \begin{turn}{80}{-2}\end{turn} & \begin{turn}{80}{}\end{turn} & \begin{turn}{80}{-1}\end{turn} & \begin{turn}{80}{}\end{turn} & \begin{turn}{80}{}\end{turn} & \begin{turn}{80}{}\end{turn} & \begin{turn}{80}{}\end{turn} & \begin{turn}{80}{}\end{turn} & \begin{turn}{80}{}\end{turn} & \begin{turn}{80}{}\end{turn} & \begin{turn}{80}{}\end{turn} & \begin{turn}{80}{}\end{turn} & \begin{turn}{80}{}\end{turn} & \begin{turn}{80}{}\end{turn} & \begin{turn}{80}{}\end{turn} & \begin{turn}{80}{}\end{turn} & \begin{turn}{80}{}\end{turn} & \begin{turn}{80}{}\end{turn} & \begin{turn}{80}{}\end{turn}  \\
\hline
6 & \begin{turn}{80}{-1}\end{turn} & \begin{turn}{80}{1}\end{turn} & \begin{turn}{80}{-1}\end{turn} & \begin{turn}{80}{1}\end{turn} & \begin{turn}{80}{-2}\end{turn} & \begin{turn}{80}{2}\end{turn} & \begin{turn}{80}{}\end{turn} & \begin{turn}{80}{}\end{turn} & \begin{turn}{80}{1}\end{turn} & \begin{turn}{80}{}\end{turn} & \begin{turn}{80}{}\end{turn} & \begin{turn}{80}{}\end{turn} & \begin{turn}{80}{}\end{turn} & \begin{turn}{80}{}\end{turn} & \begin{turn}{80}{}\end{turn} & \begin{turn}{80}{}\end{turn} & \begin{turn}{80}{}\end{turn} & \begin{turn}{80}{}\end{turn} & \begin{turn}{80}{}\end{turn} & \begin{turn}{80}{}\end{turn} & \begin{turn}{80}{}\end{turn} & \begin{turn}{80}{}\end{turn} & \begin{turn}{80}{}\end{turn}  \\
\hline
7 & \begin{turn}{80}{1}\end{turn} & \begin{turn}{80}{1}\end{turn} & \begin{turn}{80}{-3}\end{turn} & \begin{turn}{80}{2}\end{turn} & \begin{turn}{80}{1}\end{turn} & \begin{turn}{80}{-2}\end{turn} & \begin{turn}{80}{2}\end{turn} & \begin{turn}{80}{-2}\end{turn} & \begin{turn}{80}{}\end{turn} & \begin{turn}{80}{-1}\end{turn} & \begin{turn}{80}{}\end{turn} & \begin{turn}{80}{}\end{turn} & \begin{turn}{80}{}\end{turn} & \begin{turn}{80}{}\end{turn} & \begin{turn}{80}{}\end{turn} & \begin{turn}{80}{}\end{turn} & \begin{turn}{80}{}\end{turn} & \begin{turn}{80}{}\end{turn} & \begin{turn}{80}{}\end{turn} & \begin{turn}{80}{}\end{turn} & \begin{turn}{80}{}\end{turn} & \begin{turn}{80}{}\end{turn} & \begin{turn}{80}{}\end{turn}  \\
\hline
8 & \begin{turn}{80}{-1}\end{turn} & \begin{turn}{80}{1}\end{turn} & \begin{turn}{80}{}\end{turn} & \begin{turn}{80}{-3}\end{turn} & \begin{turn}{80}{5}\end{turn} & \begin{turn}{80}{-5}\end{turn} & \begin{turn}{80}{3}\end{turn} & \begin{turn}{80}{-2}\end{turn} & \begin{turn}{80}{2}\end{turn} & \begin{turn}{80}{}\end{turn} & \begin{turn}{80}{}\end{turn} & \begin{turn}{80}{1}\end{turn} & \begin{turn}{80}{}\end{turn} & \begin{turn}{80}{}\end{turn} & \begin{turn}{80}{}\end{turn} & \begin{turn}{80}{}\end{turn} & \begin{turn}{80}{}\end{turn} & \begin{turn}{80}{}\end{turn} & \begin{turn}{80}{}\end{turn} & \begin{turn}{80}{}\end{turn} & \begin{turn}{80}{}\end{turn} & \begin{turn}{80}{}\end{turn} & \begin{turn}{80}{}\end{turn}  \\
\hline
9 & \begin{turn}{80}{}\end{turn} & \begin{turn}{80}{}\end{turn} & \begin{turn}{80}{3}\end{turn} & \begin{turn}{80}{-2}\end{turn} & \begin{turn}{80}{4}\end{turn} & \begin{turn}{80}{-7}\end{turn} & \begin{turn}{80}{4}\end{turn} & \begin{turn}{80}{1}\end{turn} & \begin{turn}{80}{-3}\end{turn} & \begin{turn}{80}{2}\end{turn} & \begin{turn}{80}{-2}\end{turn} & \begin{turn}{80}{}\end{turn} & \begin{turn}{80}{-1}\end{turn} & \begin{turn}{80}{}\end{turn} & \begin{turn}{80}{}\end{turn} & \begin{turn}{80}{}\end{turn} & \begin{turn}{80}{}\end{turn} & \begin{turn}{80}{}\end{turn} & \begin{turn}{80}{}\end{turn} & \begin{turn}{80}{}\end{turn} & \begin{turn}{80}{}\end{turn} & \begin{turn}{80}{}\end{turn} & \begin{turn}{80}{}\end{turn}  \\
\hline
10 & \begin{turn}{80}{-2}\end{turn} & \begin{turn}{80}{2}\end{turn} & \begin{turn}{80}{3}\end{turn} & \begin{turn}{80}{-10}\end{turn} & \begin{turn}{80}{6}\end{turn} & \begin{turn}{80}{3}\end{turn} & \begin{turn}{80}{-10}\end{turn} & \begin{turn}{80}{11}\end{turn} & \begin{turn}{80}{-7}\end{turn} & \begin{turn}{80}{4}\end{turn} & \begin{turn}{80}{-2}\end{turn} & \begin{turn}{80}{2}\end{turn} & \begin{turn}{80}{}\end{turn} & \begin{turn}{80}{}\end{turn} & \begin{turn}{80}{1}\end{turn} & \begin{turn}{80}{}\end{turn} & \begin{turn}{80}{}\end{turn} & \begin{turn}{80}{}\end{turn} & \begin{turn}{80}{}\end{turn} & \begin{turn}{80}{}\end{turn} & \begin{turn}{80}{}\end{turn} & \begin{turn}{80}{}\end{turn} & \begin{turn}{80}{}\end{turn}  \\
\hline
11 & \begin{turn}{80}{2}\end{turn} & \begin{turn}{80}{-3}\end{turn} & \begin{turn}{80}{3}\end{turn} & \begin{turn}{80}{1}\end{turn} & \begin{turn}{80}{-8}\end{turn} & \begin{turn}{80}{14}\end{turn} & \begin{turn}{80}{-12}\end{turn} & \begin{turn}{80}{12}\end{turn} & \begin{turn}{80}{-11}\end{turn} & \begin{turn}{80}{4}\end{turn} & \begin{turn}{80}{1}\end{turn} & \begin{turn}{80}{-3}\end{turn} & \begin{turn}{80}{2}\end{turn} & \begin{turn}{80}{-2}\end{turn} & \begin{turn}{80}{}\end{turn} & \begin{turn}{80}{-1}\end{turn} & \begin{turn}{80}{}\end{turn} & \begin{turn}{80}{}\end{turn} & \begin{turn}{80}{}\end{turn} & \begin{turn}{80}{}\end{turn} & \begin{turn}{80}{}\end{turn} & \begin{turn}{80}{}\end{turn} & \begin{turn}{80}{}\end{turn}  \\
\hline
12 & \begin{turn}{80}{}\end{turn} & \begin{turn}{80}{-2}\end{turn} & \begin{turn}{80}{3}\end{turn} & \begin{turn}{80}{1}\end{turn} & \begin{turn}{80}{-11}\end{turn} & \begin{turn}{80}{21}\end{turn} & \begin{turn}{80}{-23}\end{turn} & \begin{turn}{80}{6}\end{turn} & \begin{turn}{80}{11}\end{turn} & \begin{turn}{80}{-18}\end{turn} & \begin{turn}{80}{15}\end{turn} & \begin{turn}{80}{-7}\end{turn} & \begin{turn}{80}{4}\end{turn} & \begin{turn}{80}{-2}\end{turn} & \begin{turn}{80}{2}\end{turn} & \begin{turn}{80}{}\end{turn} & \begin{turn}{80}{}\end{turn} & \begin{turn}{80}{1}\end{turn} & \begin{turn}{80}{}\end{turn} & \begin{turn}{80}{}\end{turn} & \begin{turn}{80}{}\end{turn} & \begin{turn}{80}{}\end{turn} & \begin{turn}{80}{}\end{turn}  \\
\hline
13 & \begin{turn}{80}{3}\end{turn} & \begin{turn}{80}{-6}\end{turn} & \begin{turn}{80}{3}\end{turn} & \begin{turn}{80}{20}\end{turn} & \begin{turn}{80}{-33}\end{turn} & \begin{turn}{80}{14}\end{turn} & \begin{turn}{80}{15}\end{turn} & \begin{turn}{80}{-36}\end{turn} & \begin{turn}{80}{41}\end{turn} & \begin{turn}{80}{-28}\end{turn} & \begin{turn}{80}{19}\end{turn} & \begin{turn}{80}{-13}\end{turn} & \begin{turn}{80}{3}\end{turn} & \begin{turn}{80}{1}\end{turn} & \begin{turn}{80}{-3}\end{turn} & \begin{turn}{80}{2}\end{turn} & \begin{turn}{80}{-2}\end{turn} & \begin{turn}{80}{}\end{turn} & \begin{turn}{80}{-1}\end{turn} & \begin{turn}{80}{}\end{turn} & \begin{turn}{80}{}\end{turn} & \begin{turn}{80}{}\end{turn} & \begin{turn}{80}{}\end{turn}  \\
\hline
14 & \begin{turn}{80}{}\end{turn} & \begin{turn}{80}{4}\end{turn} & \begin{turn}{80}{-16}\end{turn} & \begin{turn}{80}{15}\end{turn} & \begin{turn}{80}{-5}\end{turn} & \begin{turn}{80}{-16}\end{turn} & \begin{turn}{80}{45}\end{turn} & \begin{turn}{80}{-59}\end{turn} & \begin{turn}{80}{51}\end{turn} & \begin{turn}{80}{-30}\end{turn} & \begin{turn}{80}{}\end{turn} & \begin{turn}{80}{19}\end{turn} & \begin{turn}{80}{-22}\end{turn} & \begin{turn}{80}{17}\end{turn} & \begin{turn}{80}{-7}\end{turn} & \begin{turn}{80}{4}\end{turn} & \begin{turn}{80}{-2}\end{turn} & \begin{turn}{80}{2}\end{turn} & \begin{turn}{80}{}\end{turn} & \begin{turn}{80}{}\end{turn} & \begin{turn}{80}{1}\end{turn} & \begin{turn}{80}{}\end{turn} & \begin{turn}{80}{}\end{turn}  \\
\hline
15 & \begin{turn}{80}{}\end{turn} & \begin{turn}{80}{2}\end{turn} & \begin{turn}{80}{-20}\end{turn} & \begin{turn}{80}{33}\end{turn} & \begin{turn}{80}{3}\end{turn} & \begin{turn}{80}{-60}\end{turn} & \begin{turn}{80}{105}\end{turn} & \begin{turn}{80}{-83}\end{turn} & \begin{turn}{80}{2}\end{turn} & \begin{turn}{80}{65}\end{turn} & \begin{turn}{80}{-88}\end{turn} & \begin{turn}{80}{72}\end{turn} & \begin{turn}{80}{-38}\end{turn} & \begin{turn}{80}{20}\end{turn} & \begin{turn}{80}{-14}\end{turn} & \begin{turn}{80}{3}\end{turn} & \begin{turn}{80}{1}\end{turn} & \begin{turn}{80}{-3}\end{turn} & \begin{turn}{80}{2}\end{turn} & \begin{turn}{80}{-2}\end{turn} & \begin{turn}{80}{}\end{turn} & \begin{turn}{80}{-1}\end{turn} & \begin{turn}{80}{}\end{turn}  \\
\hline
16 & \begin{turn}{80}{-7}\end{turn} & \begin{turn}{80}{25}\end{turn} & \begin{turn}{80}{-27}\end{turn} & \begin{turn}{80}{-11}\end{turn} & \begin{turn}{80}{88}\end{turn} & \begin{turn}{80}{-139}\end{turn} & \begin{turn}{80}{67}\end{turn} & \begin{turn}{80}{60}\end{turn} & \begin{turn}{80}{-147}\end{turn} & \begin{turn}{80}{165}\end{turn} & \begin{turn}{80}{-127}\end{turn} & \begin{turn}{80}{75}\end{turn} & \begin{turn}{80}{-30}\end{turn} & \begin{turn}{80}{-9}\end{turn} & \begin{turn}{80}{25}\end{turn} & \begin{turn}{80}{-22}\end{turn} & \begin{turn}{80}{17}\end{turn} & \begin{turn}{80}{-7}\end{turn} & \begin{turn}{80}{4}\end{turn} & \begin{turn}{80}{-2}\end{turn} & \begin{turn}{80}{2}\end{turn} & \begin{turn}{80}{}\end{turn} & \begin{turn}{80}{}\end{turn}  \\
\hline
17 & \begin{turn}{80}{-14}\end{turn} & \begin{turn}{80}{22}\end{turn} & \begin{turn}{80}{10}\end{turn} & \begin{turn}{80}{-86}\end{turn} & \begin{turn}{80}{126}\end{turn} & \begin{turn}{80}{-74}\end{turn} & \begin{turn}{80}{-33}\end{turn} & \begin{turn}{80}{190}\end{turn} & \begin{turn}{80}{-295}\end{turn} & \begin{turn}{80}{257}\end{turn} & \begin{turn}{80}{-91}\end{turn} & \begin{turn}{80}{-79}\end{turn} & \begin{turn}{80}{148}\end{turn} & \begin{turn}{80}{-135}\end{turn} & \begin{turn}{80}{90}\end{turn} & \begin{turn}{80}{-42}\end{turn} & \begin{turn}{80}{19}\end{turn} & \begin{turn}{80}{-14}\end{turn} & \begin{turn}{80}{3}\end{turn} & \begin{turn}{80}{1}\end{turn} & \begin{turn}{80}{-3}\end{turn} & \begin{turn}{80}{2}\end{turn} & \begin{turn}{80}{-2}\end{turn}  \\
\hline
18 & \begin{turn}{80}{-16}\end{turn} & \begin{turn}{80}{20}\end{turn} & \begin{turn}{80}{31}\end{turn} & \begin{turn}{80}{-188}\end{turn} & \begin{turn}{80}{266}\end{turn} & \begin{turn}{80}{29}\end{turn} & \begin{turn}{80}{-425}\end{turn} & \begin{turn}{80}{526}\end{turn} & \begin{turn}{80}{-326}\end{turn} & \begin{turn}{80}{-42}\end{turn} & \begin{turn}{80}{332}\end{turn} & \begin{turn}{80}{-413}\end{turn} & \begin{turn}{80}{330}\end{turn} & \begin{turn}{80}{-188}\end{turn} & \begin{turn}{80}{80}\end{turn} & \begin{turn}{80}{-22}\end{turn} & \begin{turn}{80}{-13}\end{turn} & \begin{turn}{80}{27}\end{turn} & \begin{turn}{80}{-22}\end{turn} & \begin{turn}{80}{17}\end{turn} & \begin{turn}{80}{-7}\end{turn} & \begin{turn}{80}{4}\end{turn} & \begin{turn}{80}{-2}\end{turn}  \\
\hline
19 & \begin{turn}{80}{-12}\end{turn} & \begin{turn}{80}{-84}\end{turn} & \begin{turn}{80}{216}\end{turn} & \begin{turn}{80}{-136}\end{turn} & \begin{turn}{80}{-38}\end{turn} & \begin{turn}{80}{368}\end{turn} & \begin{turn}{80}{-761}\end{turn} & \begin{turn}{80}{556}\end{turn} & \begin{turn}{80}{184}\end{turn} & \begin{turn}{80}{-721}\end{turn} & \begin{turn}{80}{814}\end{turn} & \begin{turn}{80}{-622}\end{turn} & \begin{turn}{80}{318}\end{turn} & \begin{turn}{80}{-4}\end{turn} & \begin{turn}{80}{-193}\end{turn} & \begin{turn}{80}{219}\end{turn} & \begin{turn}{80}{-161}\end{turn} & \begin{turn}{80}{94}\end{turn} & \begin{turn}{80}{-43}\end{turn} & \begin{turn}{80}{19}\end{turn} & \begin{turn}{80}{-14}\end{turn} & \begin{turn}{80}{3}\end{turn} & \begin{turn}{80}{1}\end{turn}  \\
\hline
20 & \begin{turn}{80}{7}\end{turn} & \begin{turn}{80}{-165}\end{turn} & \begin{turn}{80}{519}\end{turn} & \begin{turn}{80}{-77}\end{turn} & \begin{turn}{80}{-991}\end{turn} & \begin{turn}{80}{1031}\end{turn} & \begin{turn}{80}{-444}\end{turn} & \begin{turn}{80}{-27}\end{turn} & \begin{turn}{80}{951}\end{turn} & \begin{turn}{80}{-1739}\end{turn} & \begin{turn}{80}{1332}\end{turn} & \begin{turn}{80}{-232}\end{turn} & \begin{turn}{80}{-616}\end{turn} & \begin{turn}{80}{871}\end{turn} & \begin{turn}{80}{-725}\end{turn} & \begin{turn}{80}{452}\end{turn} & \begin{turn}{80}{-210}\end{turn} & \begin{turn}{80}{73}\end{turn} & \begin{turn}{80}{-16}\end{turn} & \begin{turn}{80}{-13}\end{turn} & \begin{turn}{80}{27}\end{turn} & \begin{turn}{80}{-22}\end{turn} & \begin{turn}{80}{17}\end{turn}  \\
\hline
21 & \begin{turn}{80}{168}\end{turn} & \begin{turn}{80}{-74}\end{turn} & \begin{turn}{80}{245}\end{turn} & \begin{turn}{80}{-495}\end{turn} & \begin{turn}{80}{-1705}\end{turn} & \begin{turn}{80}{3209}\end{turn} & \begin{turn}{80}{117}\end{turn} & \begin{turn}{80}{-3475}\end{turn} & \begin{turn}{80}{3050}\end{turn} & \begin{turn}{80}{-1416}\end{turn} & \begin{turn}{80}{-96}\end{turn} & \begin{turn}{80}{1624}\end{turn} & \begin{turn}{80}{-2223}\end{turn} & \begin{turn}{80}{1691}\end{turn} & \begin{turn}{80}{-832}\end{turn} & \begin{turn}{80}{214}\end{turn} & \begin{turn}{80}{141}\end{turn} & \begin{turn}{80}{-285}\end{turn} & \begin{turn}{80}{255}\end{turn} & \begin{turn}{80}{-169}\end{turn} & \begin{turn}{80}{93}\end{turn} & \begin{turn}{80}{-43}\end{turn} & \begin{turn}{80}{19}\end{turn}  \\
\hline
22 & \begin{turn}{80}{638}\end{turn} & \begin{turn}{80}{-425}\end{turn} & \begin{turn}{80}{-2710}\end{turn} & \begin{turn}{80}{2145}\end{turn} & \begin{turn}{80}{2011}\end{turn} & \begin{turn}{80}{404}\end{turn} & \begin{turn}{80}{-1311}\end{turn} & \begin{turn}{80}{-5578}\end{turn} & \begin{turn}{80}{6395}\end{turn} & \begin{turn}{80}{632}\end{turn} & \begin{turn}{80}{-4769}\end{turn} & \begin{turn}{80}{4616}\end{turn} & \begin{turn}{80}{-3318}\end{turn} & \begin{turn}{80}{1372}\end{turn} & \begin{turn}{80}{602}\end{turn} & \begin{turn}{80}{-1566}\end{turn} & \begin{turn}{80}{1455}\end{turn} & \begin{turn}{80}{-944}\end{turn} & \begin{turn}{80}{498}\end{turn} & \begin{turn}{80}{-209}\end{turn} & \begin{turn}{80}{69}\end{turn} & \begin{turn}{80}{-13}\end{turn} & \begin{turn}{80}{-13}\end{turn}  \\
\hline
23 & \begin{turn}{80}{468}\end{turn} & \begin{turn}{80}{-3290}\end{turn} & \begin{turn}{80}{-2544}\end{turn} & \begin{turn}{80}{16300}\end{turn} & \begin{turn}{80}{-1620}\end{turn} & \begin{turn}{80}{-21938}\end{turn} & \begin{turn}{80}{8761}\end{turn} & \begin{turn}{80}{5388}\end{turn} & \begin{turn}{80}{1424}\end{turn} & \begin{turn}{80}{1676}\end{turn} & \begin{turn}{80}{-11716}\end{turn} & \begin{turn}{80}{9157}\end{turn} & \begin{turn}{80}{-435}\end{turn} & \begin{turn}{80}{-4214}\end{turn} & \begin{turn}{80}{4986}\end{turn} & \begin{turn}{80}{-3974}\end{turn} & \begin{turn}{80}{2252}\end{turn} & \begin{turn}{80}{-794}\end{turn} & \begin{turn}{80}{35}\end{turn} & \begin{turn}{80}{259}\end{turn} & \begin{turn}{80}{-329}\end{turn} & \begin{turn}{80}{263}\end{turn} & \begin{turn}{80}{-171}\end{turn}  \\
\hline
\end{tabular}
\end{center}

\vspace{.4cm}
\caption{Table of Euler characteristics $\chi_{ij}$ by complexity $j$ and Hodge degree $i$ of $H_*(\overline{Emb}_d,\BQ)$ for even $d$.}
\end{table}


\end{document}